\documentclass[11pt]{amsart}
\usepackage[latin1]{inputenc}
\usepackage{amsfonts}
\usepackage[toc,page,title,titletoc,header]{appendix}
\usepackage{graphicx,psfrag,epsfig,multirow,caption,subcaption}
\usepackage{amssymb,amsmath,amscd,amsthm,amssymb,verbatim,setspace, mathtools}
\usepackage{dsfont} 
\numberwithin{equation}{section}
\usepackage{mathrsfs,wrapfig}
\usepackage{indentfirst}
\usepackage{extarrows}
\usepackage{float}
\usepackage{xcolor}
\usepackage{enumerate}
\usepackage{enumitem} 
\usepackage{romannum}
\usepackage{marginnote}

\usepackage[colorlinks=true, backref=page]{hyperref}

\newtheorem{theorem}{Theorem}[section]
\newtheorem{definition}[theorem]{Definition}
\newtheorem{conjecture}[theorem]{Conjecture}
\newtheorem{proposition}[theorem]{Proposition}
\newtheorem{lemma}[theorem]{Lemma}
\newtheorem{remark}[theorem]{Remark}
\newtheorem{corollary}[theorem]{Corollary}

\newtheorem{Thm}[theorem]{Theorem}
\newtheorem*{Claim}{Claim}

\newtheorem{Lem}[theorem]{Lemma}
\newtheorem{Prop}[theorem]{Proposition}
\newtheorem{Cor}[theorem]{Corollary}
\newtheorem{Rem}[theorem]{Remark}

\newtheorem*{remark*}{Remark}

\numberwithin{equation}{section}

\newcommand{\mfk}{\mathfrak}

\newcommand{\eps}{\varepsilon}

\newcommand{\al}{\alpha}

\newcommand{\bM}{\mathbf{M}}
\newcommand{\bS}{\mathbb{S}}

\newcommand{\BB}{\mathbb{B}}

\newcommand{\RR}{\mathbb{R}}
\newcommand{\SSp}{\mathbb{S}}

\newcommand{\cA}{\mathcal{A}}

\newcommand{\cC}{\mathcal{C}}

\newcommand{\cE}{\mathcal{E}}
\newcommand{\cF}{\mathcal{F}}

\newcommand{\cH}{\mathcal{H}}
\newcommand{\cI}{\mathcal{I}}

\newcommand{\cL}{\mathcal{L}}
\newcommand{\cM}{\mathcal{M}}
\newcommand{\cN}{\mathcal{N}}

\newcommand{\cP}{\mathcal{P}}
\newcommand{\cQ}{\mathcal{Q}}

\newcommand{\cS}{\mathcal{S}}
\newcommand{\cT}{\mathcal{T}}

\newcommand{\cZ}{\mathcal{Z}}

\newcommand{\scU}{\mathscr{U}}
\newcommand{\scV}{\mathscr{V}}

\newcommand{\mbfA}{\mathbf{A}}
\newcommand{\mbfB}{\mathbf{B}}
\newcommand{\mbfC}{\mathbf{C}}

\newcommand{\mbfL}{\mathbf{L}}
\newcommand{\mbfM}{\mathbf{M}}

\newcommand{\mbfQ}{\mathbf{Q}}

\newcommand{\mbfU}{\mathbf{U}}

\newcommand{\rmB}{\mathrm{B}}

\newcommand{\rmD}{\mathrm{D}}

\newcommand{\rmG}{\mathrm{G}}

\newcommand{\rmO}{\mathrm{O}}

\newcommand{\rmR}{\mathrm{R}}

\newcommand{\rmT}{\mathrm{T}}

\newcommand{\rmW}{\mathrm{W}}

\newcommand{\bL}{\mathbf{L}}
\newcommand{\bQ}{\mathbf{Q}}
\newcommand{\mbf}{\mathbf{f}}
\newcommand{\bc}{\mathbf{c}}
\newcommand{\bl}{\mathbf{\ell}}
\newcommand{\bq}{\mathbf{q}}
\newcommand{\bu}{\mathbf{u}}

\newcommand{\bx}{\mathbf{x}}
\newcommand{\by}{\mathbf{y}}

\newcommand{\sN}{\mathscr{N}}

\newcommand{\R}{\mathbb{R}}

\newcommand{\pr}{\partial}

\newcommand{\mbfd}{\mathbf{d}}
\newcommand{\mbfx}{\mathbf{x}}

\newcommand{\graph}{\mathrm{graph}}

\newcommand{\orig}{\mathbf{0}}

\newcommand{\spine}{\operatorname{spine}}
\newcommand{\odist}{\overline{\dist}}
\newcommand{\Jet}{\operatorname{Jet}}

\newcommand{\Rot}{\operatorname{Rot}}
\newcommand{\sgn}{\operatorname{sgn}}

\newcommand{\Span}{\operatorname{span}}
\newcommand{\gfrd}{\mathbf{r}}
\newcommand{\dg}{\mathrm{dg}}
\newcommand{\ndg}{\mathrm{ndg}}
\newcommand{\Id}{\ensuremath{\mathds{1}}}

\DeclareMathOperator{\area}{Area}

\DeclareMathOperator{\spt}{spt}

\DeclareMathOperator{\dist}{dist}

\DeclareMathOperator{\loc}{loc}
\DeclareMathOperator{\tr}{tr}

\DeclareMathOperator{\id}{\operatorname{Id}}

\DeclareMathOperator{\Tr}{\operatorname{tr}}

\title{Regularity of cylindrical singular sets of mean curvature flow}
\author{Ao Sun}
\address{Lehigh University, Department of Mathematics, Chandler-Ullmann Hall, Bethlehem, PA 18015}
\email{aos223@lehigh.edu}

\author{Zhihan Wang}
\address{Cornell University, Department of Mathematics, 310 Malott Hall, Ithaca, NY 14853}
\email{zw782@cornell.edu}

\author{Jinxin Xue}
\address{New Cornerstone Science Laboratory, Department of Mathematics, Rm A115, Tsinghua University, Haidian District, Beijing, 100084}
\email{jxue@tsinghua.edu.cn}

\date{}

\begin{document}
	\pagenumbering{arabic}
	
	\begin{abstract}
		In this paper, we study the $k$-cylindrical singular set of mean curvature flow in $\mathbb R^{n+1}$ for each $1\leq k\leq n-1$. We prove that they are locally contained in a $k$-dimensional $C^{2,\alpha}$-submanifold after removing some lower-dimensional parts. Moreover, if the $k$-cylindrical singular set is a $k$-submanifold, then its curvature is determined by the asymptotic profile of the flow at these singularities. As a byproduct, we provide a detailed asymptotic profile and graphical radius estimate at these singularities. The proof is based on a new $L^2$-distance non-concentration property that we introduced in \cite{SunWangXue1_Passing}, modified into a relative version that allows us to modulo those low eigenmodes that are not decaying fast enough and do not contribute to the curvature of the singular set.
	\end{abstract}
	
	\maketitle
	
	\setcounter{tocdepth}{1}
	\tableofcontents

	\section{Introduction} \label{Sec_Intro}
	This is the second in a series of papers to study the connections between geometry, topology, and dynamics of cylindrical singularities of mean curvature flow (MCF). In \cite{SunXue2022_generic_cylindrical, SunWangXue1_Passing}, we observed that the dynamics of the cylindrical singularities influence the shape of mean curvature flows near the singularities, and moreover, we introduced the notion of {\it nondegenerate cylindrical singularities}, proved that they are isolated in spacetime, and provided a relatively complete picture about the geometry and topology change of the mean curvature flows near such singularities. 
	On the other hand, the famous ``marriage ring'' example (see Figure 17 in \cite{ColdingMinicozziPedersen15_MCFSurvery}), which is the evolution of a very thin torus that collapses to a round circle as the singular set under the mean curvature flow, shows that the singular set can be a smooth curve. As we shall see later, these singularities are {\it degenerate}. 
	
	To state our main theorem, let us begin with some basic settings. We consider mean curvature flows in $\R^{n+1}$, $n\geq 2$. Throughout this paper, ``mean curvature flow'' is always referred to as a unit-regular $n$-dimensional Brakke flow $t\mapsto \bM(t)$ with finite entropy (see these terminologies in Section \ref{Subsec_Brakke}). For readers who are not experts in geometric measure theory, one can view it as a time-dependent family of hypersurfaces satisfying the mean curvature flow equation \[
	\pr_t X=\vec{H}_{\bM(t)}(X) \,,
	\]
	but the hypersurfaces can have singularities. 
	
	Given a singularity $p_\circ = (X_\circ, t_\circ)$ of $\bM$, Huisken \cite{Huisken90} introduced the ``rescaled mean curvature flow (RMCF)" of $\bM$ based at $p_\circ$: \[
	\tau\mapsto \cM^{p_\circ}(\tau)=e^{\tau/2}(\bM(t_\circ-e^{-\tau})-X_\circ)\,,
	\]
	whose long-time subsequential (weak) limits, known as \textit{self-shrinkers}, model the infinitesimal behavior of the flow $\bM$ near $p_\circ$ and hence are called tangent flows of $\bM$ at $p_\circ$. A typical family of self-shrinkers that we focus on in this paper is the \textbf{round (generalized) cylinders}: \[
	\cC_{n,k}:=S^{n-k}\left(\sqrt{2(n-k)}\right)\times\R^k
	\] 
	defined for $k=1,\cdots,n-1$. We say $p_\circ$ is a {\bf $k$-cylindrical singularity} if $\cM^{p_\circ}(\tau)$ locally smoothly converges to a rotation of $\cC_{n,k}$. It was proved by \cite{CM15_Lojasiewicz} that if a subsequence $\cM^{p_\circ}(\tau_i)$ locally smoothly converges to $\cC_{n,k}$, then there exist radius $R(\tau)\to +\infty$ as $\tau\to +\infty$ such that for $\tau\gg 1$, $\cM^{p_\circ}(\tau)$ can be written as a graph of some smooth function $u(\cdot,\tau)$ over $\cC_{n,k}$ in $B_{R(\tau)}$, and that $u(\cdot,\tau)\to 0$ smoothly as $\tau\to +\infty$. We call a cylindrical singularity $p_\circ$ {\bf degenerate} if $u(\cdot, \tau)$ decays exponentially in $\tau$. 
	
	Given a mean curvature flow $\bM$, we define 
	\begin{align}
		\begin{split}
			\cS_{k}(\bM) & := \{p:  p \text{ is a $k$-cylindrical singularity}\} \,;    \\
			\cS_{k}(\bM)_{+} & := \{p\in \cS_{k}(\bM): \text{$p$ is degenerate}\} \,;    \\
			\cS_{k}(\bM)_0 & := \cS_{k}(\bM)\backslash \cS_{k}(\bM)_{+} \,.
		\end{split} \label{Equ_Intro_cS^k, cS^k_+, cS^k_0}
	\end{align}
	
	Our main theorem describes the regularity of the $k$-cylindrical singular sets. In the following, $\dim_H$ denotes the space Hausdorff dimension and $\dim_\cP$ denotes the spacetime parabolic Hausdorff dimension in $\R^{n+1}\times\R$ (see the definition in Section \ref{Subsec_Parab Dist}). 
	
	\begin{theorem} \label{Thm_Intro_C^2 Reg}
		Let $\bM$ be a mean curvature flow in $\R^{n+1}$. Then $\forall\, 1\leq k\leq n-1$, 
		\begin{enumerate}[label={\normalfont(\roman*)}]
			\item\label{Item_Intro_(cS_k)_0} $\dim_\cP \cS_{k}(\bM)_0 \leq k-1$; moreover, when $k=1$, $\cS_{k}(\bM)_0$ consists of isolated points;
			\item\label{Item_Intro_(cS_k)_+ close} $\cS_{k}(\bM)_+$ is relatively closed in $\cS_{k}(\bM)$;
			\item\label{Item_Intro_(cS_k)_+ reg} $\forall\, \al\in (0, \min\{1, \frac2{n-k}\})$, $\cS_{k}(\bM)_+$ is locally contained in a $k$-dimensional $C^{2,\al}$-submanifold, i.e. for every $p\in \cS_k(\bM)_+$, there is a neighborhood $U\subset \R^{n+1}\times \R$ of $p$ and a $k$-dimensional $C^{2,\al}$-submanifold $\Gamma_p$ properly embedded in $U$ such that $\cS_k(\bM)_+\cap U \subset \Gamma_p$;
			\item\label{Item_Intro_t|(cS_k)_+ Holder} let $\mathfrak{t}:\R^{n+1}\times\R\to \R$, $(X,t) \mapsto t$ be the time function, then
			\[
			\dim_H(\mathfrak{t}(\cS_k(\bM)_+))\leq \left(3+\min\left\{1,\, \frac{2}{n-k}\right\} \right)^{-1}\cdot k \,.
			\]
		\end{enumerate}
	\end{theorem}
	
	In addition, if $\cS_{k}(\bM)_+$ is a $k$-dimensional submanifold near $p_\circ\in \cS_k(\bM)_+$, then its second fundamental form at $p_\circ$ is explicitly determined by the asymptotics of the graphical function of the rescaled mean curvature flow $\cM^{p_\circ}$. We will discuss more about this after we discuss the local asymptotics in Theorem \ref{thm:AsyProfile}, see Remark \ref{Rem_Intro_Asymp Profile vs 2nd Fund Form}.
	
	It might be a surprise to see that the singular set carries some regularity, because it does not satisfy any PDE. The proof is based on the connection between the local dynamics of the singularity and the shape of the mean curvature flow near the singularity. The key observation is that the better we know about the local dynamics of the singularity, the better the regularity of the singular set we can get.
	
	As the singularity model, the round cylinders $\cC_{n,k}$ are prevalent from the following aspects:
	
	\begin{enumerate}
		\item the round sphere $S^n(\sqrt{2n})$ and round cylinders $\cC_{n,k}$ are the only stable singularity models from the dynamic point of view; c.f. \cite{ColdingMinicozzi12_generic, SunXue2021_initial_closed, SunXue2021_initial_conical}. Recently, with the work of Chodosh-Choi-Mantoulidis-Schulze and Chodosh-Choi-Schulze \cite{CCMS20_GenericMCF, CCS23_LowEntropy_ii} and the resolution of the multiplicity one conjecture of mean curvature flows in $\R^3$ by Bamler and Kleiner \cite{BamlerKleiner23_Multiplicity1}, it is known that starting from a generic closed embedded surface in $\R^3$, the mean curvature flow has only cylindrical or spherical singularities;
		\item a mean convex mean curvature flow has only cylindrical or spherical singularities; c.f. the work of White \cite{White00, White03, White15_SubseqSing_MeanConvex}, and other approaches by Sheng-Wang \cite{ShengWang09_SingMCF}, Andrews \cite{Andrews12_Noncollapsing} and Haslhofer-Kleiner \cite{HaslhoferKleiner17}. In addition, Huisken-Sinestrari \cite{HuiskenSinestrari99_Acta} proved that the $m$-convex\footnote{$m$-convex means the sum of the smallest $m$ number of principal curvatures at every point is nonnegative; mean convex is equivalent to $n$-convex of hypersurfaces in $\R^{n+1}$} mean curvature flow has only $\cC_{n,k}$ as the tangent flow when $k\leq m-1$;
		\item a mean curvature flow in $\R^3$ starting from a closed embedded surface $\Sigma$ has tangent flow at each point in the top stratum of the singular set to be a round cylinder; and if $\Sigma$ has genus $0$, then the flow can have only cylindrical or spherical singularities, based on the work of Brendle \cite{Brendle16_genus0} and the recent resolution of the multiplicity one conjecture \cite{BamlerKleiner23_Multiplicity1};
		\item a rotational symmetric mean curvature flow has only cylindrical or spherical singularities; c.f. \cite{AAG95_RotationMCF}.
	\end{enumerate}
	
	Let us discuss some interesting special cases of our main theorem. In many situations mentioned above, the singularities of the flow must be cylindrical, and we can apply our main theorem. When the ambient space is $\R^3$ or the initial data of the mean curvature flow is $2$-convex, the lower strata consist of isolated points, and the statement of our main theorem can be simplified. 
	
	\begin{corollary} 
		Suppose $\bM$ is a mean curvature flow in $\R^{n+1}$ starting from a closed embedded hypersurface $\bM(0)$. Then $\forall\, \al\in (0, 1)$, any non-$C^{2,\alpha}$-curve can not be the singular set of $\bM$ if
		\begin{enumerate}
			\item $n=2$ and $\bM(0)$ is mean convex or has genus $0$;
			\item $n\geq 3$ and $\bM(0)$ is $2$-convex.
		\end{enumerate}
	\end{corollary}
	
	It would be interesting to know if $\cS_{k}(\bM)_+$ is not only \textbf{contained} in a $k$-dimensional submanifold, but itself \textbf{is} a $k$-dimensional submanifold. When $n=2$ and $k=1$, this is a part of a conjecture by White \cite[Page 533]{White02ICM}. When $k\geq 2$, we expect that $\cS_{k}(\bM)_+$ could be a non-smooth variety with dimension strictly less than $k$, such as the union of two curves that intersect at a point; see Remark \ref{rem_h_3_dominant}. In semilinear PDE, singularities with such an asymptotic profile have been constructed by Merle-Zaag in \cite{MerleZaag22_DegBlowup} (although their singular set is still isolated). 
	
	We conjecture the following:
	\begin{conjecture}
		Suppose $n\geq 3, k\geq 2$. There exists a mean curvature flow $\bM$ in $\R^{n+1}$ so that $\cS_{k}(\bM)_+$ is not a $k$-dimensional submanifold but a nonsmooth subvariety of dimension $\leq (k-1)$.
	\end{conjecture}
	
	A key ingredient in proving Theorem \ref{Thm_Intro_C^2 Reg} is a general asymptotic profile of cylindrical singularity. In \cite{SunXue2022_generic_cylindrical}, the first two authors proved a normal form theorem for rescaled mean curvature flow at a cylindrical singularity in a ball of size $O(\sqrt{\tau})$; c.f. \cite{Gang21_meanconvexity, Gang22_dynamics} for some special generalized cylinders, and \cite{AngenentDaskalopoulosSesum19_AncientConvexMCF, ChoiHaslhoferHershkovits18_MeanConvNeighb, ChoiHaslhoferHershkovitsWhite22_AncientMCF, DuZhu22_quantization} for the study of ancient asymptotically cylindrical flows. Motivated by the study of rotationally symmetric mean curvature flow by Angenent-Vel{\'a}zquez \cite{AngenentVelazquez97_DegenerateNeckpinches}, it is expected that if the asymptotic profile has a dominant term that decays exponentially fast, then the graphical region should expand exponentially. 
	
	To state our theorem on asymptotic profile, we introduce the following notations. Given a round cylinder $\cC_{n,k}$, we use $\theta=(\theta_1,\theta_2,\cdots,\theta_{n-k+1})$ to denote the coordinates on the sphere factor, and use $y=(y_1,y_2,\cdots,y_k)$ to denote the coordinates on $\R^k$. Note that the outward unit normal vector of $\cC_{n,k}$ at $(\theta, y)$ is $\hat\theta:= \theta/|\theta|$. It's usually convenient to work within the following cube centered at $X=(x, y)\in \R^{n-k+1}\times \R^k$ instead of balls: \[
	Q_r(X):= B^{n-k+1}_r(x) \times B^k_r(y) \,,
	\]
	and we shall omit $X$ if $X=0$. 
	
	We shall use the following convention for ``graphical functions": given a hypersurface $\Sigma\subset \R^{n+1}$, $\delta\in (0, 1]$ and $r>\sqrt{2(n-k)}$, we call $\Sigma$ a \textbf{$\delta$-graph} over $\cC_{n,k}$ in $Q_r$ if there exists a function $v\in C^2(\cC_{n,k}\cap Q_r)$ with $\|v\|_{C^2}\leq \delta$ such that \[
	\Sigma\cap Q_r = \graph_{\cC_{n,k}}(v) := \left\{\left(\theta+v(\theta, y)\hat\theta, y\right):(\theta, y)\in \cC_{n,k}\cap Q_r \right\}\,.
	\]
	The maximal possible $r$ that makes $\Sigma$ a $\delta$-graph is called the \textbf{$\delta$-graphical radius} of $\Sigma$ over $\cC_{n,k}$, denoted by $\gfrd_\delta(\Sigma)$; and the corresponding $v\in C^2(\cC_{n,k}\cap Q_r)$ (extended by $0$ to be defined on the whole $\cC_{n,k}$) is called the \textbf{graphical function} of $\Sigma$ over $\cC_{n,k}$. (If there's no such $r$ that makes $\Sigma$ a $\delta$-graph, we shall use the convention that $\gfrd_\delta(\Sigma)=0$ and the graphical function is also $0$.)
	Similar notions can be generalized to Radon measures $\Sigma$.

	The linearization of the graphical rescaled mean curvature flow equation over $\cC_{n,k}$ is given by $\pr_\tau u=L_{n,k}u$, where \[
	L_{n,k} u = \Delta_{\cC_{n,k}} u - \frac{1}{2}\langle X,\nabla_{\cC_{n,k}} u\rangle + u \,. \]
	We denote by $\sigma(\cC_{n,k})$ the collection of Gaussian $L^2$-eigenvalues of $-L_{n,k}$, and for each $\gamma\in \sigma(\cC_{n,k})$, denote by $\rmW_\gamma(\cC_{n,k})$ the space of Gaussian $L^2$-eigenfunctions of $-L_{n,k}$ with eigenvalue $\gamma$. A more precise discussion of them is included in Section \ref{Subsec_CylinderGeom}.
	
	Using the technique in this paper, we proved the following trichotomy.
	
	\begin{theorem}[Asymptotic profile]\label{thm:AsyProfile}
		Let $\tau\mapsto\cM(\tau)$ be a RMCF in $\RR^{n+1}$ over $\RR_{\geq 0}$ which $C_{\loc}^\infty$-converges to $\cC_{n,k}$ as $\tau\to\infty$. Then one of the following holds for the graphical function $v(\cdot, \tau)$ and $1$-graphical radius $\gfrd_1(\cM(\tau))$ of $\cM(\tau)$ over $\cC_{n,k}$:
		\begin{enumerate}[label={\normalfont(\roman*)}] 
			\item\label{item_Intro_PolynDecay} there exists $\cI\subset\{1,2,\cdots,k\},\ \cI\not=\emptyset$, and some constant $K=K(n,k)>0$ such that after possibly a rotation in $\R^k$ factor, for any $\tau$ sufficiently large, we have $\gfrd_1(\cM(\tau))\geq K\sqrt{\tau}$, and \[
			\left\|v(\cdot,\tau)- \frac{\sqrt{2(n-k)}}{4\tau}\sum_{i\in\cI}(y_i^2-2)\right\|_{L^2(\cC_{n,k})}
			=o(\tau^{-1});
			\footnote{In \cite{SunXue2022_generic_cylindrical}, the definition of the graphical radius and the $C^1$-asymptotic form is in slightly different shapes; see  \cite[Theorem 1.3]{SunXue2022_generic_cylindrical}.}
			\]
			\item\label{item_Intro_ExpDecay} there exist:
			\begin{itemize}
				\item an eigenvalue $\gamma\in \sigma(\cC_{n,k}) \cap [1/2,\infty)$ of $-L_{n,k}$,
				\item an eigenfunction $\psi\in \rmW_\gamma(\cC_{n,k})$ of $L_{n,k}$ with eigenvalue $\gamma$,
				\item a rescaled mean curvature flow called ``low spherical flow'', whose graphical function over $\cC_{n,k}$ is denoted by $\varphi_\bu(\cdot, \tau)\in C^2(\cC_{n,k})$;
				\item an explicit constant $\gamma^+ := \gamma^+(n,k,\gamma) > \gamma$ (defined in \ref{Item_AsyPrNew_Assum2} in Section \ref{Sec_AsympProfile}),
			\end{itemize}
			such that for every $\eps\in (0, \gamma)$ and sufficiently large $\tau$, we have $\gfrd_1(\cM(\tau)) \geq e^{\frac{\gamma-\eps}{2(\gamma+1)} \cdot\tau}$,
			
			and \[
			\left\|v(\cdot,\tau)- (\varphi_\bu(\cdot,\tau) + e^{-\gamma\tau}\psi)\right\|_{L^2(\cC_{n,k})}=O(e^{-(\gamma^+-\eps)\tau})\,;
			\]
			\item\label{item_Intro_SuperExpDecay} there exists a low spherical flow, whose graphical function is $\varphi_\bu$, such that for any $\lambda>0$ and sufficiently large $\tau$, we have $\gfrd_1(\cM(\tau)) \geq e^{\frac{\lambda}{2(\lambda+1)}\cdot \tau}$ and,
			\[
			\|v(\cdot,\tau)- \varphi_\bu(\cdot,\tau)\|_{L^2(\cC_{n,k})}=o(e^{-\lambda\tau})\,.
			\]
		\end{enumerate}
	\end{theorem}
	
	In the statement of this theorem, the ``low spherical flows'' are rescaled mean curvature flows near $\cC_{n,k}$, whose graphical functions $\varphi_\bu$ over $\cC_{n,k}$ satisfy 
	\[
	\begin{cases}
		\|\varphi_\bu(\cdot,\tau)\|_{C^4}\lesssim e^{-\frac{\tau}{n-k}}, & \text{ when $n-k\geq 2$;}
		\\
		\varphi_\bu(\cdot,\tau)=0, & \text{ when $n-k=1$}.
	\end{cases}
	\]
	They are called ``spherical flows'' because they are invariant in $\R^k$ directions, and hence only the spherical part matters; the word ``low'' indicates that when $n-k\geq 2$, they are the rescaled mean curvature flow converging to the sphere factor with the slowest rate. We construct a family of low spherical flows in Lemma \ref{Subsec_Mod Low SphericalMode} based on Appendix \ref{App:ExistenceSphericalArrival}. We include the term $\varphi_\bu$ in our asymptotic profile because it represents the lowest order asymptotics when $n-k\geq 2$, but it does not contribute to the geometry of the singular set. We also remark that when $n-k=1$, there's no spherical flow (except $\cC_{n,k}$ itself) that represents the lowest order asymptotics on the cylinder, and we do not need them in our analysis, so we set them to be zero. 
	
	\begin{Rem} \label{Rem_Intro_super exp decay rigid}
		The case \ref{item_Intro_PolynDecay} of Theorem \ref{thm:AsyProfile} has been proved in \cite{SunXue2022_generic_cylindrical}. Using our stratification of cylindrical singular set \eqref{Equ_Intro_cS^k, cS^k_+, cS^k_0}, a singularity $p\in \cS_k(\bM)_+$ if and only if its RMCF $\cM^p$ satisfies case \ref{item_Intro_ExpDecay} or \ref{item_Intro_SuperExpDecay} after some rotation in $\R^{n+1}$. 
		
		In \cite{AngenentVelazquez97_DegenerateNeckpinches}, Angenent-Vel{\'a}zquez constructed examples of rotationally symmetric mean curvature flows that have the asymptotic profile as in \ref{item_Intro_ExpDecay} near their singularities, and the graphical radii are proportional to $e^{\frac{\gamma}{2(\gamma+1)}\tau}$. Therefore, our graphical radius bound $e^{\frac{\gamma-\eps}{2(\gamma+1)}\tau}$ obtained in \ref{item_Intro_ExpDecay} is almost optimal. We believe that in case \ref{item_Intro_ExpDecay}, the $\eps$ in the graphical radius bound can be removed. Nevertheless, as it is not necessary in this paper, we do not pursue this sharp graphical radius estimate. 
		
		The case \ref{item_Intro_SuperExpDecay} is the case of super-exponential decay rate, and by the spirit of unique continuation of PDEs, it seems plausible that it only happens if $(v-\varphi_\bu)(\cdot, \tau)\equiv 0$. 
	\end{Rem}
	
	\begin{conjecture}
		Suppose $\cM$ is a rescaled mean curvature flow in $\R^{n+1}$ without boundary, and suppose the case \ref{item_Intro_SuperExpDecay} of Theorem \ref{thm:AsyProfile} holds. Then $\cM$ coincides with some low spherical flow.
	\end{conjecture}
	
	\begin{Rem} \label{Rem_Intro_Asymp Profile vs 2nd Fund Form}
		A particularly interesting situation in \ref{item_Intro_ExpDecay} is that $\gamma=1/2$ (and for simplicity $n-k\neq 2$). In this situation, the eigenspace $\rmW_{1/2}(\cC_{n,k})$ is spanned by $\theta_i (y_j^2-2)$, $\theta_i y_{j_1}y_{j_2}$, and $y_j^3-6y_j$. If we collect all the terms of the first two types, we can write any such eigenfunction $\psi\in \rmW_{1/2}$ as \[
		\psi(\theta, y) = \sum_{i=1}^{n-k+1}|\theta|^{-1}\theta_i(\by^\top \cA_{i} \by-2\tr\cA_{i}),
		\]
		where $\cA_i$'s are symmetric $k\times k$ matrices. These coefficients indeed record the geometric information of the singular set: if near a $k$-cylindrical singularity $p$, the singular set $\cS_k(\bM)$ is a $k$-dimensional submanifold, then by Theorem \ref{Thm_Intro_C^2 Reg} \ref{Item_Intro_(cS_k)_0} and \ref{Item_Intro_(cS_k)_+ close}, in a small neighborhood $U_p$ of $p$, we have $\cS_k(\bM)\cap U_p=\cS_k(\bM)_+\cap U_p$. After an appropriate rotation, let us assume the rescaled mean curvature flow $\cM^p$ converges to $\cC_{n,k}$, then the tangent space $T_p (\cS_k(\bM)_+)$ is $\{0\}\times \R^k$, and our proof of Theorem \ref{Thm_Intro_C^2 Reg} also yields:
		\begin{itemize}
			\item if case \ref{item_Intro_ExpDecay} happens with $\gamma=1/2$, then $2\cA_i$ is the projection of the second fundamental form of $\cS_k(\bM)_+$ at $p$ to the normal direction $\partial_i$;
			\item if case \ref{item_Intro_ExpDecay} happens with $\gamma>1/2$ or case \ref{item_Intro_SuperExpDecay} happens, then the second fundamental form of $\cS_k(\bM)_+$ at $p$ vanishes.
		\end{itemize}
		See the general statement in Proposition \ref{Prop_C^2 Reifenberg}.
		
		In contrast, the terms $(y_j^3-6y_j)$ could be obstructions to the $k$-dimensionality and the manifold structure of the singular set. The mechanism behind this is similar to the singularities in $\cS_k(\bM)_0$. We refer the readers to Remark \ref{rem_h_3_dominant}.
	\end{Rem}    
	
	\subsection{Regularity of singular sets of mean curvature flows}
	As a classical problem, the fine structure of singular sets has been studied in many other geometric and PDE problems, such as (far from a complete list) the minimal surfaces \cite{Federer70_DimReduc, Almgren83_BAMS, Simon93_Cylin, White97_Stratif}, 
	semilinear PDEs \cite{Zaag02_C1Reg, Zaag02_C1alphaReg, Zaag06_C2Reg}, and obstacle problem \cite{FigalliSerra19_FineStructObstable, FranceschiniZaton25_CInfty_Reg_Obstacle}.
	
	Let us review some history on the study of the cylindrical singular set of mean curvature flow. Given a self-shrinker $\Sigma\times\R^k$ with $\Sigma$ being non-minimal, the $\R^k$ factor is usually called the spine\footnote{In \cite{White97_Stratif}, the spine is defined for a general shrinking mean curvature flow, including the static or quasi-static cases.} of this shrinker. In \cite{White97_Stratif}, White proved that the singular set of mean curvature flow has a stratification based on the dimension of the spine. Furthermore, he proved that the singular set that is modeled by shrinkers with a $k$-dimensional spine has parabolic Hausdorff dimension at most $k$. Later, Cheeger-Haslhofer-Naber \cite{CheegerHaslhoferNaber13} refined White's result and proved a quantitative stratification theorem; for mean convex flow, this estimate was further refined recently by Hanbing Fang and Yu Li \cite{FangLi25_VolMCFSing}.

	The regularity of the cylindrical singular set was first studied by Colding-Minicozzi in their seminal work \cite{ColdingMinicozzi16_SingularSet}. They proved that the $k$-th strata $\cS_{k}(\bM)$ are contained in finitely many $k$-dimensional Lipschitz submanifolds. Their proof shows that these submanifolds can be taken to be $C^1$, see the footnote of \cite[Page 815]{ColdingMinicozzi18_Regularity_LSF}. The proof used their earlier result on the uniqueness of cylindrical tangent flow \cite{CM15_Lojasiewicz}, which allows one to compare the spines at two nearby cylindrical singularities.

	It would be interesting to compare our result with the pioneering work in semilinear PDE by Zaag \cite{Zaag02_C1Reg, Zaag02_C1alphaReg, Zaag06_C2Reg}. Zaag proved the $C^2$-regularity of the singular set of semilinear PDEs under the assumption that the singular set is a submanifold. The mechanism is similar to getting higher regularity of $\cS_k(\bM)_0$ for mean curvature flows, which we do not address in the current paper. The curvature of this singular set is due to the rotation of the degenerate direction within the spine. In contrast, Remark \ref{Rem_Intro_Asymp Profile vs 2nd Fund Form} suggests that the curvature of the degenerate cylindrical singular set originates from the spherical part, which is different from the semilinear PDE setting. Also, the semilinear PDE has only singularities for the first singular time. In contrast, the weak flow theory of MCF allows us to study all time singularities, even those appearing after the first singular time.

	\subsection{Sketch of the proof}
	Our strategy of deriving higher regularity of singular sets is inspired by \cite{FigalliSerra19_FineStructObstable}. Intuitively, the singularities in $\cS_{k}(\bM)_+$ can only show up near the spine directions of each other. To show that $\cS_k(\bM)_+$ is locally contained in some $C^{\ell,\al}$ submanifold $\Gamma$, we need to associate to each such singularity some ``Whitney data" that will approximate $\Gamma$ up to order $\ell$, estimate how close these assigned data are at different singularities, and then apply Whitney extension theorem to construct such $\Gamma$. When $\ell=2$, the Whitney data associated to $p\in \cS_k(\bM)_+$ are: 
	\begin{itemize}
		\item a $k$-dimensional linear subspace $|\bL_p|$, which will be the tangent direction of $\Gamma$ at $p$, given by the spine of tangent flow of $\bM$ at $p$;
		\item an $|\bL_p|^\perp$-valued quadratic form $\bq_p$ on $|\bL_p|$, which will be the second fundamental form of $\Gamma$ at $p$, determined by the asymptotic profile of the RMCF $\cM^p$ as in Remark \ref{Rem_Intro_Asymp Profile vs 2nd Fund Form}. (For later reference, also let $\bQ_p$ be the pre-composition of $\bq_p$ with the orthogonal projection onto $|\bL_p|$, which is an $\R^{n+1}$-valued quadratic form on $\R^{n+1}$.)
	\end{itemize}
	Moreover, the ``closeness" condition of Whitney data at two different singularities will (by a translation and rotation) pin down to the following location estimate: if
	\begin{itemize}
		\item $(\orig,0)\in \cS_k(\bM)_+$ has Whitney data $|\bL_{(\orig,0)}|=\{0\}\times \R^k$, $\bq_{(\orig,0)}=\bq_\circ$ (an $\R^{n-k+1}$-valued quadratic form on $\R^k$) and $\bQ_{(\orig,0)}=\bQ_\circ$;
		\item $\bar p\in \cS_k(\bM)_+$ has Whitney data $\bar\bL, \bar\bq, \bar\bQ$, and set $\bar p = (\bar x, \bar y, \bar t)\in \R^{n-k+1}\times \R^k\times \R$, suppose $\max\{|\bar x|^2+|\bar y|^2, |\bar t|\} =: r^2\in (0, 1]$.
	\end{itemize}
	Then there exists a linear map $\bl: \R^k\to \R^{n-k+1}$ (also viewed as an $(n-k+1)\times k$-matrix) such that the rotation $e^{\bar\mbfA}$ maps the linear subspace $|\bL_\circ|$ to $|\bar\bL|$, where $\bar\mbfA:= \begin{bmatrix} 0 & \bl \\ -\bl^\top & 0 \end{bmatrix}$; and the following estimate holds for every $0<\eps\ll 1$:
	\begin{align}
		r^{-2}|\bar t| + r^{-1}|\bar x - \bq_\circ(\bar y)| + \|\bar\bl - \nabla\bq_\circ(\bar y)\| + r\|\bar \bQ - \bQ_\circ\| \lesssim_{\eps} r^{2\gamma^+_{n,k} - 2\eps}\,, \label{Equ_Intro_Location Est}
	\end{align}
	where $\gamma^+_{n,k}=\gamma^+(n,k,1/2)>1/2$ is determined in the last bullet point of \ref{item_Intro_ExpDecay} in Theorem \ref{thm:AsyProfile}.

	\subsubsection{Proof of location estimate}
	Taking Theorem \ref{thm:AsyProfile} for granted, we now describe the proof idea of the location estimate. Recall that $(\orig,0)$ being a degenerate cylindrical singularity implies, by Theorem \ref{thm:AsyProfile} \ref{item_Intro_ExpDecay} and \ref{item_Intro_SuperExpDecay}, that the RMCF $\tau\mapsto \cM(\tau)$ based at $(0, 0)$ has its graphical function $v(\cdot, \tau)$ satisfying the following asymptotic estimate, 
	\begin{align}
		v(\cdot, \tau) = \varphi_{\bu_\circ}(\cdot, \tau) + e^{-\tau/2}\psi_\circ + O(e^{-(\gamma^+_{n,k}-\eps)\tau})\,,  \label{Equ_Intro_v sim e^(-tau/2)psi} 
	\end{align}
	where $\varphi_{\bu_\circ}$ is the graphical function over $\cC_{n,k}$ of a low spherical flow, $\psi_\circ\in \rmW_{1/2}(\cC_{n,k})$ is an eigenfunction with eigenvalue $1/2$.
	Note that we allow the scenario when $\psi_\circ = 0$, which corresponds to faster asymptotics $\gamma>1/2$ in case \ref{item_Intro_ExpDecay} or case \ref{item_Intro_SuperExpDecay} of Theorem \ref{thm:AsyProfile}. 
	
	On the other hand, for every $r, \bar p$ specified as above, we can write the RMCF $\cM^{\bar p}$ based at $\bar p$ as the translation, dilation and time-reparametrization of $\cM$: \[
	\cM^{\bar p}(\tau) = \sqrt{1-\bar t e^{\tau}}\cdot
	\cM(\tau-\log(1 - \bar t e^\tau))-e^{\tau/2}(\bar x, \bar y) \,.
	\]
	By a quantitative uniqueness of tangent flow by Colding and Minicozzi \cite{ColdingMinicozzi25_quantitativeMCF}, when $\bar p$ is close enough to $(0, 0)$, we shall have
	\begin{itemize}
		\item $|\bar t|, |\bar x|^2 \ll |\bar y|^2 \sim r^2$;
		\item $|\bar\bL|$ is close to $|\bL_\circ|$.
	\end{itemize}
	The second bullet point guarantees that there exists such a linear map $\bl$ as described above \eqref{Equ_Intro_Location Est}. In particular, the rotation $e^{-\bar\mbfA}$ maps $\cM^{\bar p}$ to a RMCF $\tau\mapsto \bar \cM(\tau):= (e^{-\bar\mbfA})_\sharp\cM^{\bar p}(\tau)$ which locally smoothly converges to $\cC_{n,k}$ as $\tau\to +\infty$. 
	We proved in Appendix \ref{Append_Graph over Cylinder} that the graphical function $\bar v(\cdot, \tau)$ of $\bar\cM(\tau)$ is related to $v$ by unwrapping these translation, dilation and rotation: \[
	\bar v(\theta, y; \tau) \sim v(\theta, y + \by, \tau') + \varrho(\bar\lambda -1) - \psi_{\bar\bx} - \psi_{\bar\mbfA}\,,
	\]
	where $\varrho:= \sqrt{2(n-k)}$ and
	\begin{align*}
		\tau':= \tau-\log(1 - \bar t e^\tau)\,, & &
		\bar\lambda := \sqrt{1-\bar t e^{\tau}}\,, & &
		\bar\mbfx := e^{\tau/2}\bar x\,, & & 
		\bar\by := e^{\tau/2}\bar y \,;
	\end{align*}
	and $\psi_{\bar\mbfx}, \psi_{\bar\mbfA}$ are eigenfunctions of $-L_{n,k}$ induced by translation and rotation (see Section \ref{SSubsec_Spectrum of L} for their precise definitions). Combined with \eqref{Equ_Intro_v sim e^(-tau/2)psi}, this implies that if we choose $\tau$ so that $|\bar\by|\sim 1$ (which means $r\sim |\bar y|\sim e^{-\tau/2}$), then
	\begin{align}
		\begin{split}
			\bar v(\theta, y; \tau) \sim \varphi_{\bu_\circ}(\theta; \tau') + e^{-\tau'/2}\psi_\circ(\theta, y+\by) + \varrho(\bar\lambda -1) - \psi_{\bar\bx} - \psi_{\bar\mbfA} + O(e^{-(\gamma^+_{n,k}-\eps)\tau})  \,.
		\end{split} \label{Equ_Intro_bar v = psi_o(by+ ) + ...}
	\end{align}
	
	On the other hand, since $\bar p$ is also a degenerate cylindrical singularity for $\bM$, Theorem \ref{thm:AsyProfile} \ref{item_Intro_ExpDecay}, \ref{item_Intro_SuperExpDecay} also applies to $\bar\cM$, which gives the asymptotic estimate, \[
	\bar v(\cdot, \tau) = \varphi_{\bar\bu}(\cdot, \tau) + e^{-\tau/2}\bar\psi + O(e^{-(\gamma^+_{n,k}-\eps)\tau})
	\]
	for some low spherical flow $\varphi_{\bar\bu}$ and eigenfunction $\bar\psi\in \rmW_{1/2}(\cC_{n,k})$. Combined with \eqref{Equ_Intro_bar v = psi_o(by+ ) + ...}, this implies (modulo estimates on the errors),
	\begin{align*}
		\left\|(\varphi_{\bu_\circ}-\varphi_{\bar \bu})(\cdot, \tau) + e^{-\tau'/2}\left(\psi_\circ(\cdot+\by)-\bar\psi \right) + \left(\varrho(\bar\lambda -1) - \psi_{\bar\bx} - \psi_{\bar\mbfA}\right) \right\|_{L^2} \lesssim e^{-(\gamma^+_{n,k}-\eps)\tau}\,.
	\end{align*}
	Projecting the left-hand side onto the eigenspaces of $-L_{n,k}$ with eigenvalues $-1, -1/2, 0, 1/(n-k), 1/2$ proves the estimates in each term of the location estimate \eqref{Equ_Intro_Location Est}.
	
	Let us finish this discussion by comparing our approach with that employed in the obstacle problem by Figalli-Serra \cite{FigalliSerra19_FineStructObstable} and Franceschini-Zato{\'n} \cite{FranceschiniZaton25_CInfty_Reg_Obstacle}. 
	In these works, the object of study is a function $u$ solving a certain PDE. At each singularity belonging to the subset with improved regularity, one can approximate $u$ by polynomials to arbitrarily high order; singular points that do not admit such increasingly accurate polynomial approximations are shown to form a lower-dimensional set.
	In contrast, in the setting of MCF $\bM$, due to the nature of nonlinearity, it is hard to exhaust the ``approximate solutions" at the level of the flow that would play the role of polynomials\footnote{To obtain $C^{2,\al}$ regularity, our approximations are tangent flows, low spherical flows and eigenfunctions of Jacobi operator with eigenvalue $1/2$}. For example, higher-order expansions at singularities may naturally contain spherical modes: these do not influence the geometry of the singular set (and thus the singularity need not belong to a lower stratum), yet they still contribute to the approximation of the flow.
	Moreover, translating such approximations of $\bM$ at the level of flow into Whitney data that capture the singular set and verifying the compatibility condition is more involved than in the polynomial case.
	For these reasons, we only get $C^{2,\al}$ regularity for singular sets in the MCF setting instead of $C^\infty$.

	\subsubsection{Some ingredients in proving Theorem \ref{thm:AsyProfile}} 
	Let us briefly describe the proof of the asymptotic profile Theorem \ref{thm:AsyProfile} in the case $n-k=1$ (where we don't need to modulo low spherical flows) and the flow decays exponentially toward the cylinder. For a RMCF $\tau\mapsto \cM(\tau)$ with entropy $\leq \eps^{-1}$, in \cite{SunWangXue1_Passing}, we introduced the ``(Gaussian) $L^2$-distance" $\mbfd_{n,k}(\cM(\tau))$ to the cylinder $\cC_{n,k}$ and the ``decay order" (a discrete parabolic analog of Almgren's frequency function): \[
	\cN_{n,k}(\tau, \cM) := \ln \left(\frac{\mbfd_{n,k}(\cM(\tau))}{\mbfd_{n,k}(\cM(\tau+1))} \right) \,,
	\] 
	which measures how fast $\mbfd_{n,k}(\cM(\tau))$ is decaying near $\tau$-time slice. In \cite[Section 3]{SunWangXue1_Passing}, we proved a new $L^2$-nonconcentration property of rescaled mean curvature flow and, as an application, a discrete monotonicity of decay order, which roughly speaking has the following useful consequences: 
	\begin{itemize}
		\item if $\cM(\tau)$ converges locally smoothly to $\cC_{n,k}$ as $\tau\to +\infty$, then $\lim\limits_{\tau\to +\infty} \cN_{n,k}(\tau, \cM)$ is either $\infty$ or an eigenvalue of $L_{n,k}$;
		\item for every $\tau'\geq 0$, we have the following universal bound (recall that $-1$ is the smallest eigenvalue of $-L_{n,k}$) \[
		\mbfd_{n,k}(\cM(\tau+\tau')) \lesssim_{n, \eps} e^{(1+\eps)\tau'} \mbfd_{n,k}(\cM(\tau)) \,.
		\]
	\end{itemize}

	With these two bullet points and a novel translation argument, we can prove the graphical radius of a RMCF $\cM(\tau)$ over $\cC_{n,k}$ is $\gtrsim e^{\frac{\gamma-\eps}{2(\gamma+1)}\cdot\tau}$ when $\lim\limits_{\tau\to +\infty} \cN_{n,k}(\tau, \cM)= \gamma< +\infty$, see a precise discussion in Lemma \ref{Lem_AsympProf_Rough C^2 est}. After estimating the graphical radius, we can write $\cM(\tau)$ as a graph of some $v(\cdot, \tau)\in C^2$ over $\cC_{n,k}$ in an exponentially expanding cube. The RMCF equation on $\cM$ can be interpreted as a parabolic equation of $v$, \[
	(\partial_\tau - L_{n,k}) v = \cQ(u, \nabla u, \nabla^2 u) \,,
	\]
	where the right-hand side is a quadratic error term. The argument above also gives an a priori spacetime $C^2$ estimate on $v$ of form \[
	\|v(\cdot, \tau)\|_{C^2(\cC_{n,k}\cap Q_R)} \lesssim_{n, \eps} e^{-(\gamma-\eps)\tau} R^{2(1+\gamma)} \,,
	\]  
	which leads to an upper bound on the error term decaying like $e^{-2(\gamma-\eps)\tau}$, faster than $e^{-\gamma\tau}$. Making use of this, the principal decaying mode can be extracted from $v(\cdot, \tau)$ with an error estimate by a purely PDE argument. 
	
	As a final remark about the proof of Theorem \ref{thm:AsyProfile}, when $n-k\geq 2$, one needs to replace the $L^2$-distance and decay order above by the ones relative to a low spherical flow to reprove the discrete monotonicity and its consequences. This is done in Section \ref{Sec_L^2 Noncon}. 
	
	\subsubsection{Structure of $\cS_{k}(\bM)_0$}
	Finally we discuss the structure of $\cS_{k}(\bM)_0$. These singularities have asymptotic profiles as indicated in item \ref{item_Intro_PolynDecay} of Theorem \ref{thm:AsyProfile}. When $\cI=\{1,2,\cdots,k\}$, the singularities are called nondegenerate, and have been studied extensively in \cite{SunWangXue1_Passing}. In particular, they are isolated in spacetime. When $\emptyset\neq \cI\subsetneqq\{1,2,\cdots,k\}$, the singularities are called {\bf partially nondegenerate}, and they carry the structure of nondegenerate singularities in the directions generated by $\{y_i|i\in\cI\}$ in the spine. In particular, the asymptotic profile allows us to conclude that the nearby singularities can not approach a partially nondegenerate singularity from those nondegenerate directions. Equivalently, this shows that the partially nondegenerate singularities can only stay close to the degenerate directions in the spine. Hence, the Hausdorff dimension of partially nondegenerate singularities must be smaller.

	It is plausible to get higher regularity of $\cS_{k}(\bM)_0$, for example, by combining the ideas in the current paper with the idea of Zaag in \cite{Zaag02_C1Reg, Zaag02_C1alphaReg, Zaag06_C2Reg}. Nevertheless, the technical details can be very different from the current paper.

	\subsection{Organization of the paper}
	After a preliminary section, Section \ref{Sec_Prelim}, we define and discuss the family of low spherical flows, and basic properties of the relative decay order $\cN_{\bu}$ in Section \ref{Sec_L^2 Noncon}. In Section \ref{Sec_AsympProfile}, we discuss the key asymptotic profile analysis at the singularities in $\cS_{k}(\bM)_+$ and the proof of Theorem \ref{thm:AsyProfile}. In Section \ref{Sec_C^(2,al) Reg}, we discuss how to use the results in Section \ref{Sec_AsympProfile} to estimate the location of singularities in $\cS_{k}(\bM)_+$. In Section \ref{Sec_PfMainThm}, we prove Theorem \ref{Thm_Intro_C^2 Reg} by using the Whitney extension theorem, based on the location estimates derived in Section \ref{Sec_C^(2,al) Reg}.  Finally, we construct in Appendix \ref{App:ExistenceSphericalArrival} a continuous family of convex mean curvature flows with prescribed asymptotics to the round sphere, and prove some key properties of this family. We collect basic properties on parabolic Jacobi fields and geometry of graph over the round cylinder in Appendices \ref{Append_Ana Parab Jac} and \ref{Append_Graph over Cylinder}. These will be used throughout the paper.
	
	\subsection{Acknowledgment}
	We thank Sigurd Angenent, Toby Colding, Alessio Figalli, and Bill Minicozzi for their interest.
	A.S. is supported by the AMS-Simons Travel Grant. J. X. is supported by NSFC grants (No. 12271285) in China, the New Cornerstone investigator program and the Xiaomi endowed professorship of Tsinghua University. 
	
	\section{Preliminary} \label{Sec_Prelim}
	
	\subsection{Parabolic metric spacetime and Hausdorff measure} \label{Subsec_Parab Dist}
	We recall the following geometric structure of the parabolic spacetime. Throughout this paper, the ambient space is $\R^{n+1}$, and the ambient spacetime is $\R^{n+1}\times\R$.    
	For $X_\circ\in \R^{n+1}$ and $\lambda>0$, we let $\eta_{X_\circ, \lambda}: X\mapsto X_\circ+\lambda X$ denote the translation and dilation map. For a subset $S\subset \R^{n+1}$ or a varifold $V$ in $\R^{n+1}$, we may simply denote their translation and dilation by $\lambda S + X_\circ := \eta_{X_\circ, \lambda}(S)$, $\lambda V + X_\circ:= (\eta_{X_\circ, \lambda})_\sharp V$.
	
	For $(X_1,t_1),\,(X_2,t_2)\in \R^{n+1}\times\R$, the parabolic distance between them is defined to be \[
	d_{\cP}((X_1,t_1),(X_2,t_2))=\max\{|X_1-X_2|,|t_1-t_2|^{1/2}\}.
	\]
	In particular, the metric ball of radius $r$ centered at the spacetime origin $(\orig,0)$ is the parabolic neighborhood $B_r(\orig)\times[-r^2,r^2]$, denoted by $P_r(\orig,0)$.

	Recall that given $\beta>0$ and $\delta \in (0,\infty]$, the parabolic Hausdorff premeasures $\mathcal H^{\beta}_{\cP,\delta}(E)$ of a subset $E\subset \R^{n+1}\times \R$ are defined as follows (up to a dimensional constant multiple):
	\begin{equation}
		\mathcal H^{\beta}_{\cP,\delta}(E):=\inf\biggl\{\sum_i {\rm diam}(E_i)^\beta\,:\,E\subset \bigcup_i E_i,\,{\rm diam}(E_i)<\delta \biggr\}.
	\end{equation}
	Here ${\rm diam}(E_i)$ is the diameter of $E_i$ measured under $d_\cP$. The parabolic Hausdorff measure is then defined as $\cH^\beta_{\cP}(E):= \lim\limits_{\delta\to 0}\cH^\beta_{\cP,\delta}(E)$, and the parabolic Hausdorff dimension is \[
	\dim_\cP(E) := \inf\{\beta\geq 0: \cH^\beta_{\cP}(E)=0\} \,.
	\]
	Note that if $K\subset\R^{n+1}$, then $\cH^k(K\times\{0\})$ is just the usual Hausdorff measure of $K$ in the Euclidean space; on the other hand, the time axis $\{\orig\}\times\R$ has parabolic Hausdorff dimension $2$.
	
	The following lemma is a variant of \cite[Lemma 3.5]{FigalliSerra19_FineStructObstable} in the parabolic setting, whose proof also generalizes directly, and we skip it here. It will be used to estimate $\dim_\cP \cS_k(\bM)_0$ in Section \ref{Subsec_cS_0 lower dim}.
	\begin{Lem}[{\cite[Lemma 3.5]{FigalliSerra19_FineStructObstable}}] \label{lem:HausdorffMeasure}
		Let $E\subset \R^n\times \R$ be a set with $\cH^\beta_{\cP}(E)>0$ for some $\beta\in (0, n+2]$. Then for $\cH^\beta_{\cP}$-almost every point $p_\circ \in E$, there is a sequence $r_j\downarrow 0$ such that the following ``accumulation set'' for $E$ at $0$: \[
		\mathcal A_{E, p_\circ} := \left\{z\in \overline{P_{1/2}} \,:\exists\,(j_\ell)_{\ell \ge 1}\nearrow +\infty, (z_\ell)_{\ell\ge 1} \text{ s.t. $z_{\ell}\in r_{j_\ell}^{-1}.(E-p_\circ) \cap P_{1/2}$ and $z_{\ell}\to z$} \right\}, 
		\]
		(where $r_{j_\ell}^{-1}. E:= \{(r_{j_\ell}^{-1}X, r_{j_\ell}^{-2}t): (X, t)\in E\}$ denotes the parabolic dilation of $E$) satisfies \[
		\cH_{\cP}^\beta(\mathcal A_{E, p_\circ}) >0 \,.
		\]
	\end{Lem}

	\subsection{Brakke flow and rescaled mean curvature flow} \label{Subsec_Brakke}
	We recall that a (classic) mean curvature flow in $\R^{n+1}$ over an interval $I$ is a family of smooth embeddings $F: M \times I \to \R^{n+1}$ of an $n$-dimensional manifold $M$, satisfying the equation $\pr_t F(x,t)=\vec{H}(F(x,t))$. 
	In order to study singular sets, we need to work with weak flows. 
	\begin{definition}
		An $n$-dimensional \textbf{(integral) Brakke flow} in $\RR^{n+1}$ over an interval $I\subset \RR$ is a one-parameter family of Radon measures $\{\Sigma_t\}_{t\in I}$, such that
		\begin{enumerate}
			\item for almost every $t\in I$, $\Sigma_t$ is a Radon measure associated to an $n$-dimensional integral varifold with mean curvature $\vec H_t\in L^2(\Sigma_t)$, 
			\item for every non-negative function $\Phi\in C^1_c(\R^{n+1}\times I)$, we have
			\begin{equation}
				\frac{d}{dt}\int \Phi\ d\Sigma_t
				\leq 
				\int \left(
				\frac{\partial \Phi}{\partial t}
				+\nabla \Phi\cdot \vec{H}_t
				-\Phi |\vec H_t|^2
				\right)
				\ d\Sigma_t, \label{Equ_Pre_DefBrakke}
			\end{equation}
			in the distribution sense.
		\end{enumerate}
		The support of the Brakke flow is defined to be $\spt (\Sigma):= \overline{\bigcup_t \spt(\Sigma_t)\times\{t\}}$, where the closure is taken in the spacetime.
	\end{definition}
	
	\begin{definition}
		A point $p_\circ = (X_\circ, t_\circ)\in \R^{n+1}\times \R$ in the support of $\bM$ is called a \textbf{regular point} of $\bM$ if there exists $r>0$ such that in $B_r(X_\circ)$, $t\mapsto \bM(t)$ is the area measure associated to a classical mean curvature flow over $(t_\circ-r^2, t_\circ]$; otherwise, $p_\circ$ is called a \textbf{singularity} of $\bM$. 
	\end{definition}

	We say that a sequence of flow $\{\Sigma_t^j\}_{t\in I}$ converges to $\{\Sigma_t^\infty\}_{t\in I}$ \textbf{in the Brakke sense}, if $\Sigma^j_t$ measure-converges to $\Sigma^\infty_t$ for all $t\in I$, and the associated varifolds converge for a.e. $t\in I$. 
	We will use the important Brakke-White's $\eps$-regularity \cite{White05_MCFReg}, which implies that if $\{\Sigma_t^j\}_{t\in I}$ converges to $\{\Sigma_t^\infty\}_{t\in I}$ in the Brakke sense and $\{\Sigma_t^\infty\}_{t\in I}$ is a classic mean curvature flow of smooth hypersurfaces, then for every compact subset $K\subset \R^{n+1}\times I$ and $j\gg 1$ (depending on $K$), inside $K$, $\{\Sigma_t^j\}_{t\in I}$ is also a classic mean curvature flow of smooth hypersurfaces, and the convergence is smooth.

	We remark that another important weak flow is the so-called level set flow, which is closely related to the Brakke flow whenever it does not fatten. We refer the readers to the preliminary section of \cite{SunWangXue1_Passing} for a more comprehensive preliminary study of weak flows of mean curvature flow, and \cite{Brakke78, Ilmanen94_EllipReg, White95_WSF_Top} for detailed discussions. In particular, the level set flow of convex mean curvature flow can be formulated as the {\bf arrival time function}. Let $\Omega\subset \RR^{n+1}$ be a convex open subset, then a twice differentiable function $\bu:\Omega\to \RR$ is called the arrival time function if it satisfies the following equation
	\begin{align}
		|\nabla \bu|^2(\Delta \bu + 1) = \nabla^2\bu (\nabla \bu, \nabla \bu)\,. 
	\end{align}
	We refer the readers to \cite{EvansSpruck91, Ilmanen92_LSF, White95_WSF_Top, ColdingMinicozzi16_Differentialbility}, among other references, for further discussions of the arrival time functions.
	
	Huisken \cite{Huisken90} introduced the ``rescaled mean curvature flow" to study singularities. Given a Brakke flow $t\mapsto \mbfM(t)$ and a spacetime point $p_\circ=(X_\circ,t_\circ)$, one can define a new flow $\cM^{p_\circ}: \tau\mapsto \cM^{p_\circ}(\tau)$ associated to $\mbfM$, called the {\bf rescaled mean curvature flow (RMCF)} of $\bM$ based at $p_\circ$, whose time slices are defined by
	\begin{equation}
		\cM^{p_\circ}(\tau)=e^{\tau/2}(\mbfM(t_\circ-e^{-\tau})-X_\circ).
	\end{equation}
	Huisken observed that the rescaled mean curvature flow is a gradient flow of the Gaussian area, which is known as Huisken's monotonicity formula. Recall that given a Radon measure $\Sigma$ in $\R^{n+1}$, its $n$-dimensional {\bf Gaussian area} is defined by \[
	\cF[\Sigma]:=\int_{\RR^{n+1}} (4\pi)^{-n/2}e^{-\frac{|X|^2}{4}}d\Sigma(X).
	\]
	In particular, for a hypersurface $M$, we denote for simplicity $\cF[M]:= \cF[\|M\|]$, where $\|M\|$ is the $n$-dimensional volume measure associated to $M$. As a consequence of Huisken's monotonicity formula, we define the {\bf Gaussian density} of a mean curvature flow at the spacetime point $p$ as $$\Theta_{\bM}(p):=\lim\limits_{\tau\to+\infty}\cF[\cM^{p}(\tau)] \,.  $$
	
	It is worth noting that the Gaussian density is upper semi-continuous in the following sense: for a sequence of MCFs $\bM_j$ converging to $\bM_\infty$ in Brakke sense, and a sequence of spacetime points $p_j$ converging to $p_\infty$, we have $\Theta_{\bM_\infty}(p_\infty)\geq \limsup_j \Theta_{\bM_j}(p_j)$; see \cite[Page 4]{IlmanenWhitw15_DensityBd}.
	As a consequence, every $p\in \spt(\bM)$ has Gaussian density $\Theta_\bM(p)\geq 1$, and Brakke-White's $\eps$-regularity guarantees that equality holds if and only if $p$ is a regular point of $\bM$.
	We say an integral Brakke flow is {\bf unit regular} if every point with Gaussian density $1$ has a space-time neighborhood where it is smooth. By White \cite{White05_MCFReg}, being unit regular is preserved under Brakke convergence of MCF.
	
	Colding-Minicozzi \cite{ColdingMinicozzi12_generic} introduced the {\bf entropy} of a hypersurface (or an $n$-varifold): \[
	\lambda[M]:=\sup_{(X_\circ,t_\circ)\in\R^{n+1}\times\R_{>0}} \cF[t_\circ^{-1}(M-X_\circ)].
	\]
	By Huisken's monotonicity formula, entropy is monotone nonincreasing along a Brakke flow. For a Brakke flow $t\mapsto \bM(t)$ or a RMCF $\tau\mapsto\cM(\tau)$ over $I$, we define their entropy as \[
	\lambda[\bM]:= \sup_{t\in I} \lambda[\bM(t)]\,, \qquad \lambda[\cM]:= \sup_{t\in I} \lambda[\cM(t)] \,.
	\]
	
	An important feature of the Gaussian area is that it decays very fast at infinity. In fact, we will frequently use the following proposition, especially when $\Sigma=\cC_{n,k}$:
	
	\begin{proposition} \label{prop:entropy control outside ball of radius R}
		Suppose $\eps\in(0,1)$ and $\Sigma\subset\R^{n+1}$ is a hypersurface (or an $n$-varifold) with $\lambda[\Sigma]< +\infty$. Then for any $N>0$ and $R>0$,
		\[
		\Psi_{\Sigma,N}(R):= \int_{\Sigma\backslash B_R}(1+|X|^N)e^{-\frac{|X|^2}{4}}d\cH^n(X) \lesssim_{n, N, \delta} \lambda[\Sigma]\cdot e^{-\frac{R^2}{4+\delta}}.
		\]
	\end{proposition}
	
	\begin{proof}
		First recall that by \cite[Lemma 2.9]{ColdingMinicozzi12_generic}, for any $r>0$, \[
		\area(\Sigma\cap B_r(\orig))\lesssim_n \lambda[\Sigma]\cdot r^n
		\]
		Therefore, for any $R>0$, 
		\begin{align*}
			\int_{\Sigma\backslash B_R}(1+|X|^N)e^{-\frac{|X|^2}{4}}\ dX
			& = \sum_{j=1}^\infty \int_{\Sigma\cap(B_{R+j}\backslash B_{R+j-1})}(1+|X|^N)e^{-\frac{|X|^2}{4}}\ dX  \\
			& \lesssim_n \lambda[\Sigma]\cdot 
			\sum_{j=1}^\infty \left(1+(R+j)^{N+n} \right)\cdot e^{-\frac{(R+j-1)^2}{4}} \ \ \lesssim_{n, N, \delta} \lambda[\Sigma]\cdot e^{-\frac{R^2}{4+\delta}} \,.
		\end{align*}
		where the last inequality is a consequence of the fast decay of the exponential factor.
	\end{proof}
	
	Throughout this paper, unless otherwise noted, a ``mean curvature flow (MCF)" is referred to as a unit regular Brakke flow with finite entropy, and a ``rescaled mean curvature flow (RMCF)" means the RMCF of some mean curvature flow based at some point. Also for a Radon measure $\Sigma$ in $\R^{n+1}$, we may abuse the notation to use $\int_\Sigma\cdot$ instead of $\int\cdot\ d\Sigma$ to denote the integration with respect to $\Sigma$, and use $\Sigma\cap B$ instead of $\Sigma\llcorner B$ to denote the restriction of $\Sigma$ onto $B$.

	\subsection{Geometry of generalized cylinders} \label{Subsec_CylinderGeom}
	
	We use $\cC_{n,k}\subset\R^{n+1}$ to denote the round generalized cylinder $S^{n-k}(\varrho)\times\R^k=:\bS^{n-k}\times\R^k$, where $\varrho=\varrho_{n,k}:=\sqrt{2(n-k)}$; and we use $(\theta,y)=(\theta_1,\cdots,\theta_{n-k+1},y_1,\cdots,y_k)$ as the cylindrical coordinates over $\cC_{n,k}$, where $\theta_i$ is the restriction of the coordinate function $x_i$ of $\R^{n-k+1}$ to $\mathbb S^{n-k}$. And we use $\Rot(\cC_{n,k}):= \{\mbfB(\cC_{n,k}): \mbfB\in \rmO(n+1)\}$ to denote all possible rotations of $\cC_{n,k}$.
	
	\subsubsection{The rotation group and its subgroups} \label{Ssubsec_RotGp}
	The subgroup of $\rmO(n+1)$ that preserves $\cC_{n,k}$ is 
	\begin{align}
		\rmG_{n,k}:= \{\mbfB\in \rmO(n+1): \mbfB(\cC_{n,k}) = \cC_{n,k}\} = \rmO(n-k+1)\times \rmO(k)\,.  \label{Equ_Prelim_Stabilizer G_(n,k)}
	\end{align}
	We let $\mfk g_{n,k}$ be its Lie algebra, and $\mfk g^\perp_{n,k}$ be the orthogonal complement of $\mfk g_{n,k}$ with respect to the metric $\langle \mbfA_1, \mbfA_2\rangle:= \tr(\mbfA_1\mbfA_2^\top)$ on $\mfk{o}(n+1)$. Note that, 
	\begin{itemize}
		\item $\mfk g_{n,k}$ consists of $(n+1)\times (n+1)$ skew symmetric matrices of form: \[
		\mbfA = \begin{bmatrix}
			\mbfA_1 & 0 \\ 0 & \mbfA_2
		\end{bmatrix} \,,
		\]
		under the decomposition $\R^{n+1}=\R^{n-k+1}\oplus \R^k$.
		Therefore $\mfk g_{n,k}^\perp$ consists of $(n+1)\times (n+1)$ skew symmetric matrices of the following form under this decomposition: \[
		\mbfA = \begin{bmatrix}
			0 & \bl \\ -\bl^\top & 0
		\end{bmatrix} \,,
		\]
		where $\bl$ is an arbitrary $(n-k+1)\times k$ matrix, which can also be viewed as a linear map from $\{0\}\times\R^k$ to $\R^{n-k+1}\times \{0\}$.
		\item By implicit function theorem, there exists $\kappa_n\in (0, 1)$ such that \[
		\mfk g_{n,k}^\perp \oplus \mfk g_{n,k} \to \rmO(n+1)\,, \quad (\mbfA, \mbfA')\mapsto e^{\mbfA}e^{\mbfA'}
		\]  
		is a diffeomorphism from the $\kappa_n$-ball centered at $\orig$ on the domain to an open neighborhood of the identity matrix in $\rmO(n+1)$. By possibly shrinking $\kappa_n$, we can also require \[
		\mfk g_{n,k}^\perp \to \mathrm{Gr}_k(\R^{n+1})\,, \quad \mbfA \mapsto e^\mbfA(\{0\}\times \R^k)
		\]
		to be a diffeomorphism from the $\kappa_n$-ball centered at $\orig$ on the domain to an open neighborhood of $\{0\}\times \R^k$ in the $k$-Grassmannian $\mathrm{Gr}_k(\R^{n+1})$.
		Note that when $\mbfA'\in \mfk{g}_{n,k}$, $e^{\mbfA'}\in \rmG_{n,k}$ preserves $\cC_{n,k}$. Hence, to study the effect of rotations on the $L^2$-distance function to round cylinders, it suffices to work only in $\mfk{g}_{n,k}^\perp$. 
	\end{itemize}

	\subsubsection{The $L$-operator and the weighted $L^2$-space}
	It is natural to define the weighted Gaussian $L^2$-norm: for a hypersurface $M\subset \R^{n+1}$ and $f: M\to \R$, define
	\[
	\|f\|_{L^2(M)}^2:=(4\pi)^{-n/2}\int_{M}|f(X)|^2e^{-\frac{|X|^2}{4}}d\cH^n(X).
	\]
	Throughout this paper, without a specific notion, all the $L^2$-norms refer to this weighted one. When $M=\cC_{n,k}$, we sometimes replace the subscript $L^2(M)$ simply by $L^2$ to save notation.
	
	We are interested in the following linear operator on $\cC_{n,k}$, known as the \textit{Jacobi operator} for shrinker:
	\begin{equation}\label{EqL}
		L_{n,k}\, u = \Delta_{\cC_{n,k}} u -\frac{1}{2}\langle y,\nabla_y u\rangle + u 
		= \Delta_{\bS^{n-k}} u + (\cL_{\R^k} + 1)u\,,
	\end{equation}
	which is self-adjoint with respect to the $L^2$ norm above, where \[
	\cL_{\R^k}:= \Delta_{\R^{k}} - \frac{1}{2}\sum_{i=1}^{k}y_i\pr_{y_i}  \]
	is the Ornstein-Uhlenbeck operator on $\R^k$.

	\subsubsection{The spectrum of the $L$-operator} \label{SSubsec_Spectrum of L}
	Throughout this paper, we denote by $\sigma(\cC_{n,k})$ the collection of $L^2$-eigenvalues of $-L_{n,k}$ (i.e. $\lambda\in \sigma(\cC_{n,k})$ if and only if there exists a nonzero function $f\in L^2(\cC_{n,k})$ satisfying $-L_{n,k}\,f=\lambda f$). For $\lambda\in \sigma(\cC_{n,k})$, we denote by $\rmW_\lambda=\rmW_\lambda(\cC_{n,k})\subset L^2(\cC_{n,k})$ the space of ($L^2$-)eigenfunctions of $-L_{n,k}$ with eigenvalue $\lambda$. And for $\sim \in \{\geq, >, =, \neq, <, \leq\}$ and $\gamma\in \RR$, we denote by 
	\begin{align}
		\Pi_{\sim \gamma}: L^2(\cC_{n,k})\to L^2(\cC_{n,k})  \label{Equ_Pre_Proj onto sum of eigenspace}      
	\end{align}
	the ($L^2$-)orthogonal projection onto the direct sum $\oplus_{\lambda\sim \gamma}\rmW_\lambda$. A particularly important eigenvalue is $\gamma_{n,k}:=\frac{1}{n-k}$, which we will use in multiple places.
	
	In \cite[Section 5.2]{SunWangZhou20_MinmaxShrinker}, it was proved that the eigenvalues of $-L_{n,k}$ are given by
	\begin{align}
		\sigma(\cC_{n,k}) = \{\mu_i+j/2-1\}_{i,j=0}^\infty  \,, \label{Equ_Pre_sigma(C_n,k)}
	\end{align}
	with corresponding eigenspaces given by
	\begin{align}
		\rmW_\lambda(\cC_{n,k}) = \Span\big\{\phi_i(\theta) h_j(y): \mu_i+j/2-1 = \lambda \big\} \,, \label{Equ_Pre_W_lambda(C_n,k)}
	\end{align}
	where $\mu_i = \frac{i(i-1+n-k)}{2(n-k)}$ and $\phi_i$ are eigenvalues and eigenfunctions of $-\Delta_{\SSp^{n-k}}$ ($\phi_i$ are known to be the restriction of homogeneous harmonic polynomials on $\R^{n-k+1}$), and $h_j$ are Hermitian polynomials of degree $j$ on $\R^{k}$, which are eigenfunctions of $-\cL_{\R^k}$ with eigenvalue $j/2$. Note that, 
	\begin{itemize}
		\item degree $0$ Hermitian polynomials are constants; degree $1$ Hermitian polynomials are linear functions on $\R^k$;
		\item any degree $2$ Hermitian polynomials can be written as \[
		h_2(y) = q(y) - 2\tr q\,,
		\]
		where $q$ is a homogeneous quadratic polynomial on $\R^k$ and (letting $\{e_j\}_{j=1}^k$ be an orthonormal basis of $\R^k$) $\tr q:= \sum_{j=1}^k q(e_j)\in \R$.
	\end{itemize}
	
	For later reference, for $\mbfB\in \rmO(n+1)$ and $\cC=\mbfB(\cC_{n,k})\in \Rot(\cC_{n,k})$ we also let \[
	\rmW_\gamma(\cC) := \{\psi\circ \mbfB^{-1}: \psi\in \rmW_\gamma(\cC_{n,k}) \}\,.
	\]
	Note that this is well-defined since $\rmW_\gamma(\cC_{n,k})$ is $\rmG_{n,k}$-invariant.
	
	We remark that when $(n-k)\geq 2$, the smallest positive eigenvalue in $\sigma(\cC_{n,k})$ is $1/(n-k)\leq 1/2$. Therefore, the spectrum behaviors of $\cC_{n,n-1}$ and $\cC_{n,k}$ with $k\leq n-2$ are slightly different.
	In the following, we list the first few eigenspaces of $-L_{n,k}$ and their geometric meanings:
	\begin{enumerate} [label={\normalfont(\roman*)}]
		\item\label{Item_Pre_W_-1} $\rmW_{-1}=\{c:c\in \R\}$ consists of constant function;
		\item\label{Item_Pre_W_-1/2} $\rmW_{-1/2}$ consists of linear combinations of linear functions $h_1(y)$ on $\R^k$ and \textit{translation-like functions} $\psi_\bx(\theta)$ on $\SSp^{n-k}$, where for $\bx\in \R^{n-k+1}$, we define (set $\hat\theta:= \theta/|\theta|$)
		\begin{align}
			\psi_\bx(\theta):= \bx\cdot\hat\theta \,; \label{Equ_Pre_Transl-like eigen} 
		\end{align}
		For later reference, we denote by $\rmW_{\Tr}$ the space of translation-like functions on $\cC_{n,k}$.
		\item\label{Item_Pre_W_0} $\rmW_0$ consists of linear combinations of quadratic Hermitian polynomials $h_2(y)$ on $\R^k$ and \textit{rotation-like functions} $\psi_\mbfA(\theta, y)$ on $\cC_{n,k}$, where for $\mbfA = \begin{bmatrix}
			0 & \bl \\ -\bl^\top & 0
		\end{bmatrix}\in \mfk g_{n,k}^\perp$ we define,
		\begin{align}
			\psi_\mbfA(\theta, y):= \langle \mbfA(0, y), (\hat\theta, 0)\rangle = (\bl\, y)\cdot \hat\theta \,. \label{Equ_Pre_Rot-like eigen}
		\end{align}
		Again for later reference, we denote by $\rmW_{\Rot}$ the space of rotation-like functions on $\cC_{n,k}$.
		\item\label{Item_Pre_W_1/2} When $n-k\neq 2$, $\rmW_{1/2}$ consists of linear combinations of cubic Hermitian polynomials $h_3(y)$ on $\R^k$ and linear combinations of products of translation-like functions and quadratic Hermitian polynomials, which takes the form \[
		\sum_{i=1}^{n-k+1}\psi_{\mathbf{e}_i}(\theta)(q_i(y)-2\tr q_i) = \left(\sum_{i=1}^{n-k+1} \big(q_i(y)-2\tr q_i\big)\mathbf e_i \right)\cdot \hat\theta = (\bq(y) - 2\tr\bq)\cdot\hat\theta \,;
		\]
		where $\{\mathbf{e}_i\}_{i=1}^{n-k+1}$ is the coordinate orthonormal basis of $\R^{n-k+1}$, $q_i$ are homogeneous quadratic polynomials on $\R^k$ (possibly $0$), and \[
		\bq := \sum_{i=1}^{n-k+1} q_i\, \mathbf e_i :\{0\}\times \R^k \to \R^{n-k+1}\times \{0\}  \]
		is an $\R^{n-k+1}$-valued homogeneous quadratic polynomial on $\R^k$. Such functions play an essential role in the $C^2$-regularity of the $k$-cylindrical singular set. When $n-k=2$, $\rmW_{1/2}$ is the span of these functions discussed above together with the following \textit{low spherical modes}:
		\item\label{Item_Pre_W_S} Let $\rmW_{\SSp}\subset \rmW_{1/(n-k)}$ be the space of 2nd eigenfunctions of $-\Delta_{\SSp^{n-k}}$ (i.e. those with eigenvalue $1/(n-k)+1$), extended $y$-invariantly as functions on $\cC_{n,k}$. By the classification of harmonic polynomial, $\rmW_\SSp$ consists of functions of form \[
		\psi(\theta, y) = \sum_{i_1, i_2=1}^{n-k+1} a_{i_1 i_2}\theta_{i_1}\theta_{i_2}\,,
		\]
		where $\sum a_{ii} = 0$. These functions describe the slowest asymptotic in the spherical factor. In fact, each mode corresponds to some rescaled mean curvature flow converging to the sphere with the corresponding leading order asymptotics, see \cite{Sesum08_Rate_of_Convergence, Strehlke20_Asym_Max}, as well as Appendix \ref{App:ExistenceSphericalArrival} in this paper. These modes do not affect the $C^2$ regularity of the singularity set, but we need to modulo these modes because they may have lower eigenvalues.
	\end{enumerate}

	In summary, we have the following list of the first several eigenfunctions of $-L_{n,k}$.
	\begin{table}[H]
		\begin{tabular}{|l|l|}
			\hline
			eigenvalues of $-L_{n,k}$ & corresponding eigenfunctions (not necessarily normalized) \\ \hline
			$-1$ & $1$ \\ \hline
			$-1/2$ & $\theta_i,y_j$
			\\ \hline
			$0$ & $\theta_iy_j,\ y_j^2-2,\ y_{j_1}y_{j_2}$ 
			\\
			\hline
			$1/2$
			&
			$\theta_i(y_j^2-2)$,
			\ $\theta_i(y_{j_1}y_{j_2})$,
			\ $y_j^3-6y_j$, $\dots$
			\ 
			\\ 
			\hline
			$1/(n-k)$ & 
			$\theta_{i_1}^2-\theta_{i_2}^2$, $\dots$ 
			\\ \hline
		\end{tabular}
		\caption{Eigenvalues and eigenfunctions of $L_{n,k}$.}
		\label{TableEigen}
	\end{table}
	
	\subsubsection{Two facts on eigenfunctions}
	The following lemma will be used in Section \ref{Sec_C^(2,al) Reg} in getting bounds on Whitney data.
	\begin{Lem} \label{Lem_Facts on Hermitian Polyn} 
		Let $h$ be a Hermitian polynomial on $\R^k$ of degree $l\geq 2$. Then for every $a\in \R$ and every $\by\in \R^k$ with $|\by|\geq 1$, we have \[
		\|a + h(\cdot+\by) - h \|_{L^2} \geq c(k, l)|\by|^{1-l}|a|\,.     
		\]
	\end{Lem}
	\begin{proof}
		Because Hermitian polynomials after a rotation are still Hermitian polynomials, without loss of generality, we may assume $\by=(L,0,0,\cdots,0)$. Recall that a Hermitian polynomial $h$ with degree $l$ satisfies the equation 
		\[
		\Delta h-\frac{1}{2}y\cdot\nabla h+\frac{l}{2}h=0.
		\]
		This implies that 
		\begin{align*}
			I:=\int (h(y+\by)-h(y))d\mu
			=&
			\int \big(\frac{2}{l}\cL(h(y+\by)-h(y))-\frac{1}{l}\by\cdot \nabla h(y+\by)\big)d\mu
			\\
			=&
			-\int \frac{1}{l}\by\cdot \nabla h(y+\by) d\mu
			=
			-\int \frac{1}{l}L \pr_1 h(y+\by) d\mu
			.
		\end{align*}
		Note that the derivative of an Hermitian polynomial is still an Hermitian polynomial with eigenvalue shifted by $1/2$. Iterate the above argument $l$-times, we get
		\[
		I= \left(\int 1 d\mu\right) (-1)^{l} L^l c_{l,0,\cdots,0} \, ,
		\]
		where $c_{l,0,\cdots,0}$ is the coefficient of the monomial $y_1^{l}$ in the polynomial $h$.
		
		By projection to the constant, we have
		\[
		\|a+h(\cdot+\by)-h(\cdot)\|_{L^2}\geq \left|a\left(\int 1d\mu\right)+I\right|. 
		\]
		If $\left(\int 1 d\mu\right)^{-1}|I|<a/2$, then $\|a+h(\cdot+\by)-h(\cdot)\|_{L^2}\geq |a|/2\geq |\by|^{1-l}|a|/2$.
		
		If $\left(\int 1 d\mu\right)^{-1}|I|\geq a/2$, then at least one of $|L^{l}c_{l,0,\cdots,0}|\geq C|a|$. Let $p$ be the unique Hermitian polynomial with degree $l-1$ and leading order monomial $y_1^{l-1}$. Then
		\begin{align*}
			\int(a+h(y+\by)-h(y))p(y) d\mu
			=&
			\int(h(y+\by)-h(y))p(y)d\mu.
		\end{align*}
		Note that $h(y+\by)-h(y)$ has degree $l-1$, and it has the monomial term $y_1^{l-1}$ with coefficient $l L c_{l,0,\cdots,0}$, so the projection of $h(y+\by)-h(y)$ to $p(y)$ has length greater than or equal to $C|l L c_{l,0,\cdots,0}|$. This shows that 
		\[
		\|a+h(\cdot+\by)-h(\cdot)\|_{L^2}\geq C|l L c_{l,0,\cdots,0}|\geq CL^{1-l}|a|=C|\by|^{1-l}|a|.
		\]
		In the above discussion, the constant $C$ varies line by line but is uniformly bounded by $l$ and $k$.
	\end{proof}
	
	The following lemma estimates the difference of the $L^2$ norm after a cut-off.
	\begin{Lem} \label{Lem_Effect of Cutoff} 
		Let $\zeta\in L^\infty(\cC_{n,k})$ be such that $0\leq \zeta\leq 1$ and $\zeta|_{Q_R} = 1$; $\Pi: L^2(\cC_{n,k})\to L^2(\cC_{n,k})$ be the orthogonal projection map onto a closed linear subspace. Then for every $v\in L^2(\cC_{n,k})$, we have \[
		\big|\|\Pi(v\zeta)\|_{L^2} - \|\Pi(v)\|_{L^2}\big| \leq \|v\|_{L^2(\cC_{n,k}\setminus Q_R)} 
		\]
		In particular, if $\Lambda>0$, $v\in \oplus_{\gamma\leq \Lambda} \rmW_\gamma$, then \[
		\big|\|\Pi(v\zeta)\|_{L^2} - \|\Pi(v)\|_{L^2}\big| \lesssim_{n, \Lambda} e^{-R^2/10}\|v\|_{L^2}\,.
		\]
	\end{Lem}
	
	\begin{proof}
		By the triangle inequality, $$\|\Pi(v\zeta)\|_{L^2} - \|\Pi(v)\|_{L^2}\leq \|\Pi(v(1-\zeta))\|_{L^2}\leq \|v(1-\zeta)\|_{L^2}\leq \|v\|_{L^2(\cC_{n,k}\setminus Q_R)}.$$ Similarly, $$ \|\Pi(v)\|_{L^2}-\|\Pi(v\zeta)\|_{L^2}\leq \|\Pi(v(1-\zeta))\|_{L^2}\leq \|v(1-\zeta)\|_{L^2}\leq \|v\|_{L^2(\cC_{n,k}\setminus Q_R)}.$$ Thus, we obtain the first inequality.
		
		By the spectrum decomposition of $L_{n,k}$, $v\in \oplus_{\gamma\leq \Lambda} \rmW_\gamma$ implies that, \[
		|v(\theta,y)|\lesssim_{n,\Lambda} \|v\|_{L^2}\cdot (1+|y|^{2\Lambda+2}).
		\]
		Therefore the second inequality follows from the first and \[
		\|v\|_{L^2(\cC_{n,k}\setminus Q_R)}^2
		\lesssim_{n,\Lambda} \|v\|_{L^2}^2\cdot \int_{\cC_{n,k}\backslash Q_R}(1+|y|^{2\Lambda+2}) e^{-\frac{|X|^2}{4}}dX 
		\lesssim_{n,\Lambda} \|v\|_{L^2}^2\cdot e^{-R^2/5},
		\]
		where the last inequality uses Proposition \ref{prop:entropy control outside ball of radius R} with $\delta=1$.
	\end{proof}
	
	\subsection{Use of constants}
	Throughout this paper, we shall use the letter $C$ to denote a constant that is allowed to vary from line to line (or even within the same line); we shall stress the functional dependence of any such constant on geometric quantities by including them in brackets, writing things like $C=C(n,\eps)$.
	We shall also use $\Psi(\eps|C_1, C_2, \dots, C_l)$ to denote a constant depending on $\eps, C_1, \dots, C_l$ and tending to $0$ when $C_1, \dots, C_l$ are fixed and $\eps\to 0$.
	To save notation, we also write $A\sim_{a, b, \dots} A'$ (resp. $A\lesssim_{a, b, \dots} A'$, $A\gtrsim_{a, b, \dots} A'$) if for some constant $C(a, b, \dots)>1$, \[
	C(a, b, \dots)^{-1}A'\leq A\leq C(a, b, \dots)A'\,; \qquad (\,\text{resp. }\  A\leq C(a, b, \dots) A'; \quad A\geq C(a, b, \dots)^{-1}A'.\,)
	\]

	\section{An $L^2$-nonconcentration for a general convex cylinder}  \label{Sec_L^2 Noncon}
	In \cite{SunWangXue1_Passing}, the authors proved a general monotonicity for the $L^2$-distance to a certain level set flow. As an application, an $L^2$-nonconcentration for RMCF relative to the round cylinder was obtained in \cite[Corollary 3.3]{SunWangXue1_Passing}. In this section, we generalize such a nonconcentration property relative to a general convex cylindrical flow. The consequence of it, which will be discussed in this section, is almost the same as those discussed in \cite[Section 3]{SunWangXue1_Passing}.
	
	Let $n\geq 2$, $1\leq k\leq n-1$ be integers. We shall fix the following data $(\Omega, \bu, \beta, \rho)$ such that, 
	\begin{enumerate}[label=$\mathbf{(T{{\arabic*}})}$]
		\item \label{Item_ArrivalT1} $\Omega= \Omega_\circ\times \RR^k$, where $\Omega_\circ\subset \RR^{n-k+1}$ is a smooth strictly convex subset;
		\item \label{Item_ArrivalT2} $\bu(x, y):= \bu_\circ(x): \Omega\to \RR$, where $\bu_\circ: \Omega_\circ\to \RR_{\leq 0}$ is the arrival-time-function of the MCF starting from $\partial\Omega_\circ$; $\beta\in (0, 1]$ such that
		\begin{align*}
			\bu_\circ(\orig)=0\,, & & \bu_\circ|_{\partial\Omega_\circ}=-4\,, & & \beta I_{n-k+1} \leq -(n-k)\cdot\nabla^2 \bu_\circ \leq \beta^{-1}I_{n-k+1} \,.
		\end{align*}
		\item \label{Item_ArrivalT3} $\rho(X, t) := (-4\pi t)^{-n/2}\exp(|X|^2/4t)$. We denote for later reference $d\mu:= \rho(X, -1)dX$.
	\end{enumerate}
	
	\begin{Rem} \label{Rem_L^2 Noncon_Scal Invar of (T2)(T3)}
		If $(\Omega, \bu, \beta)$ satisfies \ref{Item_ArrivalT1} and \ref{Item_ArrivalT2} above, then one can directly check that for every $\lambda\geq 1$, $(\{\bu^{(\lambda)}\geq -4\}, \bu^{(\lambda)}, \beta)$ also satisfies \ref{Item_ArrivalT1} and \ref{Item_ArrivalT2}, where \[
		\bu^{(\lambda)}(X):= \lambda^2 \bu(\lambda^{-1}X) \,.
		\]  
		Also, the RMCF associated to $\bu$ is 
		\begin{align}
			\tau\mapsto \cC_\bu(\tau) := e^{\frac{\tau}{2}}\{\bu=-e^{-\tau}\}\,.  \label{Equ_L^2 Noncon_RMCF associated to u}
		\end{align} 
		Define the graphical function of $\cC_\bu(\tau)$ over $\cC_{n,k}$ as $\varphi_\bu(\cdot, \tau)$. Then it's easy to check that \[
		\varphi_{\bu^{(\lambda)}}(\cdot, \tau) = \varphi_\bu(\cdot, \tau+2\ln\lambda) \,.
		\]
		Also, assumption \ref{Item_ArrivalT2} implies that as $\tau\to +\infty$, $\cC_\bu(\tau)$ converges to $\cC_{n,k}$. Then by Brakke-White's $\eps$-regularity \cite{White05_MCFReg}, $\|\varphi_\bu(\cdot, \tau)\|_{C^4}\to 0$ as $\tau\to +\infty$.
	\end{Rem}
	
	\begin{Rem}
		One explicit example that has already been discussed in \cite[Section 3.2]{SunWangXue1_Passing}, is $\Omega_\circ = B^{n-k+1}(0,\sqrt{8(n-k)})$ and   
		\begin{align}
			\bu(x, y) = \mbfU_{n,k}(x, y) := -\frac{|x|^2}{2(n-k)} \,, \label{Equ_L^2 Noncon_Arriv time func for k-cylind} 
		\end{align}
		which describes the generalized shrinking cylinder: \[
		\mbfC_{n, k} := \coprod_{t\leq 0}\{(x, y)\in \RR^{n+1} :\mbfU_{n, k}(x, y)=t\}\times \{t\} = \coprod_{t\leq 0} (\sqrt{-t}\,\cC_{n, k})\times \{t\} \subset \RR^{n+1}\times \RR\,.
		\]
		In this case, one can choose $\beta=1$. Another family of examples are the ``low spherical flows", to be constructed and discussed in Section \ref{Subsec_Mod Low SphericalMode}. 
	\end{Rem}
	
	\subsection{Application of $L^2$-monotonicity formula}\label{SS:Application of $L^2$-monotonicity formula}
	
	Let us fix a non-decreasing odd function $\chi\in C^\infty(\RR)$ such that 
	\begin{itemize}
		\item $\chi''\leq 0$ on $[0, +\infty)$;
		\item $\chi(s) = s$ for $|s|\leq 1/2$, $\chi(s) = \sgn(s)$ for $|s|\geq 1$.
	\end{itemize}
	Suppose $(\Omega, \bu, \beta, \rho)$ satisfies \ref{Item_ArrivalT1}-\ref{Item_ArrivalT3}, and in addition
	\begin{align}
		\|\varphi_\bu(\cdot, \tau)\|_{C^0}\leq \frac14, \qquad \forall\, \tau\geq 0 \,. \label{Equ_L^2 Mono_Assump |varphi_u|<1/2}
	\end{align}
	We define
	\begin{align}
		\odist_\bu(X, \tau) := \chi\left(|x|-\varrho - \varphi_\bu(\varrho|x|^{-1}x) \right) \,, \label{Equ_L^2 Mono_odist to C_(n,k)}      
	\end{align}
	which is a cut-off and regularization of the radial signed distance function to $\cC_\bu$, and define
	\begin{align}
		\rmD_\bu(\tilde X, \tau) = \rmD_\bu(\tilde x,\tilde y, \tau) := f_\circ\left( \bu_\circ(e^{-\tau/2}x) + e^{-\tau}\right) \,. \label{Equ_L^2 Mono_rmD_bu def}      
	\end{align}
	where $f_\circ$ is specified in \cite[Section 3, \textbf{(S2)}]{SunWangXue1_Passing} with $T_\circ=1$.
	Then it's easy to check that for every $(\tilde X, \tau)\in \RR^{n+1}\times \RR_{\geq 0}$,
	\begin{align}
		e^{-2\tau} \overline{\dist}_\bu(\tilde X, \tau)^2 
		\lesssim_{n} \rmD_\bu(\tilde X,\tau)^2 
		\lesssim_{n} \overline{\dist}_\bu(\tilde X, \tau)^2 \,.  \label{Equ_L^2 Mono_D_bu approx dist_bu} 
	\end{align}
	This leads to the following non-concentration near infinity for rescaled mean curvature flow.

	\begin{Lem}\label{Lem_L^2 Noncon_for RMCF}
		Let $\Omega, \bu, \beta, \rho$ be satisfying \ref{Item_ArrivalT1}-\ref{Item_ArrivalT3} and \eqref{Equ_L^2 Mono_Assump |varphi_u|<1/2}. Then there exist $K=K(\beta, n)>2$ with the following property. Let $T_\pm\geq 0$, $[T_-, T_+]\ni\tau\to\cM(\tau)$ be a RMCF with finite entropy in $\RR^{n+1}$. Then for every $T_-<\tau \leq T_+$, 
		\begin{align*}
			\int_{\cM(\tau)} \overline{\dist}_\bu(X, \tau)^2 (1 + (\tau-T_-)|X|^2) \ d\mu \lesssim_{n, \beta}\ e^{K(\tau-T_-)}\int_{\cM(T_-)} \overline{\dist}_\bu(X, \tau)^2 \ d\mu \,.
		\end{align*}   
	\end{Lem}
	\begin{proof}
		By Remark \ref{Rem_L^2 Noncon_Scal Invar of (T2)(T3)}, we can assume WLOG that $T_-=0$.  If we let $\eta\in C^2(\RR^{n+1}\times \RR)$ be a non-negative function such that $\{\eta>0\} = \RR^{n+1}\times I$ for some interval $I$, $T_\circ=1$ and $f_\circ\in C^\infty(\RR)$ be chosen as in \cite[\textbf{(S2)}]{SunWangXue1_Passing}, then it's easy to check that \textbf{(S1)}-\textbf{(S4)} in \cite[Section 3]{SunWangXue1_Passing} holds for $(\eta, \Omega, \bu, \beta', \rho)$ with $\beta'>0$ being some constant depending only on $n, \beta$. 
		
		Let $t\mapsto \bM(t):= \sqrt{-t}\cM(-\ln(-t))$ be the MCF associated to $\cM$. When taking $\eta\equiv 1$, \cite[(3.4)]{SunWangXue1_Passing} applied to $\bM$ becomes
		\begin{align}
			\frac{d}{dt} \left[ (-t)^{2K}\int_{\bM(t)} F^2\rho\ dX \right] + (-t)^{2K}\int_{\bM(t)} \frac{c(n, \beta)|X|^2}{t^2} F^2 \rho\ dX \leq 0 \label{Equ_L^2 Noncon_Model Eg with eta = 1}
		\end{align}
		
		When taking $\eta(X, t) = \sqrt{1-|X|^2/t\, } \xi(t)$, where $\xi\in C_c^2(-1, 0)$, \cite[(3.4)]{SunWangXue1_Passing} implies,
		\begin{align}
			\begin{split}
				\frac{d}{dt} & \left[  (-t)^{2K}\int_{\bM(t)} F^2\cdot(1-|X|^2/t)\,\xi(t)^2\rho\ dX \right] \\ 
				\leq \; & (-t)^{2K}\int_{\bM(t)} \left[\left(-\frac{|X|^2}{t^2} - \frac{2n+8}{t} \right)\xi^2 + \left(1-\frac{|X|^2}{t}\right)|\partial_t(\xi^2)| \right] F^2\rho\ dX
			\end{split} \label{Equ_L^2 Noncon_Model Eg with eta = |X|}
		\end{align}
		
		Now we rewrite everything under RMCF $\tau\mapsto \cM(\tau)$ parametrized by $(\tilde X, \tau)$ using the change of variable $(X, t) = (e^{-\tau/2}\tilde X, -e^{-\tau})$.  Recall $\rmD_\bu$ is defined in \eqref{Equ_L^2 Mono_rmD_bu def}. Then \eqref{Equ_L^2 Noncon_Model Eg with eta = 1} is equivalent to, 
		\begin{align}
			\begin{split}
				\frac{d}{d\tau} \left[e^{-2K\tau}\int_{\cM(\tau)} \rmD_\bu(X,\tau)^2 \ d\mu \right] 
				+ C(n,\beta)e^{-2K\tau}\int_{\cM(\tau)} \rmD_\bu(X,\tau)^2|X|^2 \ d\mu \leq 0 \,.
			\end{split} \label{Equ_L^2 Noncon_Mono under RMCF param w eta=1}
		\end{align}
		And let $\tilde{\xi}(\tau):= \xi(-e^{-\tau})$, then (\ref{Equ_L^2 Noncon_Model Eg with eta = |X|}) is equivalent to,
		\begin{align}
			\begin{split}
				\frac{d}{d\tau} & \left[ \, e^{-2K\tau}\int_{\cM(\tau)} \rmD_\bu(X,\tau)^2(1+|X|^2)\tilde{\xi}(\tau)^2 \ d\mu \right] \\
				\leq & \; e^{-2K\tau}\int_{\cM(\tau)} \rmD_\bu(X,\tau)^2\left[ (2n+8-|X|^2)\tilde{\xi}(\tau)^2 + (1+|X|^2)
				\left|(\tilde{\xi}(\tau)^2)'\right| \right] \ d\mu \,.
			\end{split} \label{Equ_L^2 Noncon_Mono under RMCF param w eta=|X|}
		\end{align} 
		
		Integrate \eqref{Equ_L^2 Noncon_Mono under RMCF param w eta=1} and \eqref{Equ_L^2 Noncon_Mono under RMCF param w eta=|X|} with $\tilde{\xi}(s)=s/\tau$ and use \eqref{Equ_L^2 Mono_D_bu approx dist_bu} proves the Lemma.
	\end{proof}

	\subsection{Low spherical flows} \label{Subsec_Mod Low SphericalMode}
	We shall construct a new family of arrival time functions satisfying \ref{Item_ArrivalT1}, \ref{Item_ArrivalT2} and \eqref{Equ_L^2 Mono_Assump |varphi_u|<1/2}.
	Recall that $\rmW_{\SSp}$ is specified in \ref{Item_Pre_W_S} of Section \ref{SSubsec_Spectrum of L}, and $\gamma_{n,k}=1/(n-k)$.
	\begin{Lem} \label{Lem_L^2Noncon_LowSphericalMode}
		Suppose $(n-k)\geq 2$. There exists $\kappa_{\ref{Lem_L^2Noncon_LowSphericalMode}}(n)\in (0, 1/4)$, a subset $\scU$ of the space of convex $C^2$-arrival time functions on some domain in $\R^{n+1}$ and a bijection map from $\kappa_{\ref{Lem_L^2Noncon_LowSphericalMode}}(n)$-ball in $\rmW_\SSp$ to $\scU$, such that, 
		\begin{enumerate} [label={\normalfont(\arabic*)}]
			\item\label{item_0_Lem_L^2Noncon_LowSphericalMode}  for every $\mbfB\in \rmG_{n,k}$ (the stabilizer of $\cC_{n,k}$, see its definition in \eqref{Equ_Prelim_Stabilizer G_(n,k)}) and every $\psi\in \rmW_\SSp$ with image $\bu\in \scU$, we have $\bu\circ \mbfB \in \scU$ is the image of $\psi\circ\mbfB$; 
			\item\label{item_1_Lem_L^2Noncon_LowSphericalMode} there exists $\beta=\beta(n)>0$ such that every $\bu\in \scU$ satisfies the conditions \ref{Item_ArrivalT1} \& \ref{Item_ArrivalT2};
			\item\label{item_2_Lem_L^2Noncon_LowSphericalMode} the image of $0\in\rmW_{\SSp}$ is just the arrival time $\mbfU_{n,k}$ of the round cylinder; for any images $\bu_1,\bu_2\in \scU$ of $\psi_1,\psi_2\in \rmW_\SSp$ respectively and any $\tau\geq 0$, \[
			\left\|(\varphi_{\bu_1}(\cdot, \tau)-e^{-\gamma_{n,k}\tau}\psi_1)-(\varphi_{\bu_2}(\cdot, \tau)-e^{-\gamma_{n,k}\tau}\psi_2)\right\|_{C^4(\cC_{n,k})}\leq \frac12e^{-\frac{3}{2}\gamma_{n,k}\tau}\|\psi_1-\psi_2\|_{L^2(\cC_{n,k})} \,.
			\] 
			In particular, $\|\varphi_\bu(\cdot, \tau)\|_{C^4}\leq \frac14 e^{-\gamma_{n,k} \tau}$ for every $\bu\in \scU$. 
		\end{enumerate}
	\end{Lem}
	
	The key part of the proof is to construct convex RMCF with given asymptotics on the sphere. This is done is Appendix \ref{App:ExistenceSphericalArrival}. 
	
	\begin{definition} \label{Def_LowSphericalFlow}
		We define $\scU$ to be a specific family of arrival time functions as: 
		\begin{itemize}
			\item $\scU := \{\mbfU_{n,k}\}$, if $n-k=1$;
			\item $\scU$ is given by Lemma \ref{Lem_L^2Noncon_LowSphericalMode}, if $n-k\geq 2$.
		\end{itemize}
		Elements in $\scU$, after a possible parabolic dilation, are called \textbf{low spherical flows} asymptotic to $\cC_{n,k}$. 
		We also define
		\begin{align}
			\tilde\scU := \{\bu\circ \mbfB: \bu\in\scU,\ \mbfB\in \rmO(n+1) \}\,.  \label{Def_tilde scU}  
		\end{align}
		Note that by Lemma \ref{Lem_L^2Noncon_LowSphericalMode} \ref{item_0_Lem_L^2Noncon_LowSphericalMode}, $\bu\circ\mbfB\in \scU$ if and only if $\mbfB\in \rmG_{n,k}$. 
	\end{definition}
	
	\begin{Cor} \label{Cor_L^2 Mono_Low Bd Upsilon_(u,u'; w)}
		For every $\gamma\geq -1$, there exists $\Lambda_{\ref{Cor_L^2 Mono_Low Bd Upsilon_(u,u'; w)}}(n, \gamma)>2n$ such that if $\bu, \bu'\in \scU$ are low spherical flows and $w\in (\oplus_{\lambda\leq \gamma} \rmW_\lambda(\cC_{n,k}))\cap \rmW_\SSp^\perp$ is a linear combination of eigenfunctions of $-L_{n,k}$. Then for every $R\geq \Lambda(n, \gamma)$ and every $\tau\geq 0$, we have \[
		\|(\varphi_\bu - \varphi_{\bu'})(\cdot, \tau) + w\|_{L^2(Q_R)} \gtrsim_n \|(\varphi_\bu - \varphi_{\bu'})(\cdot, \tau)\|_{L^2} + \|w\|_{L^2} \,.
		\]
	\end{Cor}
	\begin{proof}
		We assume the leading eigenmodes of $\varphi_\bu$ and $\varphi_{\bu'}$ are given by $\psi$ and $\psi'$ respectively. By the item \ref{item_2_Lem_L^2Noncon_LowSphericalMode} of Lemma \ref{Lem_L^2Noncon_LowSphericalMode}, 
		\begin{align*}
			\|(\varphi_\bu - \varphi_{\bu'})(\cdot, \tau) + w\|_{L^2(Q_R)} & \geq  \|e^{-\gamma_{n,k}\tau}(\psi-\psi')+ w\|_{L^2(Q_R)}- \frac12\|e^{-\frac{3}{2}\gamma_{n,k}\tau}(\psi-\psi')\|_{L^2},  \\
			\|e^{-\gamma_{n,k}\tau}(\psi-\psi')\|_{L^2}& \sim_n  \|(\varphi_\bu - \varphi_{\bu'})(\cdot, \tau)\|_{L^2(Q_R)}.
		\end{align*}
		By the orthogonality, \[
		\|e^{-\gamma_{n,k}\tau}(\psi-\psi')+ w\|_{L^2}^2 = e^{-2\gamma_{n,k}\tau} \|\psi-\psi'\|_{L^2}^2 + \|w\|_{L^2}^2 \sim_n \|(\varphi_\bu - \varphi_{\bu'})(\cdot, \tau)\|_{L^2}^2+ \|w\|_{L^2}^2
		\] 
		Thus, it suffices to show that for $R\geq \Lambda(n,\gamma)$, \[
		\|e^{-\gamma_{n,k}\tau}(\psi-\psi')+ w\|_{L^2(Q_R)} \geq \frac34\|e^{-\gamma_{n,k}\tau}(\psi-\psi')+ w\|_{L^2}\,.  \]

		By the expression of the eigenfunctions in $\rmW_{\lambda}(\cC_{n,k})$ in Section \ref{SSubsec_Spectrum of L}, 
		for any unit eigenfunction $\phi$ with eigenvalue $\lambda$, we have $|\phi(\theta,y)|\leq C_\lambda(1+|y|^{2|\lambda|+2})$. Therefore, if $v=\sum_{i=1}^M a_i\phi_i\in (\oplus_{\lambda\leq \gamma} \rmW_\lambda(\cC_{n,k}))$ is a Fourier decomposition of the eigenfunctions with eigenvalues $\leq \gamma$, we have
		\[
		|v(\theta,y)|^2\lesssim_M\sum_{i=1}^M |a_i|^2|\phi_i(\theta,y)|^2\lesssim_{\gamma} \sum_{i=1}^M |a_i|^2 (1+|y|^{2|\gamma|+2})\|\phi_i\|^2_{L^2}
		\sim_{\gamma}(1+|y|^{2|\gamma|+2})\|v\|^2_{L^2}.
		\]
		By taking $v=e^{-\gamma_{n,k}\tau}(\psi-\psi')+ w$ and integrating $|v(\theta,y)|^2$ outside $Q_R$, using Proposition \ref{prop:entropy control outside ball of radius R},
		\[
		\|v\|_{L^2(\cC_{n,k}\backslash Q_R)}\leq \Psi_{\cC_{n,k},2|\gamma|+2}(R)^{1/2} \|v\|_{L^2}.
		\]
		When $R$ is sufficiently large, this implies that $\|v\|_{L^2(\cC_{n,k}\backslash Q_R)}\leq 1/4\|v\|_{L^2}$, which implies that $\|v\|_{L^2(Q_R)}\geq 3/4\|v\|_{L^2}$. This completes the proof.
	\end{proof}

	\subsection{Decay order and asymptotic rate} \label{Subsec_Doub Const, AR of RMCF}  
	
	Recall that by Lemma \ref{Lem_L^2Noncon_LowSphericalMode}, every low spherical flow $\bu\in \scU$ satisfies the conditions \ref{Item_ArrivalT1}, \ref{Item_ArrivalT2} and \eqref{Equ_L^2 Mono_Assump |varphi_u|<1/2} for some $\beta=\beta(n)>0$. In particular, the $L^2$ nonconcentration Lemma \ref{Lem_L^2 Noncon_for RMCF} applies to $\odist_\bu$. We shall discuss consequences of this $L^2$ nonconcentration property in the rest of this section analogous to those in \cite[Section 3]{SunWangXue1_Passing}. Although most lemma in this subsection work for general arrival time function satisfying \ref{Item_ArrivalT1}, \ref{Item_ArrivalT2} and \eqref{Equ_L^2 Mono_Assump |varphi_u|<1/2}, to avoid technicality, we restrict ourselves to the low spherical flows $\scU$.
	
	For every $\bu\in \scU$, let $\cC_\bu:= \{\cC_\bu(\tau)\}_{\tau\geq 0}$ be the RMCF associated to $\bu$ as in Remark \ref{Rem_L^2 Noncon_Scal Invar of (T2)(T3)}. For a hypersurface $\Sigma\subset \RR^{n+1}$, we define the $L^2$-distance from $\Sigma$ to $\cC_\bu(\tau)$ by,
	\begin{align}
		\mbfd_\bu(\tau, \Sigma)^2 := \int_{\Sigma}  \overline{\dist}_\bu(X, \tau)^2\ d\mu\,;  \label{Equ_L^2 Noncon_Def L^2 dist to C_u(tau)}
	\end{align} 
	
	For later reference, we also introduce the distance to the rotated relevant flow. For $\tilde\bu = \bu\circ\mbfB$ with $\bu\in \scU$ and $\mbfB\in \rmO(n+1)$, define 
	\begin{align}
		\mbfd_{\tilde\bu}(\tau, \Sigma)^2 := \int_{\Sigma}  \overline{\dist}_\bu(\mbfB(X), \tau)^2\ d\mu\,;  \label{Equ_L^2 Noncon_Def L^2 dist to C_tilde(u)(tau)}
	\end{align} 
	Note that this is well-defined, since whenever $\rmB\in \rmO(n+1)$ such that $\bu\circ \mbfB = \bu$, we must have $\mbfB(\cC_{n,k})=\cC_{n,k}$ and $\overline{\dist}_\bu(\mbfB(X), \tau)^2 = \overline{\dist}_\bu(X, \tau)^2$ for every $X\in \RR^{n+1}$ and $\tau\in \RR_{\geq 0}$.

	The following Lemma comparing $L^2$-distance from translated hypersurfaces to possibly a different relative flow will be useful later.
	\begin{Lem} \label{Lem_L^2 Mono_Compare d_u(Sigma) w d_u'(Sigma) and d_u(Sigma + y)}
		Let $\Sigma\subset \R^{n+1}$ be a hypersurface with $\lambda[\Sigma]<+\infty$, $\bu, \bu'\in \scU$. Then for every $\tau\geq 0$, we have
		\begin{align}
			|\mbfd_{\bu'}(\tau, \Sigma) - \mbfd_{\bu}(\tau, \Sigma)| \leq \lambda[\Sigma]^{1/2}\|\chi\|_{C^1}\cdot \|(\varphi_{\bu'}-\varphi_\bu)(\cdot, \tau)\|_{C^0}\,.  \label{Equ_L^2 Mono_|d_u' - d_u| < |varphi_u'-varphi_u|}
		\end{align}
		Also for every $\by\in \R^k$ with $|\by|\leq 1$, and every $R\geq 1$, we have, 
		\begin{align}
			\mbfd_\bu(\tau, \Sigma + (0, \by)) \lesssim_n \mbfd_\bu(\tau, \Sigma)\cdot e^{R/4} + \lambda[\Sigma]^{1/2}\cdot e^{-R^2/64} \,. \label{Equ_L^2 Mono_d_u(Sigma - y) < d_u(Sigma) + error}
		\end{align}
	\end{Lem}
	\begin{proof}
		\eqref{Equ_L^2 Mono_|d_u' - d_u| < |varphi_u'-varphi_u|} follows by the triangle inequality of Gaussian $L^2$ norm on $\cM$ and the pointwise estimate below from definition in \eqref{Equ_L^2 Mono_odist to C_(n,k)}, \[
		\left|\odist_{\bu}(X, \tau) - \odist_{\bu'}(X, \tau)\right| \leq \|\chi\|_{C^1}\cdot\|(\varphi_\bu-\varphi_{\bu'})(\cdot, \tau)\|_{C^0} \,.
		\]
		
		To prove \eqref{Equ_L^2 Mono_d_u(Sigma - y) < d_u(Sigma) + error}, note that for $|\by|\leq 1$ and every $X\in \R^{n+1}$, 
		\begin{align}
			|X+(0,\by)|^2 \geq |X|^2 - 2|X| + 1 \geq \frac14|X|^2 - 1\,. \label{Equ_L^2 Mono_Alg ineq}
		\end{align}
		Hence, 
		\begin{align*}
			\mbfd_\bu(\tau, \Sigma+(0, \by))^2 & \sim_n \int_\Sigma \odist_\bu(X+(0, \by), \tau)^2\cdot e^{-|X+(0,\by)|^2/4}\ dX \\
			& \lesssim_n \int_{\Sigma\cap Q_R} \odist_\bu(X, \tau)^2\cdot e^{R/2}\cdot e^{-|X|^2/4}\ dX + \int_{\Sigma\setminus Q_R} \odist_\bu(X, \tau)^2\cdot e^{-|X|^2/16}\ dX \\
			& \lesssim_n \mbfd_\bu(\tau, \Sigma)^2\cdot e^{R/2} + \int_{2^{-1}\cdot (\Sigma\setminus Q_R)} e^{-|X|^2/4}\ dX \\
			& \lesssim_n \mbfd_\bu(\tau, \Sigma)^2\cdot e^{R/2} + \lambda[\Sigma]\cdot e^{-R^2/32} \,;
		\end{align*}
		where the second inequality follows from the translation invariance of $\mbfd_\bu(\cdot, \tau)$ in $\R^k$ direction; the third inequality follows from change of variable and that $|\odist_\bu(X, \tau)|\leq 1$, and the last inequality follows from Proposition \ref{prop:entropy control outside ball of radius R}.
		This finishes the proof of \eqref{Equ_L^2 Mono_d_u(Sigma - y) < d_u(Sigma) + error}.
	\end{proof}

	For $\bu\in\tilde \scU$ and a RMCF $\tau\to\cM(\tau)$ with $\tau, \tau+1\in I$, we also define the \textbf{decay order} of $\cM$ at time $\tau$ relative to $\cC_\bu$ by 
	\begin{align}
		\cN_\bu(\tau, \cM) := \ln \left(\frac{\mbfd_\bu(\tau, \cM(\tau))}{\mbfd_\bu(\tau+1, \cM(\tau+1))} \right) \label{Equ_L^2 Noncon_Def Doubl Const}
	\end{align}
	Note that when $\bu=\mbfU_{n,k}$, using the notation in \cite[Section 3.3]{SunWangXue1_Passing}, 
	\begin{align*}
		\mbfd_{n,k}(\cM(\tau)) = \mbfd_{\mbfU_{n,k}}(\tau, \cM(\tau))\,, & &
		\cN_{n,k}(\tau, \cM) = \cN_{\mbfU_{n,k}}(\tau, \cM)\,.
	\end{align*}
	Also note that by Lemma \ref{Lem_L^2 Noncon_for RMCF}, for every RMCF $\cM$, we always have the uniform lower bound 
	\begin{align}
		\cN_\bu(\tau, \cM) \geq -C(n) \,. \label{Equ_L^2 Noncon_Dim Lower Bd for cN}
	\end{align}
	
	For $R>0$, recall that $Q_R:= B_R^{n-k+1}\times B_R^k$. We also define the decay order of a RMCF $\cM$ restricted in $Q_R$: if $\cM(\tau)\cap Q_R, \cM(\tau+1)\cap Q_R \neq \emptyset$, let 
	\begin{align*}
		\cN_\bu(\tau, \cM\cap Q_R) := \ln \left(\frac{\mbfd_\bu(\tau, \cM(\tau)\cap Q_R)}{\mbfd_\bu(\tau+1, \cM(\tau+1)\cap Q_R)} \right) 
	\end{align*}  
	Similar to \cite[Corollary 3.4]{SunWangXue1_Passing}, we have the following error estimate on this localized decay order.
	\begin{Cor}\label{Cor_L^2 Noncon_Compare d(M), cN(M) and d(M in Q_R), cN(M in Q_R)}
		For every $\eps\in (0, 1)$, there exists $R_{\ref{Cor_L^2 Noncon_Compare d(M), cN(M) and d(M in Q_R), cN(M in Q_R)}}(\eps, n)\gg 1$ such that the following hold.  If $T_\circ\geq 0$, $\bu\in \scU$, $\cM$ is a RMCF in $\RR^{n+1}$ over $[T_\circ, T_\circ + 2]$ such that 
		\begin{align*}
			\cN_\bu(T_\circ, \cM), \;\; \cN_\bu(T_\circ + 1, \cM) \leq \eps^{-1}\,.
		\end{align*}
		Then $\forall\, \tau \in (T_\circ, T_\circ + 1]$, $\forall\, R\geq (\tau-T_\circ)^{-1/2}R_{\ref{Cor_L^2 Noncon_Compare d(M), cN(M) and d(M in Q_R), cN(M in Q_R)}}$, we have $\cM(\tau)\cap Q_R\neq\emptyset,\ \cM(\tau+1)\cap Q_R\neq \emptyset$ and 
		\begin{align*}
			\mbfd_\bu(\tau, \cM(\tau))^{-1}\cdot\mbfd_\bu(\tau, \cM(\tau)\cap Q_R) & \geq 1-\frac{R_{\ref{Cor_L^2 Noncon_Compare d(M), cN(M) and d(M in Q_R), cN(M in Q_R)}}}{(\tau-T_\circ) R^2} \,, \\
			\left| \cN_\bu(\tau, \cM\cap Q_R) - \cN_\bu(\tau, \cM) \right| & \leq \frac{R_{\ref{Cor_L^2 Noncon_Compare d(M), cN(M) and d(M in Q_R), cN(M in Q_R)}}}{(\tau-T_\circ) R^2} \,. 
		\end{align*}
	\end{Cor}
	\begin{proof}
		By Lemma \ref{Lem_L^2 Noncon_for RMCF}, for every $\tau\in (T_\circ, T_\circ + 2]$ and $R>0$, 
		\begin{align*}
			0\leq 1-\frac{\mbfd_\bu(\tau, \cM(\tau)\cap Q_R)}{\mbfd_\bu(\tau, \cM(\tau))} 
			& \leq \frac{C(n)}{(\tau-T_\circ) R^2}\cdot \frac{\mbfd_\bu(T_\circ, \cM(T_\circ))}{\mbfd_\bu(\tau, \cM(\tau))} \\
			& \leq \frac{C(n,  \epsilon)}{(\tau-T_\circ) R^2}\cdot \frac{\mbfd_\bu(T_\circ+2, \cM(T_\circ + 2))}{\mbfd_\bu(\tau, \cM(\tau))} \leq \frac{C(n,  \epsilon)}{(\tau-T_\circ) R^2}
		\end{align*}
		Therefore, when $(\tau-T_\circ) R^2\geq R_{\ref{Cor_L^2 Noncon_Compare d(M), cN(M) and d(M in Q_R), cN(M in Q_R)}}(n, \epsilon)^2 \gg 1$, we have $\cM(\tau)\cap Q_R\neq \emptyset$ and 
		\begin{align*}
			& \left| e^{\cN_\bu(\tau, \cM\cap Q_R) - \cN_\bu(\tau, \cM)} - 1 \right| \\
			=\ & \left| \frac{\mbfd_\bu(\tau, \cM(\tau)\cap Q_R)}{\mbfd_\bu(\tau, \cM(\tau))} \cdot \left(\frac{\mbfd_\bu(\tau+1, \cM(\tau+1)\cap Q_R)}{\mbfd_\bu(\tau+1, \cM(\tau+1))}\right)^{-1} - 1\right| \leq \frac{C(n,  \epsilon)}{(\tau-T_\circ) R^2}\,.         
		\end{align*}
		Thus \[
		|\cN_\bu(\tau, \cM\cap Q_R) - \cN_\bu(\tau, \cM)| \leq \frac{C(n,  \epsilon)}{(\tau-T_\circ) R^2} \,.
		\]
	\end{proof}
	
	Let $I\subset \RR_{\geq 0}$. We call a RMCF $I\ni\tau\to \cM(\tau)$ \textbf{$\delta$-$L^2$ close to the RMCF $\cC_{n,k}$ on $I$} if $\forall\, s\in I$, 
	\begin{align}
		\frac12 \cF[\cC_{n,k}]
		\leq 
		\cF[\cM(s)] \leq \frac32 \cF[\cC_{n,k}] \,, & &
		\mbfd_{n,k}(s, \cM(s)) \leq \delta \,, \label{Equ_L^2 Mono_delta L^2 close to cylinder C_(n,k)}
	\end{align}
	where $\cF$ denotes the Gaussian area functional. By White's regularity, if $\cM$ is $\delta$ close to $\cC_{n,k}$ in the Brakke sense over $I$, then $\cM$ is $\Psi(\delta|n, \eps, I)$-$L^2$ close to $\cC_{n,k}$, and vice versa.

	Later, we have various estimates to derive the upper bounds for $\mbfd$ in \eqref{Equ_L^2 Mono_delta L^2 close to cylinder C_(n,k)}, and hence, to show a flow is $\delta$-$L^2$ close to $\cC_{n,k}$, the main challenge is to obtain the Gaussian bound. The following lemma implies that whenever $\cM(\tau)$ is $\delta$-$L^2$ close to $\cC_{n,k}$ for $\tau\in[T-1,T]$, then $\cM(\tau)$ has the desired Gaussian bound for $\tau\in[T,T+1]$.
	\begin{lemma}\label{lem_delta_L2_close_implies_Gaussian_lower_bound}
		There exists $\delta_{\ref{lem_delta_L2_close_implies_Gaussian_lower_bound}}(n,\eps)>0$ with the following significance. Suppose $\tau\to \cM(\tau)$ is a rescaled mean curvature flow over $\tau\in[T-1,T+1]$, $\lambda[\cM]\leq \eps^{-1}$ and $\cM$ is $\delta_{\ref{lem_delta_L2_close_implies_Gaussian_lower_bound}}$-$L^2$ close to $\cC_{n,k}$ on $[T-1,T]$. Then for $\tau\in[T,T+1]$, \[
		\frac12\cF[\cC_{n,k}]\leq \cF[\cM(\tau)]\leq \frac32\cF[\cC_{n,k}] \,.
		\]
	\end{lemma}
	\begin{proof}
		By Lemma \ref{Lem_L^2 clos => C^2 close}, when $\delta(n,\eps)$ is sufficiently small,  $\cM(T)\cap Q_{R_n}$ is a $C^2$-graph of a function $u$ over $\cC_{n,k}\cap Q_{R_n}$, with $\|u\|_{C^2(\cC_{n,k}\cap Q_{R_{n,\eps}})}\leq C(n,\eps)\delta(n,\eps)$. Then by the Brakke-White regularity theorem, $\cM$ is a graph of some function over $\cC_{n,k}\cap Q_{R_{n,\eps}/2}$, with $\|u(\cdot,\tau)\|_{C^2(\cC_{n,k}\cap Q_{R_{n,\eps}})}\leq C'(n,\eps)\delta(n,\eps)$, for $\tau\in[T,T+1]$. Therefore, for $\tau\in[T,T+1]$ and a sufficiently large choice of $R_{n,\eps}$, when $\delta\ll 1$, we have $\cF[\cM(\tau)]\geq \frac{2}{3}\cF[\cC_{n,k}\cap Q_{R/2}]>\frac{1}{2}\cF[\cC_{n,k}]$; $\cF[\cM(\tau)]\leq \frac{4}{3}\cF[\cC_{n,k}\cap Q_{R_n/2}]+O(\eps^{-1}e^{-R^2/16})\leq \frac{3}{2}\cF[\cC_{n,k}]$.
	\end{proof}
	
	The following Lemma will be used in estimating the graphical radius in Section \ref{Sec_AsympProfile}.
	\begin{Lem} \label{Lem_Prelim_M cap Q_R in Q_2n}
		For every $\eps\in (0, 1)$, there exists $\delta_{\ref{Lem_Prelim_M cap Q_R in Q_2n}}(n, \eps)\in (0, \eps)$ with the following property. 
		Let $T\geq 1$, $\tau\mapsto \cM(\tau)$ be a RMCF in $\R^{n+1}$ with entropy $\lambda[\cM]\leq \eps^{-1}$ and $\delta_{\ref{Lem_Prelim_M cap Q_R in Q_2n}}$-$L^2$ close to $\cC_{n,k}$ over $[0, T]$. Then for every $\tau\in [1, T]$, \[
		(\spt\cM)(\tau) \cap \left(B^{n-k+1}_{2n+e^{\tau/2}}\times B^k_{2n} \right) \subset Q_{2n} \,.
		\]
	\end{Lem}
	\begin{proof}
		Suppose for contradiction that there exist $\tau_j\geq 1$, RMCF $\cM_j$ with $\lambda[\cM]\leq \eps^{-1}$ which is $1/j$-$L^2$ close to $\cC_{n,k}$ over $[0, \tau_j]$, and $(x_j, y_j)\in \R^{n-k+1}\times \R^k$ such that 
		\begin{align*}
			(x_j, y_j)\in (\spt\cM_j)(\tau_j)\,, & & 
			2n\leq |x_j| \leq 2n + e^{\tau_j/2}, \quad |y_j|\leq 2n\,.
		\end{align*}
		Also suppose $|x_j|^{-1}$ converges to some $a\in [0, 1/2n]$.
		
		Consider the time-translated RMCF $\tau\mapsto \hat\cM_j(\tau):= \cM_j(\tau + \tau_j- 2\ln|x_j|)$. We then see that $\hat\cM_j$ is $1/j$-$L^2$ close to $\cC_{n,k}$ over $[2\ln|x_j|-\tau_j, 2\ln|x_j|]$, and that $(x_j, y_j)\in (\spt\hat\cM_j)(2\ln|x_j|)$. 
		Let $t\mapsto \hat\bM_j(t):= \sqrt{-t}\, \hat\cM_j(-\ln(-t))$ be the MCF associated to $\hat\cM_j$, then when $j\to \infty$, we have $\bM_j$ converges in the Brakke sense to the shrinking round cylinder $t\mapsto \sqrt{-t}\, \cC_{n,k}$. However, by the upper semi-continuity of Gaussian density under Brakke convergence, since $(x_j/|x_j|, y_j/|x_j|, -|x_j|^{-2})\in \spt\hat\bM_j$, its subsequential limit $(\hat x_\infty, 0, t_\infty=-a^2)$ must be a point in the support of a shrinking round cylinder. This is a contradiction since $|\hat x_\infty|=1$ but $a\leq 1/2n$.
	\end{proof}

	An application of Lemma \ref{Lem_L^2 Noncon_for RMCF} is that the decay order upper bound allows us to take the normalized limit of the graphical function of rescaled mean curvature flow over round cylinders.
	\begin{Lem} \label{Lem_L^2 Noncon_Induced Parab Jac}
		Let $\beta>0$, $\delta_j\searrow 0$; Let $\cM_j$ be a sequence of RMCF in $\RR^{n+1}$ over $[0, T]$ converging to $\cC_{n,k}$ in the Brakke sense, where $T\geq 1$; $\bu_j\in \scU$ be a sequence of arrival time functions approaching $\mbfU_{n,k}$ as $j\to \infty$. Suppose \[
		\limsup_{j\to \infty} \cN_{\bu_j}(0, \cM_j) <+\infty \,.
		\]
		Let $v_j(\cdot, \tau)$ be the graphical function of $\cM_j(\tau)$ over $\cC_{n,k}$, defined on a larger and larger domain as $j\to \infty$. Then after passing to a subsequence, $\hat{v}_j:= \mbfd_{\bu_j}(1, \cM_j)^{-1}(v_j - \varphi_{\bu_j})$ converges to some non-zero $\hat{v}$ in $C^\infty_{loc}(\cC_{n,k}\times (0, T])$ solving \[
		\partial_\tau \hat v - L_{n,k} \hat v = 0\,.
		\]
		Moreover, for every $\tau\in (0, T]$, we have \[
		\|\hat{v}(\cdot, \tau)\|_{L^2(\cC_{n,k})} = \lim_{j\to \infty} \mbfd_{\bu_j}(1,\cM_j)^{-1}\mbfd_{\bu_j}(\tau, \cM_j) < +\infty\,.     \]
	\end{Lem}
	We shall call such non-zero renormalized limit $\hat v$ a \textbf{induced (parabolic) Jacobi field} from the sequence $\{\cM_j\}$ relative to $\{\bu_j\}$. We refer the readers to Lemma \ref{Lem_App_Analysis of Parab Jacob field} for properties of parabolic Jacobi fields over $\cC_{n,k}$.

	\begin{proof}
		By Brakke-White regularity of mean curvature flow and interior parabolic estimate, $v_j$ is defined on a larger and larger domain exhausting $\cC_{n,k}$ and $(v_j-\varphi_{\bu_j})\to 0$ in $C^\infty_{loc}(\cC_{n,k}\times (0, T])$ as $j\to \infty$.  
		By Lemma \ref{Lem_App_Graph over Cylinder}, Corollary \ref{Cor_L^2 Noncon_Compare d(M), cN(M) and d(M in Q_R), cN(M in Q_R)} and the upper bound of decay order (denoted by $\epsilon^{-1}$), there exists $R_1(n,\epsilon)\gg1$ such that for every $\tau\in (0, T]$, every $R\geq \tau^{-1}R_1$ and $j\gg1$, we have \[
		\|(v_j-\varphi_{\bu_j})(\tau, \cdot)\|_{L^2(\cC_{n,k}\setminus Q_R)} + \mbfd_{\bu_j}(\tau, \cM_j\setminus Q_R) \leq \frac{C(n,\epsilon)}{\tau R^2}\cdot \mbfd_{\bu_j}(1, \cM_j\setminus Q_R).
		\]
		Therefore, combined with Lemma \ref{Lem_App_Analysis of Parab Jacob field} and the classical parabolic regularity estimates, $\hat{v}_j:= \mbfd_{\bu_j}(1,\cM_j)^{-1}\, (v_j-\varphi_{\bu_j})$ subconverges to some non-zero $\hat{v}\in C^\infty(\cC_{n,k}\times (0, T])$, and such that for every $\tau\in (0, T]$, \[
		\|\hat{v}(\cdot,\tau)\|_{L^2(\cC_{n,k})} =  \lim_{j\to \infty} \mbfd_{\bu_j}(1,\cM_j)^{-1}\cdot \mbfd_{\bu_j}(\tau, \cM_j) <+\infty \,.
		\] 
	\end{proof}

	\begin{Cor} \label{Cor_L^2 Mono_RMCF w graphical eigenfunc has cN = spectrum}
		For every $\eps\in (0, 1/2)$, there exists a $\delta_{\ref{Cor_L^2 Mono_RMCF w graphical eigenfunc has cN = spectrum}}(n, \eps)\in (0, \epsilon)$ with the following significance. 
		Let $\bu\in \scU$, $T_\circ>0$ so that $\|\varphi_\bu(\cdot,T_\circ)\|_{C^4}<\delta_{\ref{Cor_L^2 Mono_RMCF w graphical eigenfunc has cN = spectrum}}$, $T\in [1, \eps^{-1}]$,  $\cM$ be a RMCF in $\RR^{n+1}$ $\delta_{\ref{Cor_L^2 Mono_RMCF w graphical eigenfunc has cN = spectrum}}$-$L^2$ close to $\cC_{n,k}$ over $[T_\circ, T_\circ+T]$ such that \[
		\cN_\bu(0, \cM) \leq \epsilon^{-1}\,.
		\] 
		Let $\gamma\in [-\eps^{-1}, \eps^{-1}]$, $\tau_\circ\in [\eps, T]$, $\sim\in \{\geq, =, \leq\}$, and $v$ be the graphical function of $\cM$ over $\cC_{n,k}$. Also, suppose that \[
		\|\Pi_{\sim\gamma}((v-\varphi_\bu)(\cdot, \tau_\circ))\|_{L^2} \geq (1-\delta)\|(v-\varphi_\bu)(\cdot, \tau_\circ)\|_{L^2} \,.  \]
		Then for every $\tau\in [\eps, T]$, 
		\begin{align*}
			\cN_\bu(\tau, \cM) - \gamma \begin{cases}
				\leq  \eps , & \text{ if }\sim \text{ is }\leq\,; \\
				\geq -\eps, & \text{ if }\sim \text{ is }\geq\,; \\
				\in [-\eps, \eps], & \text{ if }\sim \text{ is }=\,.  \\
			\end{cases}
		\end{align*}
	\end{Cor}
	\begin{proof}
		The corollary follows by a direct contradiction argument combining Lemma \ref{Lem_L^2 Noncon_Induced Parab Jac} and (ii) of Lemma \ref{Lem_App_Analysis of Parab Jacob field}.
	\end{proof}

	\begin{Cor}[Discrete Monotonicity of the decay order] \label{Cor_L^2 Mono_Discrete Growth Mono}
		For every $\eps\in (0, 1/2)$, there exists $\delta_{\ref{Cor_L^2 Mono_Discrete Growth Mono}}(n, \eps)\in (0, \eps)$ such that the following hold.  If $T\geq 0$, $\bu\in \scU$ such that $\|\varphi_\bu(\cdot,\tau)\|_{C^4}<\delta_{\ref{Cor_L^2 Mono_Discrete Growth Mono}}$ for $\tau\in[T,T+2]$, $\cM$ is a RMCF in $\RR^{n+1}$ $\delta_{\ref{Cor_L^2 Mono_Discrete Growth Mono}}$-$L^2$ close to $\cC_{n,k}$ over $[T, T+2]$, and satisfies the decay order bound, 
		\begin{align}
			\cN_{\bu}(0, \cM) \leq \epsilon^{-1} \,. \label{Equ_L^2 Mono_Entropy and Doubl Const Bd}
		\end{align}
		Then at least one of the following holds,
		\begin{align}
			&\text{either}  & -1-\epsilon \leq \cN_\bu(T+1, \cM) & \leq \cN_\bu(T, \cM) - \delta_{\ref{Cor_L^2 Mono_Discrete Growth Mono}}\,;  \label{Equ_L^2 Mono_Strict Doubl Const drop}  \\
			&\text{or}  & \sup_{\tau\in [T+\epsilon, T+1]}|\cN_\bu(\tau, \cM) - \gamma| & \leq \epsilon\,, \;\;\;\;\; \text{ for some }\gamma\in \sigma(\cC_{n,k})\,.  \label{Equ_L^2 Mono_Doubl Const close to spectrum}
		\end{align}
		Moreover, if (\ref{Equ_L^2 Mono_Strict Doubl Const drop}) fails and $\gamma$ is given by (\ref{Equ_L^2 Mono_Doubl Const close to spectrum}), then the graphical function $v$ of $\cM$ over $\cC_{n,k}$ satisfies 
		\begin{align}
			\|\Pi_{=\gamma} ((v-\varphi_\bu)(\cdot, \tau))\|_{L^2} \geq (1-\epsilon)\|(v-\varphi_\bu)(\cdot, \tau)\|_{L^2}\,,  \quad \forall\, \tau\in [\epsilon, 2] \,. \label{Equ_L^2 Mono_Jac field concentra in eigenspace} 
		\end{align}
	\end{Cor}
	\begin{Rem} \label{Rem_L^2 Mono_cN<gamma pass to next scale}
		A useful consequence of at least one of (\ref{Equ_L^2 Mono_Strict Doubl Const drop}) and (\ref{Equ_L^2 Mono_Doubl Const close to spectrum}) being true is that, if $\gamma\in \RR$ with $\dist_\RR(\gamma, \sigma(\cC_{n,k}))\geq \epsilon$, then when $T\geq T_\circ$, $\cN_\bu(T, \cM)\leq \gamma$ implies $\cN_\bu(T+1, \cM)\leq \gamma$. 
	\end{Rem}
	\begin{proof}
		In view of Lemma \ref{Lem_L^2 Noncon_for RMCF}, it suffices to show that for every $\epsilon>0$ and sequence of rescaled mean curvature flow $\tau\mapsto \cM_j(\tau)$ over $\tau\in [0,2]$ satisfying (\ref{Equ_L^2 Mono_Entropy and Doubl Const Bd}) with $\|\varphi_{\bu_j}(\cdot,\tau)\|_{C^4}<1/j$ for $\tau\in[T,T+2]$ and converging to the multiplicity $1$ static flow $\cC_{n,k}$ in the Brakke sense as $j\to \infty$, if \eqref{Equ_L^2 Mono_Strict Doubl Const drop} fails, namely,  
		\begin{align*}
			\text{ either }\;\; \cN_{\bu}(1, \cM_j) \geq \cN_{\bu_j}(0, \cM_j) - \frac1{j} \,, & &
			\text{ or }\;\; \cN_{\bu_j}(1, \cM_j) < -1-\epsilon = \inf \sigma(\cC_{n,k}) - \epsilon \,. 
		\end{align*}
		Then \eqref{Equ_L^2 Mono_Doubl Const close to spectrum} holds, namely, there exists $\gamma\in \sigma(\cC_{n,k})$ such that
		\begin{align*}
			\limsup_{j\to \infty} \sup_{\tau\in [T+\epsilon, T+1]} |\cN_{\bu_j}(\tau, \cM_j)-\gamma| = 0 \,. 
		\end{align*}
		and that the induced Jacobi fields from $\{\cM_j\}$ relative to $\{\bu_j\}$ exist and are all given by $e^{-\gamma\tau}w$ for some $\gamma$-eigenfunction $w$ of $-L_{n,k}$.
		
		By Lemma \ref{Lem_L^2 Noncon_Induced Parab Jac}, there always exists an induced Jacobi field $\hat v$, and \[
		\lim\limits_{j\to\infty}\cN_{\bu_j}(T,\cM_j)=\lim\limits_{j\to\infty}\cN_{\bu_j}(T+1,\cM_j)=N_{n,k}(v)=:\sigma \,,
		\] 
		where $v\sim e^{-\sigma\tau}$. Here $N_{n,k}(v)$ is the linear decay order of a parabolic Jacobi field $v$ defined in Lemma \ref{Lem_App_Analysis of Parab Jacob field}. Then the desired property is a direct consequence of the parabolic Jacobi field over $\cC_{n,k}$, see Lemma \ref{Lem_App_Analysis of Parab Jacob field}.
	\end{proof}
	
	A direct consequence of Corollary \ref{Cor_L^2 Mono_Discrete Growth Mono} is that the decay order at infinity is well-defined for those RMCF that converge to $\cC_{n,k}$.
	\begin{Cor} \label{Cor_L^2 Mono_Def cN_u(infty)}
		Let $\cM$ be a RMCF in $\R^{n+1}$ with finite entropy over $\R_{\tau\geq 0}$ such that $\cM(\tau)$ locally smoothly converges to $\cC_{n,k}$ as $\tau\to +\infty$; $\beta\in (0, 1)$, $\bu\in \scU$. Then the following limit exists and belongs to $(\sigma(\cC_{n,k})\cap \R_{\geq 0})\cup \{+\infty\}$: 
		\begin{align}
			\cN_{\bu}(\infty, \cM):= \lim_{\tau\to +\infty} \cN_\bu(\tau, \cM) \in (\sigma(\cC_{n,k})\cap \R_{\geq 0})\cup \{+\infty\} \,. \label{Equ_L^2 Mono_Def cN_u(infty)}
		\end{align}
	\end{Cor}
	\begin{proof}
		By the $C^1$-normal form of the finite entropy RMCF locally smoothly converges to $\cC_{n,k}$ (Theorem \ref{thm:PseduoLocality_Partial}), $\cM(\tau)$ is $\delta(\tau)$-close to $\cC_{n,k}$, with $\delta(\tau)\to 0$ as $\tau\to\infty$. By taking $\eps\to0$ in Corollary \ref{Cor_L^2 Mono_Discrete Growth Mono}, we have $\liminf\limits_{\tau\to +\infty} \cN_\bu(\tau, \cM)\in (\sigma(\cC_{n,k})\cap \R_{\geq 0})\cup \{+\infty\}$ and $\limsup\limits_{\tau\to +\infty} \cN_\bu(\tau, \cM)\in (\sigma(\cC_{n,k})\cap \R_{\geq 0})\cup \{+\infty\}$. Moreover, Corollary \ref{Cor_L^2 Mono_Discrete Growth Mono} shows that for any $\eps>0$, once $\cN(\tau,\cM)<\lambda-\eps$ for $\lambda\in (\sigma(\cC_{n,k})\cap \R_{\geq 0})$ for a sufficiently large $\bar\tau$, then $\cN(\tau,\cM)<\lambda-\eps$ for all $\tau>\bar\tau$. This implies that $\limsup\limits_{\tau\to +\infty} \cN_\bu(\tau, \cM)=\liminf\limits_{\tau\to +\infty} \cN_\bu(\tau, \cM)$, which completes the proof.
	\end{proof}

	The following lemma shows that if the rescaled mean curvature flow detects a cylindrical singularity, then the decay order can not be too negative.
	\begin{Lem}\label{Lem_L^2 Mono_cN>-eps if singular}
		For every $\eps\in(0, 1)$, there exists a $\delta_{\ref{Lem_L^2 Mono_cN>-eps if singular}}(n,\eps)>0$ with the following properties. 
		Let $T_\circ\geq 0$, $\tau\to\cM(\tau)$ be a RMCF defined for all $\tau\geq T_\circ$, and suppose it is $\delta$-$L^2$ close to $\cC_{n,k}$ on $[T_\circ, T_\circ+1]$ with $\delta\leq \delta_{\ref{Lem_L^2 Mono_cN>-eps if singular}}$, but not identical to $\cC_{n,k}$. Suppose 
		\begin{align*}
			\lambda[\cM] \leq \eps^{-1}\,, & &
			\lim_{\tau\to +\infty} \cF[\cM(\tau)] = \cF[\cC_{n,k}]\,, 
		\end{align*}
		Then we have $\cN_{n,k}(\tau, \cM)\geq -\eps$ for every $\tau\geq T_\circ+1$.
	\end{Lem}
	\begin{proof}
		By Brakke-White's $\epsilon$-regularity \cite{White05_MCFReg}, there exists $R\geq \Psi(\delta|n, \eps)^{-1}$ and $v\in C^2(\cC_{n,k}\cap Q_R)$ with $\|v\|_{C^2}\leq \Psi(\delta|n ,\eps)$ such that if $\cM$ is described as in the Lemma with $\delta_{\ref{Lem_L^2 Mono_cN>-eps if singular}}(n, \eps)\ll 1$, then $\cM(T_\circ+1)$ is a $C^2$ graph over $\cC_{n,k}$ in $Q_R$ with graphical function $v$. In particular, 
		\begin{align*}
			\cF[\cM(T_\circ+1)] 
			& = \int_{\cM(T_\circ+1)\cap Q_R} d\mu + \int_{\cM(T_\circ+1)\setminus Q_R} d\mu \\
			& \leq (1+C_n\|v\|_{C^2}) \cF[\cC_{n,k}\cap Q_R] + C(n, \eps) e^{-R^2/8} \\
			& \leq \cF[\cC_{n,k}] + \Psi(\delta|n, \eps) \,.
		\end{align*}
		Combining with Huisken's Monotonicity formula and that $\cF[\cM(\tau)]\to \cF[\cC_{n,k}]$ as $\tau\to +\infty$, we conclude that for every $\tau\geq T_\circ+1$, \[
		|\cF[\cM(\tau)]-\cF[\cC_{n,k}]| \leq \Psi(\delta|n, \eps) \,.
		\]
		By \cite[Theorem 0.5]{ColdingMinicozzi25_quantitativeMCF} and taking $\delta_{\ref{Lem_L^2 Mono_cN>-eps if singular}}(n, \eps)\ll 1$, we have for every $\tau\geq T_\circ+1$, 
		\begin{align}
			\mbfd_{n,k}(\tau, \cM) \leq \delta_{\ref{Cor_L^2 Mono_Discrete Growth Mono}}(n, \eps) \,. \label{Equ_L^2 Mono_Pf_d_(n,k) < delta}
		\end{align}
		And then $\cM$ is $\delta_{\ref{Cor_L^2 Mono_Discrete Growth Mono}}$-$L^2$ close to $\cC_{n,k}$ over $\R_{\geq T_\circ}$.
		By Corollary \ref{Cor_L^2 Mono_Discrete Growth Mono}, if $\cN_{n,k}(\tau, \cM)\leq -\eps$ for some $\tau\geq T_\circ+1$, then inductively $\cN_{n,k}(\tau+\ell, \cM)\leq -\eps$ for every integer $\ell\geq 0$. By definition of decay order, this implies, \[
		\mbfd_{n,k}(\tau+\ell, \cM) \geq \mbfd_{n,k}(\tau, \cM) e^{\eps\ell} \,.
		\]
		Since $\mbfd_{n,k}(\tau, \cM)\neq 0$, this contradicts to \eqref{Equ_L^2 Mono_Pf_d_(n,k) < delta} by sending $\ell\to +\infty$.
	\end{proof}

	\begin{Cor} \label{Cor_L^2 Mono_cN>-1}
		Let $T\geq 0$, $\bu \in \scU$; $\cM$ be a RMCF in $\RR^{n+1}$ over $[T, T']$ with $\lambda[\cM]\leq \eps^{-1}$ and $\cF[\cM(T)]>1/2\cF[\cC_{n,k}]$. Then for every $\tau\in [T, T']$, \[
		\mbfd_\bu(\tau, \cM) \lesssim_{n, \beta, \eps} e^{(1+\eps)(\tau-T)}\mbfd_\bu(T, \cM) \,.
		\]
	\end{Cor}

	\begin{proof}
		Let $[T,\bar T]$ be the maximum interval such that $\cM(\tau)$ is $\delta$-$L^2$ close to $\cC_{n,k}$ for $\tau\in[T,\bar T]$, where $\delta$ is given as in Lemma \ref{lem_delta_L2_close_implies_Gaussian_lower_bound}. We consider several situations. 
		
		{\bf Case 1.} If $\bar T$ exists and $\bar T=T'$. Let $\eps\in (0, 1)$, $T_\circ>0$ (only depend on $\scU$), $\bu\in \scU$ such that when $\tau>T_\circ$, $\|\varphi_\bu(\cdot,\tau)\|_{C^4}\leq \delta_{\ref{Cor_L^2 Mono_Discrete Growth Mono}}$. By Corollary \ref{Cor_L^2 Mono_Discrete Growth Mono}, we have $\cN_\bu(\tau, \cM)\geq -1-\epsilon$ for $\tau\in[T_\circ+1,T'-1]$. In particular, using an induction argument, we can conclude that when $\tau\in[T_\circ+1,T'-1]$, we have
		\[
		\mbfd_\bu(\tau, \cM)\leq C(n,\eps)e^{(1+\eps)(\tau-T_\circ)}\mbfd_\bu(T_\circ, \cM).
		\]
		Again using Lemma \ref{Lem_L^2 Noncon_for RMCF} (over $[T,T_\circ]$ and $[T'-1,T']$) and because $T_\circ$ only depends on $n,\beta,\eps$, we have for all $\tau\in[T,T']$,
		\[
		\mbfd_\bu(\tau, \cM)\leq C(n,\eps)e^{(1+\eps)(\tau-T)}\mbfd_\bu(T, \cM).
		\]
		
		{\bf Case 2.} If $\bar T$ exists and $T\leq \bar T<T'$. Then by the definition of $\delta$-$L^2$ closeness, and Lemma \ref{lem_delta_L2_close_implies_Gaussian_lower_bound}, we must have $\mbfd_\bu(\bar T, \cM)=\delta$. Then we can apply the proof in Case 1 to show that $\delta\leq C(n ,\eps)e^{(1+\eps)(\bar T-T_\circ)}\mbfd_\bu(T, \cM)$, and hence for any $\tau\in[\bar T,T']$, 
		\[
		\mbfd_\bu(\tau, \cM)\leq C(n)\lambda[\cM]\leq C(n)\eps^{-1}\leq C(n,\eps ) \delta \leq C(n,\eps )e^{(1+\eps)(\tau-T)}\mbfd_\bu(T, \cM).
		\]
		
		{\bf Case 3.} If $\bar T$ does not exist. This implies that either $\cF[\cM(T)]>3/2\cF[\cC_{n,k}]$, or $\mbfd_\bu(T, \cM)>\delta$. Note that $\cF[\cM(T)]>3/2\cF[\cC_{n,k}]$ implies that $\mbfd_\bu(T, \cM)\geq C(n)$ from the proof of Lemma \ref{lem_delta_L2_close_implies_Gaussian_lower_bound}. Therefore, in the latter two situations, similar to Case 2, we can use Huisken's monotonicity formula and $\mbfd_\bu(\tau, \cM)\leq C(n)\cF[\cM(\tau)]$ to show that 
		\[
		\mbfd_\bu(\tau, \cM)\leq C(n,\eps)e^{(1+\eps)(\tau-T_\circ)}\mbfd_\bu(T, \cM).
		\]
	\end{proof}

	\section{Asymptotic profile at cylindrical singular point} \label{Sec_AsympProfile}
	The goal of this section is to prove Theorem \ref{thm:AsyProfile}. Recall that the family of low spherical flows $\scU$ has been defined in Definition \ref{Def_LowSphericalFlow}. 
	Throughout the section, we assume that $\cM: \tau \mapsto \cM(\tau)$ is a RMCF with finite entropy in $\RR^{n+1}$ over $\RR_{\tau \geq 0}$, such that $\cM(\tau)$ $C^\infty_{loc}$-converges to $\cC_{n,k}$ as $\tau\to +\infty$. Let $\gfrd_1(\cM(\tau))$ and $v(\cdot, \tau)$ be the $1$-graphical radius and graphical function of $\cM(\tau)$ over $\cC_{n,k}$. 
	
	Case \ref{item_Intro_PolynDecay} of Theorem \ref{thm:AsyProfile} is essentially treated in \cite{SunXue2022_generic_cylindrical}. See Section \ref{Subsec_Pf Main AsympProf}.  
	The bulk of this section is devoted to case \ref{item_Intro_ExpDecay} and \ref{item_Intro_SuperExpDecay} of Theorem \ref{thm:AsyProfile}, i.e., when the flow converges to $\cC_{n,k}$ exponentially or even faster. For later applications, we shall prove the following asymptotic profile theorems with a more quantitative estimate. Denote for simplicity the maximal decay order limit, modulo all the low spherical flow to be \[
	\cN_\scU(\infty, \cM) := \sup_{\bu\in \scU} \cN_\bu(\infty, \cM) \ \ (\in \sigma(\cC_{n, k}) \cup \{+\infty\}\, \text{ by Corollary \ref{Cor_L^2 Mono_Def cN_u(infty)}})\,.
	\]
	\begin{Thm} \label{Thm_Asymp Profile(New)_gamma>0} 
		Let $\cM$ be specified as above, such that $\cN_{n,k}(\infty, \cM)>0$. Then 
		\begin{align}
			\cN_\scU(\infty, \cM) \geq 1/2 \,.  \label{Equ_AsympProf_Main_cN(infty) in sigma(cC)}
		\end{align}
		Moreover, for every $\eps\in (0, 1/4n)$, there exists $\delta_{\ref{Thm_Asymp Profile(New)_gamma>0}}(n, \eps)\in (0, \eps)$ with the following significance. If $\cM$ (as above), $\delta, \gamma, \gamma^+, \eps$ satisfy the following with $\delta\leq \delta_{\ref{Thm_Asymp Profile(New)_gamma>0}}$: 
		\begin{enumerate}[label={\normalfont(\Roman*)}]
			\item\label{Item_AsyPrNew_Assum1} $\cM: \tau \mapsto \cM(\tau)$ is $\delta$-$L^2$ close to $\cC_{n,k}$ over $[0, 1]$ with $\lambda[\cM]\leq \eps^{-1}$; 
			\item\label{Item_AsyPrNew_Assum2} $\gamma\in \sigma(\cC_{n,k})\cap [1/2, \eps^{-1}-1]$ satisfies \[
			\gamma \leq  \cN_{\scU}(\infty, \cM) \,.
			\]
			Let $\gamma' := \min \sigma(\cC_{n,k})\cap (\gamma, +\infty)$ be the minimal spectrum of $-L_{n,k}$ greater than $\gamma$; and \[
			\gamma^+:= \min\left\{\gamma',\ \gamma+\gamma_{n,k},\ 2\gamma \right\} \in (\gamma, \gamma']\,,
			\] 
			where recall $\gamma_{n,k}=1/(n-k)$ ($\gamma^+$ will play the role of the decay order for the error term). 
		\end{enumerate}
		Then there exist a unique low spherical flow $\bu\in \scU$ (depending only on $\cM$ but not on $\gamma$) and a unique eigenfunction $\psi\in \rmW_\gamma(\cC_{n,k})\cap \rmW_{\SSp}^\perp$ of $-L_{n,k}$ (non-vanishing if and only if $\gamma=\cN_\scU(\infty,\cM)$) with $\|\psi\|_{L^2}\leq 1$, such that 
		\begin{align}
			\cN_\scU(\infty, \cM) = \cN_\bu(\infty, \cM)\,;  \label{Equ_AsympProf_Main_cN_scU(infty) = cN_u(infty)}
		\end{align}
		and that for every $\tau\geq 1$, 
		\begin{enumerate} [label={\normalfont(\roman*)}]
			\item the graphical radius of $\cM(\tau)$ over $\cC_{n,k}$ satisfies,
			\begin{align}
				\gfrd_1(\cM(\tau)) \geq 2n+ e^{\frac{\gamma-\eps}{2(\gamma+1)}\cdot\tau} =: \rmR_{\gamma, \eps}(\tau)\,; \label{Equ_AsympProf_Main_Gfrad bd}
			\end{align}
			\item for every $2n\leq R\leq \rmR_{\gamma, \eps}(\tau)$, the graphical function $v(\cdot, \tau)$ satisfies a rough $C^2$ estimate,
			\begin{align}
				\|(v-\varphi_\bu)(\cdot, \tau)\|_{C^2(\cC_{n,k}\cap Q_R)} \lesssim_{n, \eps} e^{-(\gamma-\eps)\tau} R^{2(1+\gamma)}\mbfd_\bu(0,\cM)^{1-\eps} \,;  \label{Equ_AsympProf_Main_Rough C^2 Bd}
			\end{align}
			and an improved error estimate in the asymptotic expansion: 
			\begin{align}
				\|(v-\varphi_{\bu})(\cdot, \tau) - e^{-\gamma\tau}\psi\|_{C^0(\cC_{n,k}\cap Q_R)}  \lesssim_{n, \eps} e^{-(\gamma^+-3\eps)\tau}R^{2(2+\gamma^+)} \,.  \label{Equ_AsympProf_Main_Precise C^0 Bd}
			\end{align}
		\end{enumerate}
	\end{Thm}  
	
	\begin{Rem} \label{Rem_AsympProf_Whitney data}
		An interesting consequence of Theorem \ref{Thm_Asymp Profile(New)_gamma>0} follows by taking $\gamma=1/2$, in which case \[
		\gamma^+ = \frac12 + \min \{\frac12, \frac1{n-k}\} =: \gamma_{n,k}^+ \,.
		\]
		In fact, for every degenerate $k$-cylindrical singularity $p\in \cS_k(\bM)_+$, suppose $\bM$ has tangent flow at $p$ to be $\cC_p \in \Rot(\cC_{n,k})$. Then the assumptions \ref{Item_AsyPrNew_Assum1} \& \ref{Item_AsyPrNew_Assum2} in the Theorem above are satisfied by an appropriate rotation and time translation of the RMCF $\cM^p$ based at $p$, with $\eps:= \min\{\lambda[\bM]^{-1}, 1/4n\}$, $\gamma=1/2$ and $\delta=\delta_{\ref{Thm_Asymp Profile(New)_gamma>0}}(n, \eps)$. This determines a low spherical flow $\bu_p$ asymptotic to $\cC_p$ and a unique $\psi_p\in \rmW_{1/2}(\cC_p)$ such that the graphical function $v^p(\cdot, \tau)$ of $\cM^p(\cdot, \tau)$ over $\cC_p$ satisfies \[
		v^p(\cdot, \tau) = \varphi_{\bu_p}(\cdot, \tau) + e^{-\tau/2}\psi_p + O(e^{-(\gamma_{n,k}^+-3\eps)\tau})\,, \quad \forall\, \tau\gg 1 \,.
		\]
		in $L^2(\cC_p)$. And if we denote by $|\mbfL_p|:= \spine(\cC_p)$. Then by \ref{Item_Pre_W_1/2} of Section \ref{SSubsec_Spectrum of L}, $\psi_p$ determines a unique $|\mbfL_p|^\perp$-valued quadratic polynomial $\bq_p$ on $|\mbfL_p|$ by \[
		\psi_p(\theta, y) =: \left\langle \bq_p(y) - 2\tr \bq_p, \hat\theta \right\rangle + \bc_p(y)\,,
		\]
		where $(\theta, y)\in (|\mbfL_p|^\perp\cap S^{n}(\sqrt{2(n-k)}))\times |\mbfL_p|$ parametrizes $\cC_p$, $\hat\theta:= \theta/|\theta|$; $\bc_p$ is a cubic Hermitian polynomial on $|\mbfL_p|\approx \R^k$.
	\end{Rem}

	\subsection{Asymptotic profile relative to a fixed spherical flow}  \label{Subsec_AsympProf_fix u}
	We first treat the case relative to a fixed low spherical flow $\cC_\bu$. The following Lemma provides an almost optimal graphical radius estimate, together with a rough $C^2$ estimate in spacetime. We emphasize that we don't need to assume a decay order upper bound in this Lemma.
	\begin{Lem} \label{Lem_AsympProf_Rough C^2 est}
		$\forall\, \eps\in (0, 1/(4n))$, there exists $\delta_{\ref{Lem_AsympProf_Rough C^2 est}}(n, \eps)\in (0, \eps)$ 
		with the following significance. Suppose 
		\begin{enumerate} [label={\normalfont(\alph*)}]
			\item\label{Item_AsympProf_Rough_RMCF delta close} $\cM$ is a RMCF $\delta$-$L^2$ close to $\cC_{n,k}$ over $\RR_{\geq 0}$, with $\delta\leq \delta_{\ref{Lem_AsympProf_Rough C^2 est}}$ and entropy bound $\lambda[\cM]\leq \eps^{-1}$;
			\item\label{Item_AsympProf_Rough_mu > lim cN_u} $\bu\in \scU$ is a low spherical flow with $\|\varphi_\bu\|_{C^4}\leq \delta_{\ref{Lem_AsympProf_Rough C^2 est}}$; $\mu\in (2\eps, \eps^{-1}]$ such that $\cN_\bu(\infty, \cM)\geq \mu$. 
		\end{enumerate}
		Then for every $\tau \geq 1$, we have graphical radius lower bound, 
		\begin{align}
			\gfrd_1(\cM(\tau)) \geq 2n + e^{\frac{\mu-\eps}{2(\mu+1)}\cdot\tau} =: \rmR_{\mu,\eps}(\tau)\,;  \label{Equ_AsympProf_Rough_Gfrad est}
		\end{align}
		and a rough $C^2$ estimate of graphical function $v(\cdot, \tau)$ of $\cM(\tau)$ over $\cC_{n,k}$: $\forall\, 2n \leq R\leq \rmR_{\mu,\eps}(\tau)$, 
		\begin{align}
			\|(v-\varphi_\bu)(\cdot, \tau)\|_{C^2(\cC_{n,k}\cap Q_R)} \lesssim_{n,\eps} e^{-(\mu-\eps)\tau}R^{2(1+\mu)}\cdot \mbfd_\bu(0, \cM)^{1-\eps} \,.  \label{Equ_AsympProf_Rough Spacetime C^2 est for u}
		\end{align}
	\end{Lem}
	\begin{proof}
		We first claim that by taking $\delta_{\ref{Lem_AsympProf_Rough C^2 est}}(n, \eps)\ll 1$ and $\rmT'(n,\eps)\gg 1$, we have for every $\tau\geq \rmT'$, 
		\begin{align}
			\cN_\bu(\tau, \cM)\geq \mu - \eps/2 \,.  \label{Equ_AsympProf_N_u(M)> mu - eps}
		\end{align}
		Note that this implies that (by definition) for every $\tau\geq \rmT'$, 
		\begin{align}
			\mbfd_\bu(\tau, \cM) \leq e^{-(\mu-\eps/2)\lfloor\tau-\rmT' \rfloor}\mbfd_\bu(\tau-\lfloor\tau-\rmT' \rfloor, \cM) \lesssim_{n, \eps} e^{-(\mu-\eps/2)\tau} \mbfd_\bu(0, \cM)\,;  \label{Equ_AsympProf_Rough e^(-mu tau) decay}
		\end{align}
		where the last inequality follows from Lemma \ref{Lem_L^2 Noncon_for RMCF}; while since we have fix the choice of $\rmT'(n, \eps)$ here, also by Lemma \ref{Lem_L^2 Noncon_for RMCF}, the same estimate holds for $\tau\in[0, \rmT']$.
		
		To prove \eqref{Equ_AsympProf_N_u(M)> mu - eps}, first note that there exist an $\eps'(n, \eps)\in (0, \eps)$ and $\hat\mu\in [\mu-\eps/2, \mu-\eps/4]$ such that $\dist_\R(\hat\mu, \sigma(\cC_{n,k}))\geq \eps'$. Let $\delta_{\ref{Lem_AsympProf_Rough C^2 est}}(n, \eps)\ll 1$ and $\rmT'(n,\eps)\gg 1$ such that when $\tau\geq \rmT'$, $\|\varphi_\bu(\cdot,\tau)\|_{C^4}$ is sufficiently small so that Corollary \ref{Cor_L^2 Mono_Discrete Growth Mono} can be applied. Suppose for contradiction, there exists $\bar \tau\geq \rmT'$ such that $\cN_\bu(\bar \tau, \cM)< \mu - \eps/2 \leq \hat\mu$. Then by Corollary \ref{Cor_L^2 Mono_Discrete Growth Mono}, for every integer $l\geq 0$, \[
		\cN_\bu(\bar \tau + l, \cM)\leq \cN_\bu(\bar \tau, \cM) \leq \hat\mu \leq \mu - \eps/4 \,.
		\]
		This contradicts to assumption \ref{Item_AsympProf_Rough_mu > lim cN_u} by sending $l\to +\infty$.
		
		Now we bound the $L^2$ distance of the translated flow. For every $\by\in \RR^k$ and $\tau\geq 0$, we denote for simplicity $\by_\tau:= e^{\tau/2}\by$, then recall that $\tau\mapsto \cM^{(0, \by, 0)}(\tau) = \cM(\tau) - (\orig, \by_\tau)$ is another RMCF. Hence whenever $|\by|\leq e^{-\tau/2}$, by applying \eqref{Equ_L^2 Mono_d_u(Sigma - y) < d_u(Sigma) + error} of Lemma \ref{Lem_L^2 Mono_Compare d_u(Sigma) w d_u'(Sigma) and d_u(Sigma + y)} with $R=8(-\ln \mbfd_\bu(\tau, \cM))^{1/2}$,
		\begin{align}
			\mbfd_\bu(\tau, \cM^{(0, \by, 0)}) = \mbfd_\bu(\tau, \cM(\tau)-(0, \by_\tau)) \lesssim_{n,\eps} \mbfd_\bu(\tau, \cM)^{1-\eps''}\,, \label{Equ_AsympProf_d(M^y)<d(M)}
		\end{align}
		where $\eps'':= 1-(\mu-\eps)/(\mu-\eps/2)\leq \eps$ has a positive lower bound depending only on $\eps$.
		
		Now that for every $\tau_\circ \geq 0$, every $\tau \geq \tau_\circ+1$ and $\by\in \RR^k$ with $|\by|\leq e^{-\tau_\circ/2}$, whenever the terms on the RHS are sufficiently small, we have 
		\begin{align}
			\begin{split}
				\mbfd_\bu(\tau-1, \cM^{(0, \by, 0)}) & \lesssim_{n, \eps} e^{(1+\eps)(\tau-\tau_\circ)}\cdot \mbfd_\bu(\tau_\circ, \cM^{(0, \by, 0)}) \\
				& \lesssim_{n, \eps} e^{(1+\eps)(\tau-\tau_\circ)}\cdot \mbfd_\bu(\tau_\circ, \cM)^{1-\eps''} \\
				& \lesssim_{n, \eps} e^{(1+\eps)(\tau-\tau_\circ)}\cdot \left(e^{-(\mu-\eps/2)\tau_\circ}\mbfd_\bu(0, \cM)\right)^{1-\eps''} \\
				& \lesssim_{n, \eps} e^{(1+\eps)\tau-(1+\mu)\tau_\circ}\cdot \mbfd_\bu(0, \cM)^{1-\eps} \,,
			\end{split} \label{Equ_AsympProf_d(tau, M)<e^(tau-(1+mu)l)}
		\end{align}
		where the first inequality follows from Corollary \ref{Cor_L^2 Mono_cN>-1} and Lemma \ref{lem_delta_L2_close_implies_Gaussian_lower_bound}, the second follows from \eqref{Equ_AsympProf_d(M^y)<d(M)} and the third follows from \eqref{Equ_AsympProf_N_u(M)> mu - eps}. 
		In particular, if we take $\breve\tau := (\mu+1)^{-1}\cdot (1+\eps)\tau$, then for every $\tau_\circ\in [\breve\tau, \tau-1]$ and every $|\by|\leq e^{-\tau_\circ/2}$, the inequality above gives, \[
		\mbfd_\bu(\tau-1, \cM^{(0, \by, 0)}) \lesssim_{n, \eps} \mbfd_\bu(0, \cM)^{1-\eps} \,.
		\]
		By Lemma \ref{lem_delta_L2_close_implies_Gaussian_lower_bound}, Lemma \ref{Lem_L^2 clos => C^2 close} and taking $\delta_{\ref{Lem_AsympProf_Rough C^2 est}}(n, \eps)\ll 1$, this means $\cM^{(0, \by, 0)}(\tau)=\cM(\tau)-e^{\tau/2}\by$ is $C^2$ graphical over $\cC_{n,k}$ in $Q_{2n}$, and hence $\cM(\tau)$ is a $C^2$ graph over $\cC_{n,k}$ in $Q_{2n}(0, e^{\tau/2}\by)$, with the graphical function $v(\cdot, \tau)$ satisfying \[
		\|(v - \varphi_\bu)(\cdot, \tau)\|_{C^2(\cC_{n,k}\cap Q_{2n}(0, e^{\tau/2}\by))} 
		\lesssim_{n, \eps} e^{(1+\eps)\tau - (1+\mu)\tau_\circ}\cdot \mbfd_\bu(0, \cM)^{1-\eps} \,.
		\]In view of the choice of admissible $\by$, this proves that for every $\tau \geq 1$, every $2n\leq R\leq 2n + e^{(\tau - \breve\tau)/2} =: \rmR_{\mu, \eps}(\tau)$, $\cM(\tau)$ is a $C^2$ graph over $\cC_{n,k}$ in $B^{n-k+1}_{2n}\times B^k_R$, with estimate on graphical function $v(\cdot, \tau)$:
		\begin{align*}
			\|(v-\varphi_\bu)(\cdot, \tau)\|_{C^2(\cC_{n,k}\cap Q_R)} \lesssim_{n, \eps} e^{-(\mu-\eps)\tau}R^{2(1+\mu)}\cdot \mbfd_\bu(0, \cM)^{1-\eps} \,. 
		\end{align*}
		
		To finish the proof of \eqref{Equ_AsympProf_Rough_Gfrad est} and \eqref{Equ_AsympProf_Rough Spacetime C^2 est for u}, it suffices to establish the following: 
		\begin{align}
			\cM(\tau)\cap Q_{\rmR_{\mu, \eps}(\tau)}\setminus \BB_{2n}^{n-k+1}\times \RR^k = \emptyset \,. \label{Equ_AsympProf_M(tau) cap Q_R subset Q_2n}
		\end{align}
		To see this, for every $T > 1$, note that by applying \eqref{Equ_AsympProf_d(tau, M)<e^(tau-(1+mu)l)} again with $\tau_\circ = \breve T := (1+\mu)^{-1}\cdot (1+\eps)T$, we have for every $|\by|\leq e^{-\breve T/2}$ and every $\tau\in [\breve T, T]$, \[
		\mbfd_\bu(\tau, \cM^{(0, \by, 0)}) \lesssim_{n, \eps} e^{(1+\eps)\tau - (1+\mu) \breve T} \cdot \mbfd_\bu(0, \cM)^{1-\eps} \lesssim_{n, \eps} \mbfd_\bu(0, \cM)^{1-\eps} \,;
		\]
		and the same holds when $0\leq \tau\leq \breve T$ by Lemma \ref{Lem_L^2 Mono_Compare d_u(Sigma) w d_u'(Sigma) and d_u(Sigma + y)} since $|e^{\tau/2}\by|\leq 1$ for such $\tau$. By taking $\delta_{\ref{Lem_AsympProf_Rough C^2 est}}(n, \eps)\ll 1$, this implies that $\cM^{(0, \by, 0)}$ is $\delta_{\ref{Lem_Prelim_M cap Q_R in Q_2n}}(n, \eps)$-close to $\cC_{n, k}$ over $[0, T]$. Thus by Lemma \ref{Lem_Prelim_M cap Q_R in Q_2n}, \[
		(\spt\cM^{(0, \by, 0)})(T) \cap \left(B^{n-k+1}_{2n+e^{\tau/2}}\times B^k_{2n} \right) \subset Q_{2n}.
		\]
		By ranging $\by\in \R^k$ with constraint $|\by|\leq e^{-\breve T/2}$, this proves \eqref{Equ_AsympProf_M(tau) cap Q_R subset Q_2n}.
	\end{proof}
	
	We now turn the rough estimate \eqref{Equ_AsympProf_Rough Spacetime C^2 est for u} to one with a precise time asymptotic. This is achieved by solving an inhomogeneous parabolic equation with a pointwise estimate, which we discuss in the Appendix.
	
	\begin{Lem} \label{Lem_AsympProf_Precise C^0 est}
		Let $\eps, \delta\leq \delta_{\ref{Lem_AsympProf_Rough C^2 est}}, \cM, \bu, \mu$ be the same as in Lemma \ref{Lem_AsympProf_Rough C^2 est}. Suppose further that
		\begin{align}
			\lim_{\tau\to +\infty} \cN_\bu(\tau, \cM) = \mu\,, & & 
			\mu^+ := \min\left\{\mu',\ \mu+\gamma_{n,k},\ 2\mu \right\}>\mu+4\eps\,. \label{Equ_AsympProf_Imprv_cN to mu}
		\end{align}
		where $\mu':= \min\sigma(\cC_{n,k})\cap (\mu, +\infty)$. 
		Then $\mu\in \sigma(\cC_{n,k})$, and there's a unique nonzero eigenfunction $\psi\in \rmW_{\mu}(\cC_{n,k})$ of $-L_{n,k}$ such that the following refined estimate of graphical function $v(\cdot, \tau)$ of $\cM(\tau)$ over $\cC_{n,k}$ holds: for every $\tau\geq 1$ and $2n\leq R \leq \gfrd_1(\cM(\tau))$,
		\begin{align}
			\|(v-\varphi_\bu)(\cdot, \tau) - e^{-\mu\tau}\psi\|_{C^0(\cC_{n,k}\cap Q_R)} \lesssim_{n, \eps} e^{-(\mu^+-3\eps)\tau}R^{2(2+\mu^+)} \,.    
			\label{Equ_AsympProf_Precise C^2 est}
		\end{align}
		
		In particular, $\exists\, \delta_{\ref{Lem_AsympProf_Precise C^0 est}}(n, \eps)\in (0, \delta_{\ref{Lem_AsympProf_Rough C^2 est}}]$ such that if $\cM$ further satisfies $\mbfd_\bu(0, \cM)\leq \delta_{\ref{Lem_AsympProf_Precise C^0 est}}$, then
		\begin{align}
			\|\psi\|_{L^2} \lesssim_{n, \eps} \mbfd_\bu(0, \cM)^{(1-\eps)/2} \leq 1\,; \label{Equ_AsympProf_Precise |psi|_L^2 bd}
		\end{align}
		and for every $\tau\geq 1$,
		\begin{align}
			\|(v-\varphi_\bu)(\cdot, \tau)\|_{L^2(\cC_{n,k}\cap Q_{\gfrd_1(\cM(\tau))})} + \mbfd_\bu(\tau, \cM) \lesssim_{n, \eps} e^{-\mu\tau}\mbfd_\bu(0, \cM)^{\delta_{\ref{Lem_AsympProf_Precise C^0 est}}}\,.  \label{Equ_AsympProf_Precise d_bu(tau) bd}
		\end{align}
		
	\end{Lem}
	
	\begin{Rem}
		By applying the RMCF equation to $v-\varphi_\bu$, one can also derive a $C^{k,\alpha}$ bound on $(v-\varphi_\bu)(\cdot, \tau) - e^{-\mu\tau}\psi$ of similar form as \eqref{Equ_AsympProf_Precise C^2 est}. But this is out of the scope of this paper. 
	\end{Rem}
	\begin{proof}[Proof of Lemma \ref{Lem_AsympProf_Precise C^0 est}.]
		Recall $\rmR_{\mu,\eps}(\tau):= 2n + e^{\frac{\mu-\eps}{2(\mu+1)}\cdot \tau}$ denotes the lower bound on graphical radius as in Lemma \ref{Lem_AsympProf_Rough C^2 est}.
		Also note that there exist $\eps'\in (0, \eps)$ depending on $\mu$, $\eps''=\eps''(n, \eps)\in (0, \eps)$, and $\hat\mu\in [\mu^+-3\eps, \mu^+-2\eps]$ such that
		\begin{align}
			\dist_\R(\hat\mu, \sigma(\cC_{n,k}))\geq \eps''\,, & &
			[\mu-2\eps', \mu+2\eps']\cap \sigma(\cC_{n,k}) \subset \{\mu\}\,.  \label{Equ_AsympProf_Imprv_hat mu}
		\end{align}
		
		Let $\xi\in C^\infty(\RR)$ be a cut-off function with $\xi|_{(-\infty, 1/2]}=1$, $\xi|_{[1, +\infty)}=0$ and $|\xi'|, |\xi''|\leq 2025$; let \[
		\tilde\xi(\theta, y, \tau):= \xi(\rmR_{\mu,\eps}(\tau)^{-1}|y|), \qquad
		\tilde v(\cdot, \tau):= (v-\varphi_\bu)(\cdot, 1+\tau)\tilde\xi\,.
		\]
		
		Note that by \eqref{Equ_AsympProf_Imprv_cN to mu}, $|\cN_\bu(\tau, \cM) - \mu|<\eps'$ for $\tau \gg 1$, hence by Lemma \ref{Lem_App_Graph over Cylinder} \ref{Item_GraphCyliner_d_u vs |v-varphi_u|_L^2} and the exponential growth of $\rmR_{\mu,\eps}(\tau)$, we have
		\begin{align}
			\begin{split}
				& \limsup_{\tau\to +\infty} e^{(\mu-\eps')\tau}\|\tilde v(\cdot, \tau)\|_{L^2} <+\infty\,; \\
				&\liminf_{\tau\to +\infty}\ e^{(\mu+\eps')\tau}\|\tilde v(\cdot, \tau)\|_{L^2} \geq  \liminf_{\tau\to +\infty}\ e^{(\mu+\eps')\tau} \left(\mbfd_\bu(\tau, \cM) - \mbfd_\bu(\tau, \cM(\tau)\setminus Q_{\rmR_{\mu,\eps}(\tau)/2})) \right) > 0\,. 
			\end{split} \label{Equ_AsympProf_v sim e^(-mu tau)}
		\end{align}
		\begin{Claim}
			$\tilde\cQ:= (\partial_\tau - L_{n,k}) \tilde v$ satisfies the pointwise estimate, 
			\begin{align}
				|\tilde\cQ(\cdot, \tau)| \lesssim_{n, \eps} e^{-\hat\mu\tau}(1+|y|^2)^{\hat\mu+2+2\eps} \,. \label{Equ_AsympProf_Est |cQ|}
			\end{align}
		\end{Claim}
		\begin{proof}[Proof of the Claim.]
			We denote for simplicity $|J^2v|:=|v|+|\nabla v|+|\nabla^2 v|$. By \cite[Lemma B.2]{SunWangXue1_Passing}, \[
			\tilde\cQ = \tilde\xi\cdot (\partial_\tau - L_{n,k})(v-\varphi_\bu) + (v-\varphi_\bu)\cdot (\partial_\tau - L_{n,k}+1)\tilde\xi - 2\nabla \tilde\xi\cdot \nabla (v-\varphi_\bu)
			\]
			satisfies the pointwise estimate (we use $\Id_E$ to denote the indicator function of $E$), 
			\begin{align*}
				|\tilde\cQ(\cdot, \tau)| & \lesssim_n \Big( |J^2(v-\varphi_\bu)(\cdot, \tau)| + |J^2\varphi_\bu(\cdot, \tau)| + \Id_{\{1/2\leq |y|/\rmR_{\mu, \eps}(\tau)\leq 1\}} \Big) |J^2(v-\varphi_\bu)(\cdot, \tau)|\cdot \Id_{Q_{\rmR_{\mu,\eps}(\tau)}}
			\end{align*}
			To derive \eqref{Equ_AsympProf_Est |cQ|}, we shall bound each term on the RHS above. By the definition of $\rmR_\tau$, 
			\begin{align*}
				\Id_{\{\rmR_{\mu,\eps}(\tau)/2\leq |y|\leq \rmR_{\mu,\eps}(\tau)\}} \lesssim_{n,\eps} e^{-(\mu-\eps)\tau} \cdot(1+|y|^2)^{\mu+1} \Id_{Q_{\rmR_{\mu,\eps}(\tau)}} \lesssim_{n, \eps} 1 \,;
			\end{align*}
			Since $\bu\in \scU$, by item \ref{item_2_Lem_L^2Noncon_LowSphericalMode} of Lemma \ref{Lem_L^2Noncon_LowSphericalMode}, 
			\begin{align*}
				|J^2\varphi_\bu(\cdot, \tau)| \lesssim_n e^{-\gamma_{n,k}\tau} \,; 
			\end{align*}
			while by Lemma \ref{Lem_AsympProf_Rough C^2 est}, 
			\begin{align*}
				|J^2(v-\varphi_\bu)(\cdot, \tau)|\cdot \Id_{Q_{\rmR_{\mu,\eps}(\tau)}} \lesssim_{n, \eps} e^{-(\mu-\eps)\tau} (1+|y|^2)^{\mu+1} \Id_{Q_{\rmR_{\mu,\eps}(\tau)}} \lesssim_{n, \eps} 1\,.
			\end{align*}
			Also by our choice of $\hat\mu$, it's easy to check that $\frac{\hat\mu}{\mu-\eps}\leq 2$ and $\frac{(\mu+1)\hat\mu}{\mu-\eps}\leq \hat\mu+2+2\eps$. Hence
			\begin{align*}
				\left(e^{-(\mu-\eps)\tau} (1+|y|^2)^{\mu+1}\cdot \Id_{Q_{\rmR_{\mu,\eps}(\tau)}}\right)^2 & \lesssim_{n, \eps}  \left(e^{-(\mu-\eps)\tau} (1+|y|^2)^{\mu+1}\cdot \Id_{Q_{\rmR_{\mu,\eps}(\tau)}}\right)^{\frac{\hat\mu}{\mu-\eps}} \\
				& \lesssim_{n, \eps} e^{-\hat\mu\tau}(1+|y|^2)^{\hat\mu+2+2\eps} \,;
			\end{align*}
			
			\begin{align*}
				e^{-\gamma_{n,k}\tau}\cdot \left(e^{-(\mu-\eps)\tau} (1+|y|^2)^{\mu+1}\cdot \Id_{Q_{\rmR_{\mu,\eps}(\tau)}}\right) \leq e^{-\hat\mu\tau}(1+|y|^2)^{\hat\mu+2+2\eps}\,.
			\end{align*}
			Combining the estimates above finishes the proof of pointwise estimate \eqref{Equ_AsympProf_Est |cQ|}. 
		\end{proof}
		
		Now recall by our assumption \eqref{Equ_AsympProf_Imprv_hat mu}, $\hat\mu\in \RR\setminus\BB_{\eps''}(\sigma(\cC_{n,k}))$. Hence by \eqref{Equ_AsympProf_Est |cQ|} and Proposition \ref{Prop_App_Solve Parab Jac equ w spacetime C^0 est}, there's a solution to \[
		(\partial_\tau - L_{n,k})w = \tilde\cQ
		\]
		with the following pointwise spacetime $C^0$ estimate, 
		\begin{align}
			|w(\cdot, \tau)| \lesssim_{n,\eps} e^{-\hat\mu\tau}(1+|y|^2)^{\hat\mu+2+2\eps}\,. \label{Equ_AsympProf_Est purely nonlinear |w|}
		\end{align}
		Then notice that $(\partial_\tau - L_{n,k})(\tilde v-w)=0$, and combining \eqref{Equ_AsympProf_Est purely nonlinear |w|} with \eqref{Equ_AsympProf_v sim e^(-mu tau)}, we have
		\begin{align*}
			\limsup_{\tau\to +\infty} e^{(\mu-\eps')\tau}\|(\tilde v-w)(\cdot, \tau)\|_{L^2} < +\infty\,, & &
			\liminf_{\tau\to +\infty} e^{(\mu+\eps')\tau}\|(\tilde v-w)(\cdot, \tau)\|_{L^2} > 0\,.
		\end{align*}
		While by equation \eqref{Equ_AsympProf_Rough Spacetime C^2 est for u} of Lemma \ref{Lem_AsympProf_Rough C^2 est} and \eqref{Equ_AsympProf_Est purely nonlinear |w|}, the following pointwise estimate hold at $\tau=0$: 
		\begin{align*}
			|(\tilde v - w)(\cdot, 0)| \lesssim_{n, \eps} (1+|y|^2)^{\hat\mu+2+2\eps} \,.
		\end{align*}
		Recall again by our assumption \eqref{Equ_AsympProf_Imprv_hat mu}, $[\mu-\eps', \mu+\eps']\cap \sigma(\cC_{n,k})\subset \{\mu\}$, $\mu+\eps\leq \hat\mu\leq \mu'-\eps$. Then by applying Proposition \ref{Prop_App_Parab equ Spacetime C^0 bd from initial data} with $\lambda=\hat\mu, \lambda':=\hat\mu+2+2\eps$ we get, $\mu\in \sigma(\cC_{n,k})$, $\psi:= \Pi_{=\mu}((\tilde v - w)(\cdot, 0))$ is non-zero and the following pointwise bound holds, 
		\begin{align*}
			|(\tilde v - w)(\cdot, \tau) - e^{-\mu\tau}\psi| \lesssim_{n, \eps} e^{-\hat\mu\tau}(1+|y|^2)^{\hat\mu+2+2\eps} \leq e^{-(\mu^+-3\eps)\tau}(1+|y|^2)^{2+\mu^+}\,.
		\end{align*}
		Combining this with \eqref{Equ_AsympProf_Est purely nonlinear |w|} proves \eqref{Equ_AsympProf_Precise C^2 est} for $2n\leq R\leq \rmR_{\mu,\eps}(\tau)$.
		While if $\rmR_{\mu,\eps}(\tau)\leq R\leq \gfrd_1(\cM(\tau))$, we have (by definition of graphical radius),
		\begin{align*}
			\|(v - \varphi_\bu)(\cdot, \tau) - e^{-\mu\tau}\psi\|_{C^0(\cC_{n,k}\cap Q_R)} & \lesssim_{n, \eps} \|(v - \varphi_\bu)(\cdot, \tau)\|_{C^0(Q_R)} + e^{-\mu\tau}\|\psi\|_{C^0(Q_R)} \\
			& \lesssim_{n, \eps} 1 + e^{-\mu\tau}R^{2\mu+2} 
			\lesssim_{n, \eps} e^{-(\mu^+-3\eps)\tau}R^{2(2+\mu^+)}.
		\end{align*}
		This proves \eqref{Equ_AsympProf_Precise C^2 est} for such $R$.
		
		Finally, by combining \eqref{Equ_AsympProf_Rough Spacetime C^2 est for u} and \eqref{Equ_AsympProf_Precise C^2 est}, for every $\tau\geq 1$ we have \[
		\|\psi\|_{L^2(\cC_{n,k}\cap Q_{\rmR_{\mu,\eps}(\tau)})} \lesssim_{n,\eps} e^{-(\mu^+-\mu-3\eps)\tau} + e^{\eps\tau}\mbfd_\bu(0, \cM)^{1-\eps} \lesssim_{n,\eps} e^{-\eps\tau} + e^{\eps\tau}\mbfd_\bu(0, \cM)^{1-\eps}.
		\]
		In particular, when $\mbfd_\bu(0, \cM)\leq \delta_{\ref{Lem_AsympProf_Precise C^0 est}}(n, \eps)\ll 1$, by applying Corollary \ref{Cor_L^2 Mono_Low Bd Upsilon_(u,u'; w)} with $\bu=\bu'$ and $w=\psi$, taking $\tau:=\frac{\eps-1}{2\eps}\cdot\ln \mbfd_\bu(0,\cM)$ yields estimate \eqref{Equ_AsympProf_Precise |psi|_L^2 bd} on $\|\psi\|_{L^2}$.
		
		To prove \eqref{Equ_AsympProf_Precise d_bu(tau) bd}, first note that by item \ref{Item_GraphCyliner_d_u vs |v-varphi_u|_L^2} of Lemma \ref{Lem_App_Graph over Cylinder} and Lemma \ref{Lem_L^2 Noncon_for RMCF}, 
		\begin{align*}
			\|(v-\varphi_\bu)(\cdot, \tau)\|_{L^2(\cC_{n,k}\cap Q_{\gfrd_1(\cM(\tau))})} \lesssim_{n, \eps} \mbfd_\bu(\tau, \cM) \lesssim_{n, \eps} e^{-\mu\tau} e^{(K_n+\mu)\tau}\mbfd_\bu(0, \cM).
		\end{align*}
		This implies \eqref{Equ_AsympProf_Precise d_bu(tau) bd} if $\tau\geq 1$ such that $e^{(K_n+\eps^{-1})\tau}\leq \mbfd_\bu(0, \cM)^{-1/2}$.
		
		While if $e^{(K_n+\eps^{-1})\tau}\geq \mbfd_\bu(0, \cM)^{-1/2}$, the estimate \eqref{Equ_AsympProf_Precise |psi|_L^2 bd} with \eqref{Equ_AsympProf_Rough_Gfrad est} on $\psi$, estimate \eqref{Equ_AsympProf_Precise C^2 est} on $|v-\varphi_\bu-e^{-\mu\tau}\psi|$ and item \ref{Item_GraphCyliner_d_u vs |v-varphi_u|_L^2} of Lemma \ref{Lem_App_Graph over Cylinder} together yields, 
		\begin{align*}
			\|(v-\varphi_\bu)(\cdot, \tau)\|_{L^2(\cC_{n,k}\cap Q_{\gfrd_1(\cM(\tau))})} + \mbfd_\bu(\tau, \cM) 
			& \lesssim_{n, \eps} e^{-\mu\tau}\|\psi\|_{L^2} + e^{-(\mu^+-3\eps)\tau} \\
			& \lesssim_{n, \eps} e^{-\mu\tau}\left(\mbfd_\bu(0,\cM)^{(1-\eps)/2} + e^{-\eps\tau} \right) \,,
		\end{align*}
		which proves \eqref{Equ_AsympProf_Precise d_bu(tau) bd} if $\delta_{\ref{Lem_AsympProf_Precise C^0 est}}$ is chosen smaller.
	\end{proof}

	\subsection{Find optimal low spherical flow} \label{Subsec_AsympProf_gamma>0}    
	We shall prove Theorem \ref{Thm_Asymp Profile(New)_gamma>0} in this subsection by finding the low spherical flow that best approximates our RMCF.
	\begin{proof}[Proof of Theorem \ref{Thm_Asymp Profile(New)_gamma>0}]
		When $n-k=1$, by definition, the only low spherical flow is $\mbfU_{n,k}$. Hence \eqref{Equ_AsympProf_Main_cN(infty) in sigma(cC)} follows from Corollary \eqref{Cor_L^2 Mono_Def cN_u(infty)}; \eqref{Equ_AsympProf_Main_cN_scU(infty) = cN_u(infty)} is vacuum; \eqref{Equ_AsympProf_Main_Gfrad bd}, \eqref{Equ_AsympProf_Main_Rough C^2 Bd} and \eqref{Equ_AsympProf_Main_Precise C^0 Bd} follow from Lemmas \ref{Lem_AsympProf_Rough C^2 est} and \ref{Lem_AsympProf_Precise C^0 est}. From now on, let's assume $n-k\geq 2$.
		Recall that $\gfrd_1(\cM(\tau))$ and $v(\cdot, \tau)$ denote the $1$-graphical radius and graphical function of $\cM(\tau)$ over $\cC_{n,k}$. We may also assume without loss of generality that $\lambda[\cM]\leq \eps^{-1}$.

		Let $\gamma_0:=\cN_{n,k}(\infty, \cM)>0$. Note that by Corollary \ref{Cor_L^2 Mono_Def cN_u(infty)}, $\gamma_0\in \sigma(\cC_{n,k})$, hence $\gamma_0\geq \gamma_{n,k}$. 
		When $\gamma_0=\gamma_{n,k}$, by taking $\delta\leq \min\{\delta_{\ref{Lem_AsympProf_Rough C^2 est}}, \delta_{\ref{Lem_AsympProf_Precise C^0 est}}\}$ and applying Lemma \ref{Lem_AsympProf_Rough C^2 est} \eqref{Equ_AsympProf_Rough_Gfrad est} and Lemma \ref{Lem_AsympProf_Precise C^0 est} \eqref{Equ_AsympProf_Precise C^2 est} to $\cM$ with $\bu=\mbfU_{n,k}$, there's a unique nonzero eigenfunction $\psi_0\in \rmW_{\gamma_{n,k}}(\cC_{n,k})$ such that for every $\tau\gg 1$,
		\begin{align}
			\|v(\cdot, \tau)-e^{-\gamma_{n,k}\tau}\psi_0\|_{L^2} \lesssim_{n, \eps} e^{-(2\gamma_{n,k}-3\eps)\tau}\,, & &
			\gfrd_1(\cM(\tau))\geq 2n+e^{\frac{\gamma_{n,k}-\eps}{1+\gamma_{n,k}}\cdot\tau}\,.  \label{Equ_AsympProf_Pf_v sim e^(-gamma_0 tau) psi_0}
		\end{align}
		By Lemma \ref{Lem_AsympProf_Rough C^2 est} \eqref{Equ_AsympProf_Rough Spacetime C^2 est for u}, this is still true if $\gamma_0>\gamma_{n,k}$, with $\psi_0=0$. 
		Let $\psi_\SSp\in \rmW_\SSp$ (possibly $0$) be the $L^2$-orthogonal projection of $\psi_0$ onto $\rmW_\SSp$. Then when $\psi_\SSp\neq 0$, we must have $\gamma=\gamma_{n,k}$ and by Lemma \ref{Lem_AsympProf_Precise C^0 est} \eqref{Equ_AsympProf_Precise |psi|_L^2 bd}, we have \[
		\|\psi_\SSp\|_{L^2} \leq \|\psi_0\|_{L^2} \lesssim_{n,\eps} \mbfd_{n,k}(0, \cM)^{(1-\eps)/2} \,.
		\]
		Hence by taking $\delta_{\ref{Thm_Asymp Profile(New)_gamma>0}}(n, \eps)\ll 1$, we can apply Lemma \ref{Lem_L^2Noncon_LowSphericalMode} and find a unique low spherical flow $\bu\in \scU$ such that 
		\begin{align}
			\|\varphi_\bu(\cdot, \tau) - e^{-\gamma_{n,k}\tau}\psi_\SSp\|_{L^2} \lesssim_n e^{-\frac32 \gamma_{n,k}\tau} 
			\,.  \label{Equ_AsympProf_Pf_|varphi_u - psi_S| faster decay}
		\end{align}
		\textbf{Claim.} For every low spherical flow $\bu'\in \scU\setminus \{\bu\}$, we have \[
		\cN_{\bu'}(\infty, \cM) = \gamma_{n,k} \leq \frac12 \leq \cN_\bu(\infty, \cM) \,.
		\]
		
		We first finish the proof of Theorem \ref{Thm_Asymp Profile(New)_gamma>0} assuming this Claim. \eqref{Equ_AsympProf_Main_cN(infty) in sigma(cC)} and \eqref{Equ_AsympProf_Main_cN_scU(infty) = cN_u(infty)} follow directly from this Claim. Moreover, let $\eps, \cM, \delta, \gamma, \gamma^+$ be specified in assumption \ref{Item_AsyPrNew_Assum1}, \ref{Item_AsyPrNew_Assum2} of Theorem \ref{Thm_Asymp Profile(New)_gamma>0}, and $\bu\in \scU$ be determined as above. Then by taking $\delta_{\ref{Thm_Asymp Profile(New)_gamma>0}}(n, \eps)\ll 1$, the estimate \eqref{Equ_AsympProf_Main_Gfrad bd} on graphical radius and \eqref{Equ_AsympProf_Main_Rough C^2 Bd} on $C^2$ norm of graphical function follows from Lemma \ref{Lem_AsympProf_Rough C^2 est} with $\gamma$ in place of $\mu$. 
		
		To prove the improved error estimate \eqref{Equ_AsympProf_Main_Precise C^0 Bd}, first note that if $\gamma<\cN_\scU(\infty, \cM)$, then by definition of $\gamma', \gamma^+$ in \ref{Item_AsyPrNew_Assum2} and the fact that $\cN_\scU(\infty, \cM)\in \sigma(\cC_{n,k})\cup \{+\infty\}$, we have \[
		\gamma^+\leq \gamma'\leq \cN_\scU(\infty, \cM) \,.
		\]  
		Hence, we can again apply Lemma \ref{Lem_AsympProf_Rough C^2 est} \eqref{Equ_AsympProf_Rough Spacetime C^2 est for u} with $\gamma^+$ in place of $\mu$. This proves \eqref{Equ_AsympProf_Main_Precise C^0 Bd} with $\psi=0$.
		
		Finally, if $\gamma=\cN_\scU(\infty, \cM)=\cN_\bu(\infty, \cM)$, then by taking $\delta\leq \min\{\delta_{\ref{Lem_AsympProf_Rough C^2 est}}(n, \eps), \delta_{\ref{Lem_AsympProf_Precise C^0 est}}(n, \eps)^{10}\}$ with $\gamma$ in place of $\mu$ and applying \eqref{Equ_AsympProf_Precise C^2 est} of Lemma \ref{Lem_AsympProf_Precise C^0 est}, we derive that there exists a unique nonzero eigenfunction $\psi\in \rmW_\gamma(\cC_{n,k})$ with $\|\psi\|_{L^2}\leq 1$ such that estimate \eqref{Equ_AsympProf_Main_Precise C^0 Bd} holds (with $\gamma^+$ in place of $\mu^+$). Moreover, together with the graphical radius estimate, this implies
		\begin{align}
			\|(v-\varphi_\bu)(\cdot, \tau) - e^{-\gamma\tau}\psi\|_{L^2} \lesssim_{n, \eps} e^{-(\gamma^+-3\eps)\tau} \,.  \label{Equ_AsympProf_Pf_|v-varphi_bu - e^(-gamma tau)psi| faster decay}
		\end{align}
		We are left to show that 
		\begin{align}
			\psi\perp \rmW_\SSp \,.  \label{Equ_AsympProf_psi perp W_S}
		\end{align}
		This is obvious when $\gamma\neq \gamma_{n,k}$ since different eigenspaces are perpendicular to each other. When $\gamma=\gamma_{n,k}$, we need slightly more effort. First note that by the Claim, $\gamma_{n,k}\leq\gamma_0 \leq 1/2 \leq \gamma$, which then forces $\gamma_{n,k}=\gamma_0=\gamma$ and $n-k=2$. While by combining \eqref{Equ_AsympProf_Pf_v sim e^(-gamma_0 tau) psi_0}, \eqref{Equ_AsympProf_Pf_|varphi_u - psi_S| faster decay} and \eqref{Equ_AsympProf_Pf_|v-varphi_bu - e^(-gamma tau)psi| faster decay}, we have 
		\begin{align*}
			& e^{-\gamma\tau}\|\psi_0 - \psi_\SSp - \psi\|_{L^2} \\ 
			& \quad \leq \|v(\cdot, \tau)-e^{-\gamma_{n,k}\tau}\psi_0\|_{L^2} +
			\|\varphi_\bu(\cdot, \tau) - e^{-\gamma_{n,k}\tau}\psi_\SSp\|_{L^2} + 
			\|(v-\varphi_\bu)(\cdot, \tau) - e^{-\gamma\tau}\psi\|_{L^2} \\
			& \quad \lesssim_{n, \eps} e^{-(2\gamma_{n,k}^+-3\eps)\tau} + e^{-\frac32 \gamma_{n,k}\tau} + e^{-(\gamma^+-3\eps)\tau} \lesssim_{n, \eps} e^{-(\gamma+\eps)\tau} \,,
		\end{align*}
		where in the last inequality, we use the fact that $\eps\in (0, 1/4n)$. Sending $\tau\to +\infty$, this shows $\psi_0=\psi_\SSp+\psi$. Recall that $\psi_\SSp$ is the orthogonal projection of $\psi_0$ onto $\rmW_\SSp$, we thus conclude \eqref{Equ_AsympProf_psi perp W_S} and finish the proof when $\cN_{n,k}(\infty, \cM)<+\infty$. 
		
		\begin{proof}[Proof of Claim.]
			For every $\bu'\in \scU$ corresponding to $\psi'\neq \psi_\SSp$ as in Lemma \ref{Lem_L^2Noncon_LowSphericalMode}, by the graphical estimate \eqref{Equ_AsympProf_Pf_v sim e^(-gamma_0 tau) psi_0} and the cut-off estimate Lemma \ref{Lem_Effect of Cutoff}, $\gfrd_1(\cM(\tau))\geq \tau$ when $\tau\gg 1$. Hence,
			\begin{align*}
				& \left| \|(v-\varphi_{\bu'})(\cdot, \tau)\|_{L^2(Q_\tau)} - \|e^{-\gamma_{n,k}\tau}\psi_0 - e^{-\gamma_{n,k}\tau}\psi'\|_{L^2} \right| \\
				& \qquad \lesssim_n \|v(\cdot, \tau)-e^{-\gamma_{n,k}\tau}\psi_0\|_{L^2(Q_\tau)} + \|\varphi_{\bu'}(\cdot, \tau)-e^{-\gamma_{n,k}\tau}\psi'\|_{L^2} + e^{-\tau^2/16} \\
				& \qquad \lesssim_{n, \eps} e^{-(2\gamma_{n,k} -3\eps)\tau} + e^{-\frac32\gamma_{n,k}\tau} + e^{-\tau^2/16} \lesssim_{n, \eps} e^{-(\gamma_{n,k}+\eps)\tau} \,.
			\end{align*}
			Here the second inequality follows from \eqref{Equ_AsympProf_Pf_v sim e^(-gamma_0 tau) psi_0} and \eqref{Equ_AsympProf_Pf_|varphi_u - psi_S| faster decay}. 
			
			While by combining with Lemma \ref{Lem_App_Graph over Cylinder} \ref{Item_GraphCyliner_d_u vs |v-varphi_u|_L^2}, when $\tau\gg 1$, we derive
			\begin{align*}
				\mbfd_{\bu'}(\tau, \cM) & = (1+o(1))\|(v-\varphi_{\bu'})(\cdot, \tau)\|_{L^2(Q_\tau)} + \mbfd_{\bu'}(\tau, \cM(\tau)\setminus Q_\tau) \\
				& = (1+o(1))e^{-\gamma_{n,k}\tau}\|\psi_0-\psi'\|_{L^2} \,.
			\end{align*}
			Here we use the facts that $\mbfd_{\bu'}(\tau, \cM(\tau)\setminus Q_\tau) = O(e^{-\tau^2/16})$ since $\lambda[\cM]<+\infty$, and $\|\psi_0-\psi'\|_{L^2}\neq 0$ since $\psi'\neq \psi_\SSp$. Therefore,
			\begin{align*}
				\cN_{\bu'}(\infty, \cM) = \lim_{\tau\to +\infty} \ln\left( \frac{\mbfd_{\bu'}(\tau, \cM)}{\mbfd_{\bu'}(\tau+1, \cM)}\right) 
				= \gamma_{n,k} \,.
			\end{align*}
			This proves the first equality in the Claim.
			
			To prove the last inequality in the Claim, we may assume WLOG that $\cM$ doesn't coincide with the low spherical flow given by $\bu$. We first argue that $\cN_\bu(\infty, \cM)>0$. Suppose for contradiction that this is not true, then let $\breve\gamma:= \gamma_{n,k}/2$, we know that for some $\tau_\circ\gg 1$, there holds $\cN_\bu(\tau, \cM)\leq \breve\gamma$ for every $\tau\geq \tau_\circ$. 
			Then for every integer $\ell\geq 1$, 
			\begin{align*}
				\mbfd_{n,k}(\tau_\circ+ \ell, \cM) 
				& \geq \mbfd_\bu(\tau_\circ+\ell, \cM) - C(\cM, \chi)\|\varphi_\bu(\cdot, \tau_\circ+\ell)\|_{C^0} \\
				& \geq \mbfd_\bu(\tau_\circ, \cM)e^{-\breve\gamma\ell} - C(\cM, \chi) e^{-\gamma_{n,k}\ell} \\
				& \gtrsim_{\cM, \tau_\circ} e^{-\breve\gamma\ell}\,.
			\end{align*}
			where the first inequality follows from \eqref{Equ_L^2 Mono_|d_u' - d_u| < |varphi_u'-varphi_u|} of Lemma \eqref{Lem_L^2 Mono_Compare d_u(Sigma) w d_u'(Sigma) and d_u(Sigma + y)}, the second inequality follows from $\cN_\bu\leq \breve\gamma$ and Lemma \ref{Lem_L^2Noncon_LowSphericalMode}, and the last inequality holds when $\ell\gg 1$ since $\mbfd_\bu(\tau_\circ, \cM)\neq 0$. Taking $\ell\to \infty$, this contradicts to that $\cN_{n,k}(\infty, \cM)\geq \gamma_{n,k}$.
			
			Finally, we shall argue that if $\cN_\bu(\infty, \cM) = \gamma_{n,k}$, then $\gamma_{n,k}=1/2$. Note that since $1/2$ is the second largest positive eigenvalue in $\sigma(\cC_{n,k})$, this will finish the proof of the last inequality in the Claim. 
			Since $\cN_\bu(\infty, \cM)=\gamma_{n,k}$, by taking $\delta_{\ref{Thm_Asymp Profile(New)_gamma>0}}(n, \eps)\ll 1$ and applying Lemma \ref{Lem_AsympProf_Rough C^2 est} and \ref{Lem_AsympProf_Precise C^0 est} to $\cM, \bu$, $\mu=\cN_\bu(\infty, \cM)(=\gamma_{n,k})$, we see that \[
			\|(v-\varphi_\bu)(\cdot, \tau) - e^{-\gamma_{n,k}\tau}\psi_1\|_{L^2} \lesssim_{n, \eps} e^{-(2\gamma_{n,k}-3\eps)\tau}
			\]  
			for some nonzero eigenfunction $\psi_1\in \rmW_{\gamma_{n,k}}(\cC_{n,k})$. Combining this with \eqref{Equ_AsympProf_Pf_v sim e^(-gamma_0 tau) psi_0} and \eqref{Equ_AsympProf_Pf_|varphi_u - psi_S| faster decay} gives \[
			e^{-\gamma_{n,k}\tau} \|\psi_0-\psi_\SSp-\psi_1\|_{L^2} \lesssim_{n, \eps} e^{-(2\gamma_{n,k}-3\eps)\tau} + e^{-\frac32\gamma_{n,k}\tau} \lesssim_{n, \eps} e^{-(\gamma_{n,k}+\eps)\tau} \,.  
			\]
			Sending $\tau\to +\infty$ we derive $\psi_0=\psi_\SSp + \psi_1$. Since $\psi_\SSp$ is the orthogonal projection of $\psi_0$ onto $\rmW_\SSp$ and $\psi_1\neq 0$, we conclude that $\rmW_{\gamma_{n,k}}(\cC_{n,k})\setminus \rmW_\SSp\neq \emptyset$. By the discussion in Section \ref{SSubsec_Spectrum of L}, this forces $n-k=2$ and therefore $\gamma_{n,k}=1/2$.
		\end{proof}
		
	\end{proof}

	\subsection{Proof of Theorem \ref{thm:AsyProfile}} \label{Subsec_Pf Main AsympProf}
	We finish this section by proving our main theorem of asymptotic profile. We may assume without loss of generality that $\cM(\tau)\neq \cC_{n,k}$, otherwise case \ref{item_Intro_SuperExpDecay} holds automatically with $\varphi_\bu\equiv0$. Let $\tau\mapsto \cM(\tau)$ be the RMCF specified in Theorem \ref{thm:AsyProfile} with $1$-graphical radius $\gfrd_1(\cM(\tau))$ and graphical function $v(\cdot, \tau)$ over $\cC_{n,k}$. By Corollary \ref{Cor_L^2 Mono_Def cN_u(infty)}, $\cN_{n,k}(\infty, \cM)\in (\sigma(\cC_{n,k})\cap \R_{\geq 0})\cup\{+\infty\}$. 
	
	When $\cN_{n,k}(\infty, \cM)=0$, to prove the desired asymptotic formula as in item \ref{item_Intro_PolynDecay} of Theorem \ref{thm:AsyProfile}, in terms of \cite[Theorem 1.2]{SunXue2022_generic_cylindrical}, we only need to show that $\cI\neq \emptyset$. We prove this by contradiction: suppose $\cI=\emptyset$. Then \cite[Proposition 7.2]{SunXue2022_generic_cylindrical} shows that there exist a dimensional constant $K_0(n)>0$ and $\epsilon>0$ (possibly depends on $\cM$ a priori) such that for $\tau\gg 1$, $\cM(\tau)$ is a graph of function $v(\cdot,\tau)$ over $\cC_{n,k}$ in $Q_{K_0\sqrt{\tau}}$, $v\neq 0$ and \[
	\lim_{\tau\to +\infty} e^{\epsilon\tau} \|v(\cdot, \tau)\|_{L^2(\cC_{n,k}\cap Q_{K_0\sqrt{\tau}})} = 0 \,.
	\] 
	Hence by Proposition \ref{prop:entropy control outside ball of radius R} and item \ref{Item_GraphCyliner_d_u vs |v-varphi_u|_L^2} of Lemma \ref{Lem_App_Graph over Cylinder}, \[
	\lim_{\tau\to +\infty} e^{\epsilon\tau} \mbfd_{n,k}(\tau, \cM) \lesssim_n \limsup_{\tau\to +\infty} e^{\epsilon\tau} \left(\|v(\cdot, \tau)\|_{L^2(\cC_{n,k}\cap Q_{K_0\sqrt{\tau}})} + \cF[\cM(\tau)\setminus Q_{K_0\sqrt\tau}] \right) = 0 \,. 
	\]
	On the other hand, $\cN_{n,k}(\infty, \cM) = 0$ implies that when $\tau\geq \tau_\circ\gg 1$, $\cN_{n,k}(\tau, \cM)<\epsilon$, and hence by definition of decay order, for every integer $l\geq 0$, \[
	\mbfd_{n,k}(\tau_\circ + l, \cM) \geq e^{-\epsilon l} \mbfd_{n,k}(\tau_\circ, \cM) \,.
	\]
	which is a contradiction.
	
	When $\cN_{n,k}(\infty, \cM)>0$ and $\cN_\scU(\infty, \cM)<+\infty$, we take $\gamma:=\cN_\scU(\infty, \cM)$. For every $\eps\in (0, \gamma)$, let $\eps':= \min\{1/4n, \eps/3, 1/(\gamma+1), \lambda[\cM]^{-1}\}$ and take $T_\circ\gg 1$ such that $\cM$ is $\delta_{\ref{Thm_Asymp Profile(New)_gamma>0}}(n, \eps')$-$L^2$  close to $\cC_{n,k}$ over $[T_\circ, +\infty)$. We can then apply Theorem \ref{Thm_Asymp Profile(New)_gamma>0} to $\tau\mapsto \cM(T_\circ+\tau), \gamma, \eps'$ and conclude that there's a unique low spherical flow $\breve\bu\in \scU$ and eigenfunction $\breve\psi\in \rmW_\gamma(\cC_{n,k})$ such that \[
	\gfrd_1(\cM(T_\circ+\tau)) \geq 2n + e^{\frac{\gamma-\eps/3}{2(\gamma+1)}\cdot \tau} \geq e^{\frac{\gamma-\eps}{2(\gamma+1)}\cdot (\tau+T_\circ)}\,;
	\] 
	where the second inequality holds for $\tau\gg 1$; and that for $\tau\geq 1$, \[
	\|v(\cdot, T_\circ+\tau)-\varphi_{\breve\bu}(\cdot, \tau)-e^{-\gamma\tau}\breve\psi\|_{L^2} \lesssim_{n, \eps'} e^{-(\gamma^+-3\eps')\tau} \lesssim_{n, \eps', T_\circ} e^{-(\gamma^+- \eps)(T_\circ+\tau)} \,.
	\]
	This proves case \ref{item_Intro_ExpDecay} of Theorem \ref{thm:AsyProfile} with $\bu$ being such that $\varphi_\bu(\cdot, \tau):=\varphi_{\breve\bu}(\cdot, \tau-T_\circ)$ and $\psi:= e^{\gamma T_\circ}\breve\psi$.
	
	When $\cN_{n,k}(\infty, \cM)>0$ and $\cN_\scU(\infty, \cM)=+\infty$. By \eqref{Equ_AsympProf_Main_cN_scU(infty) = cN_u(infty)} of Theorem \ref{Thm_Asymp Profile(New)_gamma>0}, there exists a unique low spherical flow $\bu$ such that $\cN_\scU(\infty, \cM)=\cN_\bu(\infty, \cM)$. And for every $\lambda>1$, take $\eps_\lambda:= \min\{1/4n, \lambda[\cM]^{-1}, 1/(2\lambda+1)\}$, and let $T_\lambda\gg 1$ be such that $\cM$ is $\delta_{\ref{Thm_Asymp Profile(New)_gamma>0}}(n, \eps_\lambda)$-$L^2$ close to $\cC_{n,k}$ over $[T_\lambda, +\infty)$. We can then apply Theorem \ref{Thm_Asymp Profile(New)_gamma>0} with $\tau\mapsto \cM(T_\lambda + \tau), 2\lambda, \eps_\lambda$ in place of $\cM, \gamma, \eps$ therein to conclude that \[
	\gfrd_1(\cM(T_\lambda+\tau)) \geq 2n + e^{\frac{2\lambda-\eps_\lambda}{2(2\lambda+1)}\cdot \tau} \geq e^{\frac{\lambda}{2(\lambda+1)}\cdot (T_\lambda+\tau)}\,;
	\]
	where the second inequality holds for $\tau\gg 1$; and that for $\tau\geq 1$, \[
	\|(v-\varphi_{\bu})(\cdot, T_\lambda + \tau)\|_{L^2} \lesssim_{n, \eps_\lambda} e^{-(\lambda^+-3\eps_\lambda)\tau} \lesssim_{n, \eps_\lambda, T_\lambda} e^{-\lambda(T_\lambda+\tau)} \,.
	\]
	This proves case \ref{item_Intro_SuperExpDecay} of Theorem \ref{thm:AsyProfile}.

	\section{Estimate of Whitney data} \label{Sec_C^(2,al) Reg}

	Throughout this section, let $\eps\in (0, 1/(2n))$; $t\mapsto \mbfM(t)$ be a MCF in $\RR^{n+1}$ on $[-e, +\infty)$.
	Suppose
	\begin{itemize}
		\item $\lambda[\bM]\leq \eps^{-1}$;
		\item $(\orig, 0)\in \cS_k(\mbfM)_+$, with tangent flow of $\bM$ to be $\cC_{n,k}$;
		\item the RMCF $\tau\mapsto \cM(\tau)$ of $\mbfM$ based at $(\orig,0)$ is $\delta$-$L^2$ close to $\cC_{n,k}$ on $[-1, +\infty)$.
	\end{itemize}
	
	For every $p\in \cS_{k}(\bM)_+$, we associate the following data: 
	\begin{enumerate}[label={\normalfont(\roman*)}]
		\item\label{Item_C^2Reg_TanFlow C_p} $\cC_p\in \Rot(\cC_{n,k})$ be the tangent flow of $\bM$ at $p$, still parametrized by $(\theta, y)$, where $y$ is the coordinate on $\spine(\cC_p)$ and $\theta$ is the coordinate on the spherical part $\cC_p\cap \spine(\cC_p)^\perp$;
		\item\label{Item_C^2Reg_Spine L_p} $\bL_p:\RR^{n+1}\to \RR^{n+1}$ be the linear orthogonal projection onto $\spine(\cC_p) =: |\bL_p|$;
		\item\label{Item_C^2Reg_Quad Q_p} $\bq_p: |\bL_p|\to |\bL_p|^\perp$ be the quadratic polynomial determined by the asymptotic profile $\psi_p\in \rmW_{1/2}(\cC_p)$ as in Remark \ref{Rem_AsympProf_Whitney data}, and denote for later reference $\bQ_p:= \bq_p\circ \bL_p$, which is an $\R^{n+1}$-valued quadratic polynomial on $\R^{n+1}$. 
	\end{enumerate}
	
	We will use $\|\cdot\|$ to denote the operator norm of linear or quadratic $\R^{n+1}$-valued maps from $\R^{n+1}$.
	
	Recall that
	\begin{align}
		\gamma_{n,k}^+ := \min \sigma(\cC_{n,k})\cap (\frac12, +\infty) = \frac12 + \min\{\frac12, \frac1{n-k}\}\,. \label{Equ_C^2 Reg_gamma^+_(n,k)}
	\end{align}
	Also let $\bL_\circ = \{0\}\times \RR^k, \bu_\circ, \bq_\circ, \bQ_\circ$ be the data associated to the spacetime origin $(\orig, 0)$. 
	The goal of this section is the following location estimate on the data $\bar p,\ \bar\bL:= \bL_{\bar p},\ \bar\bu:= \bu_{\bar p},\ \bar\bq:= \bq_{\bar p},\ \bar\bQ:= \bQ_{\bar p}$ when $\bar p\in \cS_{k}(\mbfM)_+$ is close to $(\orig,0)$. 
	\begin{Lem} \label{Lem_C^2 Reg_Main}
		For every $\eps\in (0, 1/(10n))$, there exists $\delta_{\ref{Lem_C^2 Reg_Main}}(n, \eps)\in (0, \eps)$ such that if $\mbfM$ is as above, let $\bar p = (\bar x, \bar y, \bar t)\in \cS_{k}(\mbfM)_+$ (where $\bar x\in |\bL_\circ|^\perp, \bar y\in |\bL_\circ|$) be such that \[
		r:= \max\{|\bar x|, |\bar y|, |\bar t|^{1/2}\}\ \ \in (0, \delta_{\ref{Lem_C^2 Reg_Main}}] \,.
		\]
		Then we have $\|\bQ_\circ\|, \|\bar\bQ\|\leq 1$, $|\bar\bL|=e^{\bar\mbfA}(|\bL_\circ|)$ for some $\bar\mbfA\in \mfk g_{n,k}^\perp$ of form $\bar\mbfA=\begin{bmatrix}
			0 & \bar \bl \\ -\bar \bl^\top & 0
		\end{bmatrix}$ and
		\begin{align}
			r^{-2}|\bar t| + r^{-1}|\bar x - \bq_\circ(\bar y)| +  \|\bar \bl - \nabla \bq_\circ(\bar y)\| + r\|\bar\bQ - \bQ_\circ\| \lesssim_{n,\eps} r^{2\gamma_{n,k}^+-10\eps} \,. \label{Equ_C^2 Reg_Main}
		\end{align}
	\end{Lem}
	Note that in view of the Whitney Extension Theorem, applying this Lemma (possibly with a translation and rotation) to every pair of points in $\cS_{k}(\bM)_+$ that are sufficiently close to each other proves the $C^{2,\al}$-regularity of $\cS_{k}(\bM)_+$. We shall proceed with the details in the next section.

	\begin{proof}
		\textbf{Step 1 (A priori bounds on Whitney data).} 
		Suppose the tangent flow of $\bM$ at $\bar p$ is $\cC'\in \Rot(\cC_{n,k})$. First note that by a compactness argument and Colding-Minicozzi's quantitative uniqueness of cylindrical tangent flow \cite{ColdingMinicozzi25_quantitativeMCF}, $\cC'$ is close to $\cC_{n,k}$ as long as we take $\delta_{\ref{Lem_C^2 Reg_Main}}(n,\eps)\ll 1$, and hence there exists a unique $\bar\mbfA\in \mfk g_{n,k}^\perp$ such that $e^{\bar\mbfA}(|\bL_\circ|) = |\bar\bL|$, with estimates,
		\begin{align}
			r^{-2}|\bar t| + r^{-1}|\bar x| + |\bar\mbfA| \leq \Psi(\delta|n, \eps) \,.  \label{Equ_Improve_A priori bd on |t|,|x|,|A|}
		\end{align}
		And hence, by definition of $r$ (and taking $\delta_{\ref{Lem_C^2 Reg_Main}}(n,\eps)$ smaller), we have,
		\begin{align}
			|\bar y| = r \,. \label{Equ_Improve_A priori bd on |y|}
		\end{align}
		\noindent    \textbf{Step 2 (Estimates from Theorem \ref{Thm_Asymp Profile(New)_gamma>0}).} 
		Let $v(\cdot, \tau),\, \bar v(\cdot, \tau)$ be the graphical function of $\cM(\tau)$ and $\bar\cM(\tau):= (e^{-\bar\mbfA})_\sharp\cM^{\bar p}(\tau)$ over $\cC_{n,k}$ respectively. Note that by taking $\delta_{\ref{Lem_C^2 Reg_Main}}(n, \eps)\ll 1$, assumptions \ref{Item_AsyPrNew_Assum1}, \ref{Item_AsyPrNew_Assum2} in Theorem \ref{Thm_Asymp Profile(New)_gamma>0} are satisfied by $\cM$ and $\bar\cM$ with $\delta\leq \delta_{\ref{Thm_Asymp Profile(New)_gamma>0}}, \gamma=1/2, \eps$, with $\gamma^+=\gamma^+_{n,k}$.
		Hence, by Theorem \ref{Thm_Asymp Profile(New)_gamma>0}, for every $\tau\geq 1$ and every $2n\leq R\leq 2n+e^{\tau/12}$, we have
		\begin{align}
			\|v\|_{C^2(\cC_{n,k}\cap Q_R)} \lesssim_{n,\eps} e^{-(\min\{1/2, 1/(n-k)\}-\eps)\tau} R^3 \,; \label{Equ_Improve_|v| Rough C^2 est}
		\end{align}
		and the asymptotic expansions for $v, \bar v$:
		\begin{align}
			\begin{split}
				v(\cdot,\tau) & = \varphi_{\bu_\circ}(\cdot, \tau) + e^{-\tau/2}\psi_\circ + \cE_\circ(\cdot, \tau)\,, \\
				\bar v(\cdot,\tau) & = \varphi_{\bar\bu}(\cdot, \tau) + e^{-\tau/2}\bar\psi + \bar\cE(\cdot, \tau)\,, \\
				\|\cE_\circ(\cdot, \tau)\|_{C^0(\cC_{n,k}\cap Q_R)} 
				& + \|\bar\cE(\cdot, \tau)\|_{C^0(\cC_{n,k}\cap Q_R)} \lesssim_{n, \eps} e^{-(\gamma^+_{n,k}-\eps)\tau}R^{2(2+\gamma^+_{n,k})} \,;
			\end{split} \label{Equ_Improve_v = varphi_u + psi(tau) +cE_1 w. cE_1 est}
		\end{align}
		for some low spherical flows $\bu_\circ,\ \bar\bu\in \scU$ and eigenfunctions $\psi_\circ,\ \bar\psi \in \rmW_{1/2}(\cC_{n,k})$ with estimate $\|\psi_\circ\|_{L^2}, \|\bar\psi\|_{L^2}\leq 1$. (Note that $\bar\psi$ is related with the Whitney data at $\bar p$ by $\bar\psi=\psi_{\bar p}\circ e^{\bar\mbfA}$).
		
		\noindent    \textbf{Step 3 (Relation between $v$ and $\bar v$).} Recall that \[
		\cM^{\bar p}(\tau) = \sqrt{1-\bar t e^{\tau}}\cdot \cM\left(\tau - \log(1-\bar te^{\tau})\right) - e^{\tau/2}(\bar x, \bar y) \,. 
		\] 
		We shall apply Lemma \ref{Lem_App_Graph over Cylinder} \ref{Item_GraphCylinder_Rot} to compare $v$ and $\bar v$ at $\tau_r:=-2\ln r$ time slice. We set,
		\begin{align}
			\bar\lambda_r = \sqrt{1-\bar t e^{\tau_r}}, & & 
			(\bar\bx_r, \bar\by_r):= (e^{\tau_r/2}\bar x, e^{\tau_r/2}\bar y) \,. \label{Equ_Improve_Def lambda_r, bx_r, by_r}  
		\end{align}
		Note that by \eqref{Equ_Improve_A priori bd on |t|,|x|,|A|}, \eqref{Equ_Improve_A priori bd on |y|} and taking $\delta_{\ref{Lem_C^2 Reg_Main}}(n,\eps)\ll 1$, we have 
		\begin{align*}
			|\bar\lambda_r - 1| \sim r^{-2}|\bar t| \ll 1\,, & &
			|\bar\bx_r| = r^{-1}|\bar x| \ll 1\,, & &
			|\bar\by_r| = 1 \,.
		\end{align*}
		{\bf Claim.} {\it 
			Let $\Lambda = \Lambda_{\ref{Cor_L^2 Mono_Low Bd Upsilon_(u,u'; w)}}(n, 1/2)>2n$. By taking $\delta_{\ref{Lem_C^2 Reg_Main}}(n, \eps)\ll 1$, we have $\gfrd_1(\cM(\tau_r))>4\Lambda$ and, 
			\begin{align}
				\begin{split}
					\bar v(\cdot, \tau_r) = v(\cdot+(0,\bar\lambda_r^{-1}\bar\by_r),\ \tau_r) + \Upsilon_r + \cE'\,,
				\end{split} \label{Equ_Improve_bar v = v + Upsilon + cE_2}
			\end{align}
			in $Q_{\Lambda}$, where $\Upsilon_r$ is the sum of three eigenfunctions of $-L_{n,k}$: 
			\begin{align}
				\Upsilon_r:= \varrho(\bar\lambda_r-1) - \psi_{\bar\bx_r} - \psi_{\bar\mbfA} \,; \label{Equ_Improve_Define Upsilon_r} 
			\end{align}
			and the error estimate holds, 
			\begin{align}
				\|\cE'\|_{L^2(\cC_{n,k}\cap Q_\Lambda)} \lesssim_{n,\eps} \left(r^{2\gamma^+_{n,k}-1-4\eps} + |\bar\bx_r| + |\bar\lambda_r-1| + |\bar\mbfA| \right) \left(|\bar\bx_r| + |\bar\lambda_r-1| + |\bar\mbfA|\right)\,.    
				\label{Equ_Improve_|cE'| est from rot graph}
			\end{align}
		}
		\begin{proof}[Proof of Claim]
			By definition of $\bar\cM$ in \textbf{Step 2}, \[
			\bar\cM(\tau_r) = (e^{-\bar\mbfA})_\sharp \left(\bar\lambda_r\cdot \cM(\tau_r - 2\log\bar\lambda_r) - (\bar\bx_r, \bar\by_r)\right),
			\]
			while by \eqref{Equ_Improve_A priori bd on |t|,|x|,|A|}, \eqref{Equ_Improve_A priori bd on |y|} and \eqref{Equ_Improve_|v| Rough C^2 est},  by taking $\delta_{\ref{Lem_C^2 Reg_Main}}(n,\eps)\ll 1$ we have, 
			\begin{align}
				\|v\|_{C^2((\cC_{n,k}\cap Q_{2\Lambda}(0, \bar\lambda_r^{-1}\bar\by_r))\times [\tau_r-1, \tau_r])} \lesssim_{n,\eps} r^{2\gamma^+_{n,k}-1-2\eps} < \kappa'_{\ref{Lem_App_Graph over Cylinder}}(n)/2\,; \label{Equ_Improve_|v| Rough C^2 Est in Q_R}
			\end{align}
			\begin{align*}
				|\bar\bx_r| + (|\bar\lambda_r-1| + |\bar\mbfA|)\Lambda \leq \Psi(\delta|n,\eps) < \kappa'_{\ref{Lem_App_Graph over Cylinder}}(n)/2\,.
			\end{align*}
			Hence item \ref{Item_GraphCylinder_Rot} of Lemma \ref{Lem_App_Graph over Cylinder} applies to give 
			\begin{align*}
				\bar v(\cdot, \tau_r) = v(\cdot+(0,\bar\lambda_r^{-1}\bar\by_r),\ \tau_r-2\log \bar\lambda_r) + \Upsilon_r + \breve\cE' = v(\cdot+(0,\bar\lambda_r^{-1}\bar\by_r),\ \tau_r) + \Upsilon_r + \cE'\,,
			\end{align*}
			which implies \eqref{Equ_Improve_bar v = v + Upsilon + cE_2} with error estimate
			\begin{align}
				\begin{split}
					\|\cE'\|_{C^0(\cC_{n,k}\cap Q_\Lambda)} & \leq \|\breve\cE'\|_{C^0(\cC_{n,k}\cap Q_\Lambda)} + 2|\log\bar\lambda_r|\cdot\|\partial_\tau v\|_{C^0((\cC_{n,k}\cap Q_\Lambda(0,\bar\lambda_r^{-1}\bar\by_r))\times [\tau'_r -2\log\bar\lambda_r, \tau'_r])} \\
					& \lesssim_{n, \eps} \left(\|v\|_{C^2((\cC_{n,k}\cap Q_{\Lambda+1}(0,\bar\lambda_r^{-1}\bar\by_r))\times [\tau'_r -1, \tau'_r])} + \|\Upsilon_r\|_{L^2}\right)\|\Upsilon_r\|_{L^2} \\
					& \lesssim_{n,\eps} \left(r^{2\gamma^+_{n,k}-1-4\eps} + |\bar\bx_r| + |\bar\lambda_r-1| + |\bar\mbfA| \right) \left(|\bar\bx_r| + |\bar\lambda_r-1| + |\bar\mbfA|\right) \,.
				\end{split} \label{Equ_Improve_cE_2 est}
			\end{align}
			Here the second inequality follows from Lemma \ref{Lem_App_Graph over Cylinder} \ref{Item_GraphCylinder_Rot} and the fact that \[
			\|\Upsilon_r\|_{L^2} \sim_n |\bar\lambda_r-1| + |\bar\bx_r| + |\bar\mbfA| \,,
			\]
			and the third inequality follows from \eqref{Equ_Improve_|v| Rough C^2 Est in Q_R}. 
		\end{proof}

		\noindent    \textbf{Step 4.} Combining \eqref{Equ_Improve_v = varphi_u + psi(tau) +cE_1 w. cE_1 est} at $\tau:=\tau_r$ time slice, \eqref{Equ_Improve_bar v = v + Upsilon + cE_2} and \eqref{Equ_Improve_|cE'| est from rot graph}, we shall derive the following estimate, which leads to the final estimate on the Whitney data in the next step.
		\begin{align}
			\begin{split}
				& \left\|(\varphi_{\bu_\circ}-\varphi_{\bar\bu})(\cdot, \tau_r)\right\|_{L^2} + \left\|e^{-\tau_r/2}\left(\psi_\circ(\cdot + (0, \bar\lambda_r^{-1}\bar\by_r))-\bar\psi \right) + \Upsilon_r \right\|_{L^2} \\
				&\qquad  \lesssim_{n,\eps} r^{2\gamma^+_{n,k}-2\eps} + \left(r^{2\gamma^+_{n,k}-1-4\eps} + |\bar\bx_r| + |\bar\lambda_r-1| + |\bar\mbfA| \right) \left(|\bar\bx_r| + |\bar\lambda_r-1| + |\bar\mbfA|\right) \,.
			\end{split} \label{Equ_Improve_|Delta varphi_bu + Upsilon + Delta psi| < r^(2gamma^+)}
		\end{align}
		\begin{proof}[Proof of \eqref{Equ_Improve_|Delta varphi_bu + Upsilon + Delta psi| < r^(2gamma^+)}.]
			By \eqref{Equ_Improve_v = varphi_u + psi(tau) +cE_1 w. cE_1 est} at $\tau:=\tau_r$ time slice and \eqref{Equ_Improve_bar v = v + Upsilon + cE_2}, we have in $\cC_{n,k}\cap Q_\Lambda$, 
			\begin{align*}
				& \left(\varphi_{\bar\bu}(\cdot, \tau_r) + e^{-\tau_r/2}\bar\psi + \bar\cE(\cdot, \tau_r)\right) - \left(\Upsilon_r + \cE'\right) \\
				& \qquad = \varphi_{\bu_\circ}(\cdot, \tau_r) + e^{-\tau_r/2}\psi_\circ(\cdot + (0,\bar\lambda_r^{-1}\bar\by_r)) + \cE_\circ(\cdot + (0,\bar\lambda_r^{-1}\bar\by_r), \tau_r) \,.
			\end{align*}
			Moving $\cE_\circ, \bar\cE, \cE'$ to the same side, integrating over $\cC_{n,k}\cap Q_\Lambda$ (with respect to Gaussian weight) and applying \eqref{Equ_Improve_v = varphi_u + psi(tau) +cE_1 w. cE_1 est} and \eqref{Equ_Improve_|cE'| est from rot graph} yield,
			\begin{align}
				\begin{split}
					& \left\|(\varphi_{\bu_\circ}-\varphi_{\bar\bu})(\cdot, \tau_r) + e^{-\tau_r/2}\left(\psi_\circ(\cdot + (0, \bar\lambda_r^{-1}\bar\by_r))-\bar\psi \right) + \Upsilon_r \right\|_{L^2(\cC_{n,k}\cap Q_\Lambda)} \\
					& \qquad  \lesssim_{n,\eps} \left\|\cE_\circ(\cdot + (0,\bar\lambda_r^{-1}\bar\by_r), \tau_r)\right\|_{L^2(\cC_{n,k}\cap Q_\Lambda)} + \|\bar\cE(\cdot, \tau_r)\|_{L^2(\cC_{n,k}\cap Q_\Lambda)} + \|\cE'\|_{L^2(\cC_{n,k}\cap Q_\Lambda)} \\
					& \qquad  \lesssim_{n,\eps} r^{2\gamma^+_{n,k}-2\eps} + \left(r^{2\gamma^+_{n,k}-1-4\eps} + |\bar\bx_r| + |\bar\lambda_r-1| + |\bar\mbfA| \right) \left(|\bar\bx_r| + |\bar\lambda_r-1| + |\bar\mbfA|\right) \,.
				\end{split} \label{Equ_Improve_Pf in Step 4}
			\end{align}
			On the other hand, since the second and third terms on the left hand side above are linear combinations of eigenfunctions of $-L_{n,k}$ with eigenvalue $\leq 1/2$ and perpendicular to the low spherical mode $\rmW_\SSp$, by Corollary \ref{Cor_L^2 Mono_Low Bd Upsilon_(u,u'; w)} and the choice of $\Lambda$ in \textbf{Step 3}, we have 
			\begin{align*}
				\text{LHS of \eqref{Equ_Improve_Pf in Step 4}} \gtrsim_n \left\|(\varphi_{\bu_\circ}-\varphi_{\bar\bu})(\cdot, \tau_r)\right\|_{L^2} + \left\|e^{-\tau_r/2}\left(\psi_\circ(\cdot + (0, \bar\lambda_r^{-1}\bar\by_r))-\bar\psi \right) + \Upsilon_r \right\|_{L^2}\,.
			\end{align*}
			Together with \eqref{Equ_Improve_Pf in Step 4}, this finishes the proof.
		\end{proof}
		
		\noindent    \textbf{Step 5 (\eqref{Equ_Improve_|Delta varphi_bu + Upsilon + Delta psi| < r^(2gamma^+)} yields the location estimates).} 
		We first prove a geometric estimate,
		\begin{align}
			\begin{split}
				& \left\|e^{-\tau_r/2}\left(\psi_\circ(\cdot + (0, \bar\lambda_r^{-1}\bar\by_r))-\bar\psi \right) + \Upsilon_r \right\|_{L^2} \\
				& \qquad \gtrsim_n\ e^{\tau_r/2}|\bq_\circ(\bar\lambda_r^{-1}\bar y) - \bar x| + \|\nabla\bq_\circ(\bar\lambda_r^{-1}\bar y) - \bar\bl\| + |\bar\lambda_r-1| + e^{-\tau_r/2}\|\psi_\circ - \bar\psi\|_{L^2}\,.
			\end{split} \label{Equ_Improve_Whitney data vs tilde Upsilon}
		\end{align}
		\begin{proof}[Proof of \eqref{Equ_Improve_Whitney data vs tilde Upsilon}.]
			By the definition \ref{Item_C^2Reg_Quad Q_p} of $\bq_p$, we can set \[
			\psi_\circ(\theta, y) = \left\langle\bq_\circ(y) - 2\tr \bq_\circ(y), \hat\theta \right\rangle + \bc_\circ(y),
			\]
			where $\bc_\circ$ is a (possibly $0$) cubic Hermitian polynomial on $\R^k$. Then for every $0\neq \by\in \R^k$,
			\begin{align*}
				\psi(\theta, y+\by) - \psi(\theta, y) & = \left\langle\bq_\circ(y + \by) - \bq_\circ(y), \hat\theta \right\rangle + (\bc_\circ(y+\by) - \bc_\circ(y)) \\
				& = \left\langle \bq_\circ(\by) + \nabla\bq_\circ(\by)\cdot y, \hat\theta \right\rangle  + (\bc_\circ(y+\by) - \bc_\circ(y)) \,.
			\end{align*}
			Hence combined with \eqref{Equ_Improve_Def lambda_r, bx_r, by_r} and \eqref{Equ_Improve_Define Upsilon_r}, we have (denote for simplicity $\tilde y_r:= \bar\lambda_r^{-1}\bar y = e^{-\tau_r/2} \bar\lambda_r^{-1}\bar\by_r$)
			\begin{align*}
				& \left(e^{-\tau_r/2}\left(\psi_\circ(\cdot + (0, \bar\lambda_r^{-1}\bar\by_r))-\bar\psi \right) + \Upsilon_r\right)(\theta, y) \\
				& \qquad = e^{-\tau_r/2}(\psi_\circ - \bar\psi)(\theta, y) + e^{\tau_r/2}\left\langle\bq_\circ(\tilde y_r) - \bar x, \hat\theta \right\rangle + \left\langle (\nabla\bq_\circ(\tilde y_r) - \bar \bl)y, \hat\theta \right\rangle \\
				& \qquad\ + \left(\varrho(\bar\lambda_r-1) + \bc_\circ(y+\bar\lambda_r^{-1}\bar\by_r) - \bc_\circ(y) \right) \,,
			\end{align*}
			where recall that we set $\bar\mbfA = \begin{bmatrix}
				0 & \bar\bl \\ -\bar\bl^\top & 0
			\end{bmatrix}$. 
			Note that as a linear combination of eigenfunctions of $-L_{n,k}$ on $\cC_{n,k}$, the four terms on the right-hand side above are pairwise orthogonal to each other. Together with Lemma \ref{Lem_Facts on Hermitian Polyn}, this implies \eqref{Equ_Improve_Whitney data vs tilde Upsilon}. 
		\end{proof}
		
		Now combining \eqref{Equ_Improve_Whitney data vs tilde Upsilon} with \eqref{Equ_Improve_|Delta varphi_bu + Upsilon + Delta psi| < r^(2gamma^+)} and \eqref{Equ_Improve_Def lambda_r, bx_r, by_r}, we derive (denote as above $\tilde y_r:= \bar\lambda_r^{-1}\bar y$ and recall $\tau_r=-2\ln r$)
		\begin{align}
			\begin{split}
				& |\bar\lambda_r-1| + r^{-1}|\bq_\circ(\tilde y_r) - \bar x| + \|\nabla\bq_\circ(\tilde y_r) - \bar\bl\| + r\|\psi_\circ - \bar\psi\|_{L^2} \\
				& \qquad \lesssim_{n,\eps} r^{2\gamma^+_{n,k}-2\eps} + \left(r^{2\gamma^+_{n,k}-1-4\eps} + r^{-1}|\bar x| + |\bar\lambda_r-1| + |\bar\mbfA| \right) \left(r^{-1}|\bar x| + |\bar\lambda_r-1| + |\bar\mbfA|\right).
			\end{split} \label{Equ_Improve_Pf in Step 5}
		\end{align}
		Recall in \textbf{Step 2} we get $\|\psi_\circ\|_{L^2}, \|\bar\psi\|_{L^2}\leq 1$, and $\bq_\circ$ is a quadratic form, the remaining terms on the right-hand side have estimates,
		\begin{align*}
			r^{-1}|\bar x| & \leq r^{-1}|\bq_\circ(\tilde y_r) - \bar x| + r^{-1}\|\bq_\circ\|\cdot|\tilde y_r|^2 \leq r^{-1}|\bq_\circ(\tilde y_r) - \bar x| + 2r \,, \\
			|\bar\mbfA| & \lesssim_{n} \|\bar\bl\| \leq \|\nabla\bq_\circ(\tilde y_r) - \bar\bl\| + \|\nabla\bq_\circ\|\cdot |\tilde y_r| \lesssim_n \|\nabla\bq_\circ(\tilde y_r) - \bar\bl\| + r \,.
		\end{align*}
		Using these (combining with \eqref{Equ_Improve_A priori bd on |t|,|x|,|A|}), by taking $\delta_{\ref{Lem_C^2 Reg_Main}}(n,\eps)\ll 1$, one can absorb the second term on the right-hand side of \eqref{Equ_Improve_Pf in Step 5} by its left hand side and get, 
		\begin{align}
			|\bar\lambda_r-1| + r^{-1}|\bq_\circ(\tilde y_r) - \bar x| + \|\nabla\bq_\circ(\tilde y_r) - \bar\bl\| + r\|\psi_\circ - \bar\psi\|_{L^2}  \lesssim_{n,\eps} r^{2\gamma^+_{n,k}-4\eps} \,.
		\end{align}
		Recall that $|\bar\lambda_r-1|\sim r^{-2}|\bar t|$ and $|\tilde y_r - \bar y|\lesssim r|\bar\lambda_r-1|$ and $\bar\psi = \psi_{\bar p}\circ e^{\bar\mbfA}$. This proves the desired estimate \eqref{Equ_C^2 Reg_Main}.
	\end{proof}

	\section{Proof of Theorem \ref{Thm_Intro_C^2 Reg}} \label{Sec_PfMainThm}
	In this section, we collected all the ingredients from the previous sections to prove Theorem \ref{Thm_Intro_C^2 Reg}. Throughout this section, we let $1\leq k\leq n-1$, $\bM$ is a mean curvature flow in $\R^{n+1}$. We will further assume $\bM$ has entropy upper bound by $\eps^{-1}$ for some small $\eps\in (0, 1/10n)$. 
	
	For part \ref{Item_Intro_(cS_k)_0} of Theorem \ref{Thm_Intro_C^2 Reg}, the Hausdorff dimension estimate of $\cS_k(\bM)_0$ will be proved in Section \ref{Subsec_cS_0 lower dim} simply by a blow-up argument. When $k=1$, the singularities in $\cS_{k}(\bM)_0$ must be nondegenerate, and hence the isolatedness has been proved in \cite{SunWangXue1_Passing}.
	
	Part \ref{Item_Intro_(cS_k)_+ close} of Theorem \ref{Thm_Intro_C^2 Reg} is a direct consequence of the upper-semi-continuity of decay order at $\tau=\infty$, to be discussed in Section \ref{Subsec_cS_+ close}.
	
	Part \ref{Item_Intro_(cS_k)_+ reg} and \ref{Item_Intro_t|(cS_k)_+ Holder} of Theorem \ref{Thm_Intro_C^2 Reg} will be proved in Section \ref{Subsec_cS_+ reg} based on the location estimate in Lemma \ref{Lem_C^2 Reg_Main}. In fact, we establish a general result to get $C^{2,\al}$ regularity from the location estimate like \eqref{Equ_C^2 Reg_Main} using the Whitney extension theorem, see Proposition \ref{Prop_C^2 Reifenberg}. 
	
	\subsection{Hausdorff dimension of partially nondegenerate singularities} \label{Subsec_cS_0 lower dim}
	Recall the $C^1$-normal form theorem of cylindrical singularity. 
	\begin{theorem}[{\cite[Theorem 1.4]{SunXue2022_generic_cylindrical}}]\label{thm:PseduoLocality_Partial}
		Suppose $\cM$ is the rescaled mean curvature flow of $\bM$ at a $k$-cylindrical singularity. Then for any $\mathbf{L}>0$, there exist a rotation $\mbfB\in \mathrm{O}(n+1)$, an index set $\cI\subset\{1,2,\cdots,k\}$ and $\tau_0>0$ such that for every $\tau>\tau_0$, $\mbfB_\sharp\cM(\tau)$ can be written as a graph of function $v(\cdot,\tau)$ over $\cC_{n,k}$ in $Q_{\mathbf{L}}$, and that
		\begin{equation}
			\left\| v(\theta,y,\tau)-
			\sum_{i\in \mathcal I}\frac{\varrho}{4\tau}(y_i^2-2)
			\right\|_{C^1(\cC_{n,k}\cap Q_{\mathbf{L}})}=o(\tau^{-1}) \,.
		\end{equation}
	\end{theorem}
	
	In particular, if $\cI=\{1,2,\cdots,k\}$, the singularity is called nondegenerate. In the following, we consider $|\cI|=m$, $1\leq m\leq k-1$. 
	We use $\cS_{k}(\bM)_0^m$ to denote the set of singularities of this type, and we call them \textbf{partially nondegenerate}. 
	For every $p\in\cS_{k}(\bM)_0^m$, let $\cC_p\in \Rot(\cC_{n,k})$ be the tangent flow of $\bM$ at $p$. We use $V^{\dg}_p$ to denote the $(k-m)$-dimensional subspace of $\R^{n+1}$ such that $\mbfB(V_p^\dg)=\text{span}\{\partial_{y_i}: i\in\{1,2,\cdots,k\}\backslash\cI\}$. In other words, $V_p^\dg$ denotes the subspace of the degenerate directions; also, we use $V^{\ndg}_p:= \spine(\cC_{p})\cap (V_p^\dg)^\perp$ to denote the subspace of the nondegenerate directions, and use $\by_p^\ndg$ to denote the projection of the vector $\by\in\spine(\cC_p)$ to $V_p^\ndg$.
	Since by definition, \[
	\cS_k(\bM)_0=\bigcup_{m=1}^{k-1} \cS_k(\bM)_0^m \,,
	\] 
	part \ref{Item_Intro_(cS_k)_0} of Theorem \ref{Thm_Intro_C^2 Reg} is a direct consequence of the following dimension estimate of $\cS_{k}(\bM)_0^m$. 
	\begin{theorem}\label{ThmPartial}
		The parabolic Hausdorff dimension of $\cS_{k}(\bM)_0^m$ is at most $(k-m)$.
	\end{theorem}
	
	The key to proving this theorem is the following lemma asserting that $k$-cylindrical singularities converging to some $p\in \cS_k(\bM)_0^m$ are approaching in the degenerate direction $V_p^\dg$.
	
	\begin{lemma}\label{lem:Partial_Nondeg_Nearby_Sing}
		Suppose $p=(\orig,0)\in \cS_k(\bM)_0^m$ is a partially nondegenerate singularity of $\bM$ with tangent flow $\cC_{n,k}$. 
		If $\{p_j=(x_j,y_j,t_j)\}_{j\geq 1}\subset \cS_k(\bM)$ is a sequence of singularities converging to $(\orig,0)$ as $j\to\infty$, such that $y_j\neq 0$ and $y_j/|y_j|\to \hat \by\in \R^k$. Then $\hat\by_p^\ndg = 0$.
	\end{lemma}
	\begin{proof}
		Denote for simplicity $\tau\mapsto \cM_j(\tau):= \cM^{p_j}(\tau)$ to be the RMCF of $\bM$ based at $p_j$. The following facts have essentially been proved in \cite[Section 4]{SunWangXue1_Passing}, especially Claims 4.4 and 4.5. We refer the readers to the proof there. We fix $\mbfL\gg 2n$ to be determined. For $j\gg 1$, let \[
		a_j:= 2\log(|y_j|^{-1}\mbfL) \,.
		\] 
		\begin{Claim} \label{Claim_Time Transl cM_j small d_n,k and bd cN}
			There exists $\eps_2(n,\eps)\in (0, \eps)$ such that 
			\begin{align}
				\limsup_{j\to \infty} & \sup_{|\tau-a_j|\leq 3} \mbfd_{n,k}(\cM_j(\tau)) =0 \,; \label{Equ_Isol_Time Transl cM_j has d_n,k small} \\
				\limsup_{j\to \infty} & \sup_{|\tau-a_j|\leq 3} \cN_{n,k}(\tau, \cM_j) < \eps_2^{-1}\,. \label{Equ_Isol_Time Transl cM_j has bded cN}
			\end{align}
		\end{Claim}
		
		\begin{Claim} \label{Claim_u_j dominated by h_1}
			Let $v_j(\cdot, \tau)$ be the graphical function of $\cM_j(\tau)$ in $Q_\mbfL$ and $0$-extended to an $L^\infty$ function on $\cC_{n,k}$. Then 
			
			\begin{align}
				\limsup_{j\to \infty} \inf_{c>0,\ c'\in \RR} \|c^{-1}v_j(\cdot, a_j) - c' - y\cdot \by_p^\ndg\|_{L^2} & \leq \Psi(\mbfL^{-1}|n)\,.
				\label{Equ_Isol_u_j dominated by h_1} 
			\end{align}
			Moreover, if $\by_p^\ndg\neq 0$, then 
			\begin{align}
				\limsup_{j\to \infty} \|v_j(\cdot, a_j)\|_{L^2}^{-1}\cdot \|\Pi_{\leq -1/2} (v_j(\cdot, a_j))\|_{L^2} & \geq 1-\Psi(\mbfL^{-1}|n) \,. \label{Equ_Isol_u_j dominated by leq -1/2 modes}
			\end{align}
			Here recall $\Psi(\mbfL^{-1}|n)$ denotes a function that converges to $0$ as $\mbfL^{-1}\to 0$, and the speed converging to $0$ is dimensional.
		\end{Claim}
		With these Claims, let $\bar{\delta}\in(0,\eps_2)$ that is smaller than all the $\delta$'s throughout Section \ref{Sec_L^2 Noncon}, $\mbfL(n, \eps)\gg1$ be such that $\Psi(\mbfL^{-1}|n)< \bar\delta$. Conclude from \eqref{Equ_Isol_Time Transl cM_j has bded cN}, \eqref{Equ_Isol_u_j dominated by leq -1/2 modes}, and Corollary \ref{Cor_L^2 Mono_RMCF w graphical eigenfunc has cN = spectrum}, if $\hat\by_p^\ndg\neq 0$, we have \[
		\limsup_{j\to \infty} \sup_{\tau\in [a_j, a_j+3]} \cN(\tau, \cM_j) \leq -1/2+\eps_2 < -1/4\,.
		\]
		The lemma then follows from Lemma \ref{Lem_L^2 Mono_cN>-eps if singular}. 
	\end{proof}
	\begin{proof}[Proof of Theorem \ref{ThmPartial}]
		We prove by contradiction. Suppose $\cH^\beta_{\cP}(\cS_{k}(\bM)_0^m) >0$ for some $\beta>(k-m)$. Then by Lemma \ref{lem:HausdorffMeasure} with $E:= \cS_{k}(\bM)_0^m$, there exists $p_\circ\in E$ such that the accumulation set \[
		\mathcal A_{E, p_\circ} := \left\{z\in \overline{P_{1/2}} \,:\exists\,(j_\ell)_{\ell \ge 1}\nearrow +\infty, (z_\ell)_{\ell\ge 1} \text{ s.t. $z_{\ell}\in r_{j_\ell}^{-1}.(E-p_\circ) \cap P_{1/2}$ and $z_{\ell}\to z$} \right\},
		\] 
		has $\cH^\beta_{\cP}(\cA_{E, p_\circ})>0$. In the following, we shall prove that $\cA_{E, p_\circ}\subset V^\dg_{p_\circ}\times\{0\}$, which is a contradiction to $\cH_{\cP}^{\beta}(\cA_{E, p_\circ})>0$ since $\dim V^\dg_{p_\circ}=k-m$. By a translation and rotation, we assume without loss of generality that $p_\circ=(\orig, 0)$ and the tangent flow of $\bM$ at $p_\circ$ is $\cC_{n,k}$.
		
		From its definition above, for any $(Z, s)=(x_Z, y_Z, s)\in \cA_{E, p_\circ}\subset \R^{n-k+1}\times \R^k\times \R$, there exists a sequence of singular points $p_j=(x_j,y_j,t_j)$ in $\cS_{k}(\bM)_0^m\cap P_1((\orig,0))$ such that as $j\to\infty$, 
		\begin{itemize}
			\item $\max\{|(x_j,y_j)|,|t_j|^{1/2}\} =: r_j \searrow 0$;
			\item $(r_j^{-1} x_j, r_j^{-1} y_j,r_j^{-2}t_j)\to (x_Z,y_Z,s)$.
		\end{itemize}
		We use $r_j^{-1}\bM$ to denote the mean curvature flow parabolic dilated by the factor $r_j^{-1}$. Then since $\cC_{n,k}$ is the tangent flow of $\bM$ at $(\orig, 0)$, the sequence of mean curvature flow $\{r_j^{-1}\bM\}_{j=1}^\infty$ converges, in the sense of Brakke flow, to the shrinking cylinder $(t\mapsto\sqrt{-t}\, \cC_{n,k})$. By the lower semi-continuity of the Gaussian density of singularities, we must have that $(x_Z,y_Z,s)=\lim\limits_{j\to\infty}(r_j^{-1} x_j, r_j^{-1} y_j,r_j^{-2}t_j)$ belongs to the singular set of $(t\mapsto\sqrt{-t}\, \cC_{n,k})$, hence $x_Z=\orig$ and $s=0$. Finally, Lemma \ref{lem:Partial_Nondeg_Nearby_Sing} shows that $(y_Z)_{p_\circ}^\ndg = 0$ and hence $y_Z \in V_{p_\circ}^\dg$. This finises the proof that $\cA_{E, p_\circ}\subset V^\dg_{p_\circ}\times\{0\}$.    
	\end{proof}
	
	\begin{Rem}\label{rem_h_3_dominant}
		While the major goal of this section is to study partially nondegenerate singularities, the idea that analyzing the singular sets by slightly translating the flow in the spacetime can be used to study singularities with various asymptotics. For example, using the same method as this section, one can show that if the leading asymptotics of the RMCF at the origin is $h_3(y_j)=c(y_j^3-6y_j)$, for some $j\in\{1,2,\cdots,k\}$, the singular set can not approach the origin from the $y_j$-direction.

		More generally, if the tangent flow at a singularity $(\orig, 0)$ is $\cC_{n,k}$, and the leading asymptotics of the RMCF is given by a Hermitian polynomial $h$ of degree $l\geq 3$ with top degree part $q$ such that \[
		\cZ:=\{\by\in \R^k: \nabla q(\by) = 0\}\ \ \text{ is not a linear subspace}\,.
		\]
		(One such example is when $k=2$ and $h(y_1, y_2)=(y_1^2-2)(y_2^2-2)$, in which case $q(y_1, y_2)=y_1^2y_2^2$ and $\cZ=\{y_1y_2=0\}$ is a union of two orthogonal lines.) Then, our analysis above doesn't rule out the possibility that the accumulation set of $k$-cylindrical singularities near $(\orig, 0)$ is $\cZ$. An interesting question is to construct a MCF whose $2$-cylindrical singular set is a union of curves intersecting at $\orig$ whose accumulation set is $\{y_1y_2=0\}$.
	\end{Rem}

	\subsection{Closedness of degenerate singular set} \label{Subsec_cS_+ close}
	Given a $k$-cylindrical singularity $p$, we use $\mbfU_p$ to denote the arrival time function corresponding to the tangent flow at $p$. Note that $\mbfU_p$ could possibly be different from $\mbfU_{n,k}$ by a rotation. Recall that by the proof of Theorem \ref{thm:AsyProfile} in Section \ref{Subsec_Pf Main AsympProf}, $\cS_k(\bM)_+$ consists of $k$-cylindrical singularities $p$ with the limit of decay order $\cN_{\mbfU_p}(\infty, \cM^p)>0$. Therefore, the relative closedness of $\cS_{k}(\bM)_+$ in $\cS_k(\bM)$ follows directly from the lemma below, asserting that $\cN_{\mbfU_p}(\infty,\cM^p)$ is upper semi-continuous in $p$.
	\begin{proposition}\label{prop:semi_continuity_decay_order}
		Suppose $\lambda[\bM]<+\infty$, $p_j\in\cS_{k}(\bM)$, $1\leq j\leq \infty$, and $p_j\to p_\infty$ as $j\to\infty$. Then \[
		\limsup_{j\to\infty} \cN_{\mbfU_{p_j}}(\infty, \cM^{p_j})\leq \cN_{\mbfU_{p_\infty}}(\infty, \cM^{p_\infty}) \,.
		\]
	\end{proposition}
	
	\begin{proof}
		Without loss of generality, we assume $p_\infty=(\orig,0)$, $\mbfU_{p_\infty}=\mbfU_{n,k}$, and assume $\cM$ is the associated RMCF with $\lambda[\cM]<\eps^{-1}$ for some $\eps\in(0,1)$. If $\cN_{n,k}(\infty,\cM)=+\infty$ then the inequality is automatically true, and in the following we assume $\cN_{n,k}(\infty,\cM)<+\infty$. Then for any $\eps>0$, we can find $T>0$ such that for $\tau\geq T$, $|\cN_{n,k}(\tau,\cM)-\cN_{n,k}(\infty,\cM)|<\eps/2$, and $\mbfd_{n,k}(\cM(\tau))<\eps/2$. Because $p\to\mbfd_{\mbfU_{p}}(\tau,\cM^p)$ is continuous for any fixed $\tau$, so does $p\to\cN_{n,k}(\tau,\cM^{p_j})$, and hence we have $|\cN_{\mbfU_{p_j}}(T,\cM^{p_j})-\cN_{n,k}(T,\cM)|<\eps/2$ for sufficiently large $p$. This implies that for any $\eps>0$, there exist $T>0$ and $j$ sufficiently large, such that $|\cN_{\mbfU_{p_j}}(T,\cM^{p_j})-\cN_{n,k}(\infty,\cM)|<\eps$, and $\mbfd_{\mbfU_{p_j}}(\cM^{p_j}(T))<\eps$. By \cite[Theorem 0.5]{ColdingMinicozzi25_quantitativeMCF}, if $\eps$ is chosen sufficiently small, when $j$ is sufficiently large, $\mbfd_{\mbfU_{p_j}}(\cM^{p_j}(\tau))$ is small for all $\tau\geq T$, and hence Corollary \ref{Cor_L^2 Mono_Discrete Growth Mono} can be applied to $\cM^{p_j}$. As a consequence of the discrete monotonicity formula Corollary \ref{Cor_L^2 Mono_Discrete Growth Mono}, $\cN_{\mbfU_{p_j}}(\infty,\cM^{p_j})\leq \cN_{n,k}(\infty,\cM)$.
	\end{proof}

	\subsection{A $C^{2,\alpha}$ Reifenberg-type Theorem} \label{Subsec_cS_+ reg}
	In this subsection, we use the location estimate in Lemma \ref{Lem_C^2 Reg_Main} to prove part \ref{Item_Intro_(cS_k)_+ reg} and \ref{Item_Intro_t|(cS_k)_+ Holder} of Theorem \ref{Thm_Intro_C^2 Reg}. In fact, we shall prove a general $C^{2,\al}$ Reifenberg-type result asserting that a compact subset satisfying a reasonable location estimate is locally contained in a $C^{2,\al}$ submanifold. 
	
	We first introduce some notations.  Let $\Jet_{n,k}^2$ be the collection of the data $(\bL, \bQ)$, where
	\begin{itemize}
		\item $\bL: \RR^{n+1}\to \RR^{n+1}$ is a self-adjoint linear homomorphism given by orthogonal projection onto a $k$-dimensional linear subspace of $\RR^{n+1}$, denoted by $|\bL|$. Let $\bL^\perp:= \id-\bL$ be the orthogonal projection onto $|\bL|^\perp$;
		\item $\bQ:= \bq\circ\bL$ is the quadratic polynomial $\RR^{n+1}\to |\bL|^\perp$ extended from $\bq: |\bL|\to |\bL|^\perp$, a $|\bL|^\perp$-valued quadratic polynomial on $|\bL|$.
	\end{itemize}
	
	\begin{proposition} \label{Prop_C^2 Reifenberg}
		Let $K\subset \RR^{n+1}\times \RR$ be a compact subset, $\Lambda>0$ and $\alpha\in (0, 1)$. Suppose there's a map $J: K\to \Jet^2_{n,k}$ satisfying that for every $r\in (0, 1]$ and every $p_i =: (X_i,  t_i)\in K$ with $d_{\cP}(p_1,  p_2)\leq r$ and $J(p_i) =: (\bL_i, \bQ_i)$, we have $\|\bL_1^\perp - \bL_2^\perp\|\leq \Lambda r$; $\|\bQ_1\|,\, \|\bQ_2\|\leq \Lambda$, and
		\begin{align}
			\begin{split}
				r^{-2}| t_2- t_1| & + r^{-1}|\mbfL_1^\perp ( X_2- X_1) - \mbfQ_1( X_2- X_1)|  \\
				& + \left\|\bL_1^\perp \circ \left((\bL_1^\perp - \bL_2^\perp) - \nabla \mbfQ_1( X_2- X_1)\right) \right\| + r\|\bQ_2 - \bQ_1\| \leq \Lambda r^{1+\alpha} \,; 
			\end{split} \label{Equ_Whitney_Location est assump}
		\end{align}
		where $\|\cdot\|$ in the last two terms stands for the operator norm of linear or quadratic $\R^{n+1}$-valued maps from $\R^{n+1}$. Then, 
		\begin{enumerate} [label={\normalfont(\roman*)}]
			\item\label{Item_Whitney_Concl_C^2Reg} $K$ is locally contained in a $k$-dimensional $C^{2,\alpha}$-submanifold in $\RR^{n+1}\times \RR$; and if $K$ is a $k$-dimensional submanifold near $p\in K$, then its tangent space is $T_pK = \bL_p$ and its second fundamental form at $p$ is given by $2\bq_p$.
			\item\label{Item_Whitney_Concl_t|_K Holder} $\dim_H \mathfrak{t}(K) \leq k/(3+\alpha)$.
		\end{enumerate}
	\end{proposition}
	
	To prove this proposition, recall the following Whitney's Extension Theorem in \cite[Page 5]{Fefferman09_Whitney}.
	\begin{lemma}[Whitney's Extension Theorem]\label{Lem:Whitney}
		Let $\al\in (0,1]$, $\ell \in \mathbb Z_+$, $K\subset \R^N$ be a compact set, and $f: K\rightarrow \R$ be a given mapping.
		Suppose that for any $x_\circ \in K$ there exists a polynomial $P_{x_\circ}$ such that:
		\begin{itemize}
			\item $P_{x_\circ} (x_\circ) = f(x_\circ)$;
			\item $|D^jP_{x_\circ} (x) -D^j P_x(x)| \leq  \Lambda |x-x_\circ|^{\ell+\al-j}$ for all $x\in K$ and $j\in\{0,1,\dots,\ell\}$, where $\Lambda>0$ is independent of $x_\circ$.
		\end{itemize}
		Then there exists $F:\R^N\to \R$ of class $C^{\ell,\al}$ such that \[
		F|_{K}\equiv f \qquad \text{and}\qquad F(x) = P_{x_\circ}(x) + O(|x-x_\circ|^{\ell+\al}), \quad \forall\, x_\circ, x \in K \,. 
		\]
	\end{lemma}
	The basic strategy is then to construct a polynomial from the data $(p, J(p))$ and check that it satisfies the second bullet point above.

	\begin{proof}[Proof of Proposition \ref{Prop_C^2 Reifenberg}]
		Associated to every $p_\circ = (X_\circ, t_\circ)\in K$ and $J_\circ(p_\circ) = (\bL_\circ, \bq_\circ)\in \Jet_{n,k}^2$, we define a $4$-th order polynomial on $\RR^{n+1}$ by 
		\begin{equation}
			\mbf_{p_\circ}(X) := t_\circ - \left|\bL_\circ^\perp(X-X_\circ) - \bQ_\circ(X-X_\circ) \right|^2 \,.
		\end{equation}
		\textbf{Claim.} There exists $\delta_\Lambda>0$ depending only on $\Lambda$ such that for every $p_1=(X_1, t_1)\in K$ and $p_2=(X_2, t_2)\in K\cap P_{2\delta_\Lambda}(p_1)$, we have $X_2 \neq X_1$ and 
		\begin{align}
			|D^j\mbf_{p_1}(X_2) - D^j\mbf_{p_2}(X_2)| \lesssim_{n, \Lambda} |X_2 - X_1|^{3+\al-j}, \qquad \forall\, 0\leq j\leq 3 \,.  \label{Equ_Whitney_Compatible Condition}
		\end{align}
		
		Let's first finish the proof assuming this Claim. By Lemma \ref{Lem:Whitney}, for every $p_1=(X_1, t_1)\in K$, there exists $\mbf\in C^{3,\alpha}(\RR^{n+1})$ such that for every $\bar p = (\bar X, \bar t)\in K\cap P_{\delta_\Lambda}(p_1)$, we have 
		\begin{align}
			D^j \mbf(\bar X) = D^j\mbf_{\bar p}(\bar X), \qquad \forall\, 0\leq j\leq 3 \,. \label{Equ_Whitney_f coincide w f_p}
		\end{align}
		Hence, there's a neighborhood $U_1\subset B^{n+1}_{\delta_\Lambda}(X_1)$ of $X_1$ such that $\nabla_{\bL_1^\perp} \mbf: U_1\to \RR^{n-k+1}$ is a $C^{2,\alpha}$ map with non-degenerate differential at $X_1$. Thus $\breve\Gamma:=\{X\in U_1:\nabla_{\bL_1^\perp} \mbf(X) = 0\}$ is a $k$-dimensional $C^{2,\alpha}$-submanifold of $\RR^{n+1}$. Also, by taking $j=0, 1$ in \eqref{Equ_Whitney_f coincide w f_p}, it's easy to check that \[
		K\cap (U_1\times [t_\circ-\delta_\Lambda^2, t_\circ+\delta_\Lambda^2]) \subset \{(\bar X, \mbf(\bar X)): \bar X\in U_1\cap \breve\Gamma\}=: \Gamma\,,
		\]
		where $\Gamma$ is the graph of a $C^{3,\al}$ function over a $C^{2,\al}$ submanifold $\breve\Gamma$ in $\RR^{n+1}\times\RR$, hence is a $k$-dimensional $C^{2,\alpha}$-submanifold in $\RR^{n+1}\times \RR$. Now suppose $K$ is a submanifold, and without loss of generality, assume $X_1=\orig$, $\mbfL_{1}=\{(0,y):y\in\R^{k}\}$, we may assume $G:\R^k\to\R^{n-k+1}$ is locally defined such that $\nabla_{\bL_1^\perp} \mbf(G(y),y) = 0$ in a neighborhood of $\orig$. Then, by direct calculations, at $y=0$, we have
		\begin{align}
			\nabla_y G=\nabla_y \bq,&\quad 
			\nabla^2_y G=\nabla^2_y \bq=2\bq.
		\end{align}
		This shows the criterion for the second fundamental form of $K$.
		
		Moreover, the estimate \eqref{Equ_Whitney_Location est assump} on $|t_2-t_1|$ term implies that $\mfk{t}|_{K\cap P_{\delta_\Lambda}(p_1)}$ is $(3+\al)$-H\"older. The Hausdorff dimension estimate of $\mathfrak{t}(K)$ then follows by a standard covering argument, see \cite[Proposition 7.7]{FigalliRos-OtonSerra_20_Generic}. 
		\begin{proof}[Proof of the Claim.]
			Let us do the Taylor expansion of $(\mbf_{p_1}-\mbf_{p_2})$ at $X_2$. Since this a polynomial, it suffices to rewrite $(\mbf_{p_1}-\mbf_{p_2})(X_2+X)$ into sum of homogeneous polynomials in $X$:
			\begin{align*}
				& (\mbf_{p_1}-\mbf_{p_2})(X_2+X) \\
				& \quad = (t_1-t_2) + |\bL_2^\perp(X)-\bQ_2(X)|^2 - \left| \bL_1^\perp((X_2-X_1)+X) - \bQ_1((X_2-X_1)+X) \right|^2 \\
				& \quad = (t_1-t_2) + |\bL_2^\perp(X)-\bQ_2(X)|^2 \\
				& \quad\ - \Big| \underbrace{\left(\bL_1^\perp(X_2-X_1) - \bQ_1(X_2-X_1)\right)}_{=:\, a} + \underbrace{\left(\bL_1^\perp(X) - \nabla_X\bQ_1(X_2-X_1)\right)}_{=:\, b(X)} - \bQ_1(X) \Big|^2 \\
				& \quad =: \cT_0(X) + \cT_1(X) + \cT_2(X) + \cT_3(X) + \cT_4(X) \,;
			\end{align*}
			where 
			\begin{align*}
				\cT_0(X) & = (t_1-t_2) - |a|^2\,; & 
				\cT_1(X) & = -2\left\langle a,\ b(X)\right\rangle \,; \\
				\cT_2(X) & = |\bL_2^\perp(X)|^2 - |b(X)|^2 + 2\left\langle a,\ \bQ_1(X)\right\rangle \,; \\
				\cT_3(X) & = -2\langle \bL_2^\perp(X),\ \bQ_2(X)\rangle + 2\langle b(X),\ \bQ_1(X) \rangle \,; &
				\cT_4(X) & = |\bQ_2(X)|^2 - |\bQ_1(X)|^2 \,.
			\end{align*}
			If we set $r:= d_\cP(p_1, p_2)$, then by \eqref{Equ_Whitney_Location est assump}, when $r\leq 2\delta_\Lambda \ll 1$, we have
			\begin{align*}
				|X_2-X_1| \sim_n r\,, & &
				|a|\lesssim_{n, \Lambda} r^{2+\al}\,, & &
				\left|\bL_1^\perp\left(b(X)-\bL_2^\perp(X)\right) \right| \lesssim_{n, \Lambda} r^{1+\al}|X| \,.
			\end{align*}
			Also note that $(\bL_1^\perp)^2=\bL_1^\perp$, $\bL_1^\perp\circ \bQ_1 = \bQ_1$, and
			\begin{align*}
				|\bL_i^\perp(X)|\lesssim_n |X|, & &
				|(\bL_1^\perp - \bL_2^\perp)(X)|+ |\nabla_X \bQ_1(X_2-X_1)|\lesssim_{n,\Lambda} r|X|, & &
				|\bQ_i(X)|\lesssim_{n, \Lambda} |X|^2\,.
			\end{align*}
			Hence \eqref{Equ_Whitney_Location est assump} and estimates above imply that for every $X\in \R^{n+1}$ and every $0\leq j\leq 3$,
			\begin{align*}
				|\cT_j(X)| \lesssim_{n, \Lambda} r^{3+\al-j}|X|^j \,.
			\end{align*}
			While since $\cT_j$ is the $j$-th order Taylor polynomial of $(\mbf_{p_1}-\mbf_{p_2})$ at $X_2$, we then have \[
			|D^j(\mbf_{p_1}-\mbf_{p_2})(X_2)| \sim_{n, j} \|\cT_j\|_{C^0(B_1^{n+1})} \lesssim_{n, \Lambda} r^{3+\al-j} \,.
			\]
		\end{proof}
	\end{proof}
	
	\begin{proof}[Proof of Theorem \ref{Thm_Intro_C^2 Reg} \ref{Item_Intro_(cS_k)_+ reg} and \ref{Item_Intro_t|(cS_k)_+ Holder}.]
		Let $\al\in (0, \min\{1, \frac{2}{n-k}\})$ and we choose $\eps>0$ smaller such that $\gamma^+_{n,k}-10\eps\geq \al$. For every $p\in \cS_k(\bM)_+$, we first claim that there exists $s_p>0$ such that for every $r\in (0, s_p]$, $\cS_k(\bM)\cap \overline{P_r(p)}$ is compact (note that by item \ref{Item_Intro_(cS_k)_+ close} proved in Section \ref{Subsec_cS_+ close}, this implies the compactness of $\cS_k(\bM)_+\cap \overline{P_r(p)}$). Suppose for contradiction that there exist sequences of radius $s_j\searrow 0$ and limit points $p_j\in \overline{P_{s_j}(p)}\setminus \cS_k(\bM)$ of $k$-cylindrical singularities. Then by the upper semi-continuity of the Gaussian density, we have \[
		\Theta_\bM(p_j) \geq \cF[\cC_{n,k}] = \Theta_{\bM}(p), \quad \forall\, j\geq 1\,; \qquad 
		\liminf_{j\to \infty} \Theta_\bM(p_j) \leq \Theta_{\bM}(p) \,.
		\]
		Hence as $j\to \infty$, $\Theta_\bM(p_j) \to \Theta_{\bM}(p)$, and since the tangent flow of $\bM$ at $p$ is a rotation of round cylinder $\cC_{n,k}$ with multiplicity $1$, from the rigidity of round cylinders by Colding-Ilmanen-Minicozzi \cite{ColdingIlmanenMinicozzi15} and the quantitative uniqueness of tangent flow by Colding-Minicozzi \cite{ColdingMinicozzi25_quantitativeMCF}, when $j\gg 1$, the tangent flow of $\bM$ at $p_j$ is also a rotation of $\cC_{n,k}$. This contradicts the assumption $p_j\notin \cS_k(\bM)$. 
		
		Now for every $p\in \cS_k(\bM)_+$, again by the quantitative uniqueness of tangent flow by Colding-Minicozzi \cite{ColdingMinicozzi25_quantitativeMCF}, we can take $r_p\in (0, s_p]$ such that for every $p'\in \cS_k(\bM)_+\cap \overline{P_{r_p}(p)}$, after a rotation, the RMCF of $r_p^{-1}\cdot (\bM-p')$ is $\delta_{\ref{Lem_C^2 Reg_Main}}(n, \eps)$-$L^2$ close to $\cC_{n,k}$ over $\tau\in [-1, +\infty)$, and converges to $\cC_{n,k}$ as $\tau\to +\infty$. By an a priori rescaling of $\bM$, we may assume without loss of generality that $r_p=1$. We shall conclude item \ref{Item_Intro_(cS_k)_+ reg} by applying Proposition \ref{Prop_C^2 Reifenberg} to $K:= \cS_k(\bM)_+\cap \overline{P_{\delta_{\ref{Lem_C^2 Reg_Main}}(n, \eps)}(p)}$, where to each $p\in K$, the Whitney data assigned to $p$ is $(\bL_p, \bQ_p)$ as specified at the beginning of Section \ref{Sec_C^(2,al) Reg}. For every $p_1, p_2\in K$, the upper bound on $\bQ_i$ holds by Lemma \ref{Lem_C^2 Reg_Main}, hence it suffices to verify \eqref{Equ_Whitney_Location est assump}. By a translation and rotation, we assume without loss of generality that $p_1=(\orig, 0)$ and $\bL_1=\{0\}\times \R^k$. Hence by writing $p_2=(x_2, y_2, t_2)$ and applying \eqref{Equ_C^2 Reg_Main} with $\bL_1, \bq_1, p_2, \bL_2, \bq_2$ in place of $\bL_\circ, \bq_\circ, \bar p, \bar\bL, \bar\bq$, we get desired estimates on terms $|t_2-t_1|$, $|\bL_1^\perp(X_2-X_1)-\bQ_1(X_2-X_1)|$ and $|\bQ_2-\bQ_1|$. We also derive that for some $\bar\mbfA = \begin{bmatrix}
			0 & \bar\bl \\ -\bar\bl^\top & 0
		\end{bmatrix}$ satisfying $\|\bar\bl\|\lesssim_{n, \eps} r$, we have 
		\begin{align*}
			\bL_2^\perp = e^{\mbfA}\circ \bL_1^\perp\circ e^{-\mbfA} = (I_n + \mbfA + O(|\mbfA|^2))\circ \bL_1^\perp \circ (I_n - \mbfA + O(|\mbfA|^2)) \,, 
		\end{align*}
		hence \[
		\left|\bL_1^\perp - \bL_2^\perp - \begin{bmatrix}
			0 & \bar\bl \\ \bar\bl^\top & 0
		\end{bmatrix} \right| \lesssim_n |\bar\mbfA|^2 \lesssim_{n, \eps} r^2
		\]
		While since $\nabla\bQ_1(X_2-X_1) = \begin{bmatrix}
			0 & \nabla\bq_1(y_2) \\ 0 & 0
		\end{bmatrix}$, we thus conclude the estimates on $\|\bL_1^\perp-\bL_2^\perp\|$ and $\left\|\bL_1^\perp \circ \left((\bL_1^\perp - \bL_2^\perp) - \nabla \mbfQ_1( X_2- X_1)\right) \right\|$ from these estimates above and \eqref{Equ_C^2 Reg_Main} again.
		
		Finally, item \ref{Item_Intro_t|(cS_k)_+ Holder} follows from item \ref{Item_Whitney_Concl_t|_K Holder} by letting $\al$ approach $\min\{1, \frac{2}{n-k}\}$.
	\end{proof}
	
	\medskip
	
	\appendix
	
	\section{A family of convex mean curvature flow with slowest asymptotics} \label{App:ExistenceSphericalArrival}
	
	In this section, we will construct rescaled mean curvature flows that are asymptotic to the sphere with prescribed decay. Such flows have been constructed in \cite{Sesum08_Rate_of_Convergence, Strehlke20_Asym_Max}, but we need a variant version for our purpose. Let us first recall some of the results by Strehlke from \cite{Strehlke20_Asym_Max}. For $m\geq 1$, $\bS^m:=S^m(\sqrt{2m})\subset\R^{m+1}$, a rescaled mean curvature flow as a graph of a function $(u(\tau))_{\tau\geq 0}$ over it satisfies the equation
	\begin{equation} \label{Equ_App_LowSpher_RMCF equ}
		\pr_\tau u=L_{\bS^m}u +\sN(u,\nabla u,\nabla^2 u),
	\end{equation}
	where $\sN(u,\nabla u,\nabla^2 u)$ is a quadratic nonlinear term which can be written as $f(u,\nabla u)+\tr (B(u,\nabla u)\nabla^2 u)$ for some smooth function $f,B$ with $f(0,0)=0$, $df(0,0)=0$ and $B(0,0)=0$. For simplicity, we write $\sN(u):=\sN(u,\nabla u,\nabla^2 u)$.

	$-L_{\bS^m}:=-(\Delta_{\SSp^m}+1)$ has eigenvalues $\frac{i(i+m-1)}{2m}-1$ for $i=0,1,2,\cdots$, and first several eigenvalues are $\lambda_1=-1$, $\lambda_2=-\frac12$, $\lambda_3=\frac1m$, and $\lambda_4=\frac{m+6}{2m}$. Let $\mfk{X}_m$ be the space of $L^2$-eigenfunctions of the $-L_{\bS^m}$ with eigenvalue $1/m$. We use $\|\psi\|$ to denote the $L^2$-norm of $\psi\in \mfk{X}_m$, and use $B^{\mfk{X}_m}_r$ to denote the closed ball of radius $r$ in $(\mfk{X}_m, \|\cdot\|)$ centered at $0$. Note that since $\mfk{X}_m$ is a finite-dimensional vector space, all the norms are equivalent over $\mfk{X}_m$. 
	
	Let us state the main results of this section.
	
	\begin{Lem}\label{Lem:AppCMainSpherical}
		Let $m\geq 1$, $\mfk{X}_m$ be defined as above. Then for any $\sigma\in(1/m,2/m)$, there exist a $\varkappa = \varkappa(m, \sigma)>0$ and a $C^1$ map $\Xi: B_\varkappa^{\mfk{X}_m}\to C^2(\bS^m\times \R_{\geq 0})$ such that 
		\begin{enumerate} [label={\normalfont(\roman*)}]
			\item\label{Item_App_LowSph_Xi(0)=0} $\Xi(0) = 0$;
			\item\label{Item_App_LowSph_Xi RMCF} for every $\psi\in B_\varkappa^{\mfk{X}_m}$, $w:=\Xi(\psi)$ satisfies $\|w\|_{C^2}\leq 1$ and $\tau\mapsto \graph_{\bS^m}(w(\cdot, \tau))$ forms a rescaled mean curvature flow.
			\item\label{Item_App_LowSph_Xi-psi decay} for every $\varkappa'\in (0, \varkappa]$ and $\psi_1, \psi_2\in B_{\varkappa'}^{\mfk{X}_m}$, let $w_i:= \Xi(\psi_i)$, $i=1,2$. Then we have, \[
			\|(w_1 - w_2)(\cdot, \tau) - e^{-\frac{\tau}{m}}(\psi_1-\psi_2)\|_{C^4(\bS^m)} \lesssim_{m, \sigma} \varkappa'e^{-\sigma\tau}\|\psi_1 - \psi_2\|_{L^2(\bS^m)}, \qquad \forall\, \tau\geq 0 \,.
			\]
		\end{enumerate}
	\end{Lem}
	Throughout this section, every constant in our estimate is also allowed to depend on the dimension $m$ of the sphere. For notational simplicity, we shall not indicate it in the statements.
	
	\subsection{Estimates of the nonlinearity and the solution}
	
	The following Lemma of Strehlke \cite{Strehlke20_Asym_Max} estimates the nonlinear term $\sN$ in the equation. 
	\begin{Lem}[{\cite[Lemma 3.5]{Strehlke20_Asym_Max}}] \label{Lem:AppC_Nonlinear}
		For $\ell>m/2+1$ and $u, v\in H^\ell(\SSp^m)$ with $\|u\|_{H^{\ell}}\leq 1$, $\|v\|_{H^{\ell}}\leq 1$,
		\begin{equation}
			\begin{split}
				\| \sN(u) - \sN(v)\|_{H^{\ell-1}}
				\lesssim_\ell (\|u\|_{H^{\ell+1}}\|u-v\|_{H^{\ell}}
				+
				\|u\|_{H^{\ell}}\|u-v\|_{H^{\ell+1}}
				).
			\end{split}
		\end{equation}
		
	\end{Lem}
	
	Now let us fix $1/m=\lambda_3<\eta<\sigma<2/m<\lambda_4$. For any integer $\ell>0$, we use $\mfk{Y}_{\ell,\sigma,\eta}$ to denote the space of functions $u:[0,\infty)\to H^{\ell+1}(\bS^m)$ such that 
	\begin{equation} \label{Equ_App_LowSpher_Def ||_l,sigma,eta}
		\|u\|_{\ell,\sigma,\eta}:=\left(\int_0^\infty e^{2\eta\tau}\|u(\tau)\|_{H^{\ell+1}(\bS^m)}^2\ d\tau\right)^{1/2}
		+\sup_{\tau\geq 0}e^{\sigma \tau}\|u(\tau)\|_{H^\ell(\bS^m)}<+\infty.
	\end{equation}
	Note that our defined norm is slightly different from \cite{Strehlke20_Asym_Max} because we need to treat different decay rates. We use $\Pi: L^2(\SSp^m)\to L^2(\SSp^m)$ to denote the $L^2$-orthogonal projection onto the space of eigenfunctions with eigenvalue $\geq\lambda_4$.  We need the following two estimates that are variant versions of Corollaries 3.3 and 3.4 in \cite{Strehlke20_Asym_Max}.
	\begin{Lem}\label{Lem:AppCNormEst}
		If $v(s)$ is a continuously differentiable path in $H^{\ell}$ with $\Pi v(s)=v(s)$ for all $s\geq 0$, then for $1/m<\eta<\sigma<2/m$ and integer $\ell\geq 1$,
		\begin{equation} \label{Equ_App_LowSph_C^0_tau est by Strehk}
			e^{2\sigma \tau}\|v(\tau)\|_{H^\ell}^2
			\leq 
			\|v(0)\|_{H^{\ell}}^2+\frac{\lambda_4}{2(\lambda_4-\sigma)}\int_0^\tau e^{2\sigma s}\|(\pr_s-L)v(s)\|_{H^{\ell-1}}^2\ ds,
		\end{equation}
		and if there is a sequence of $\tau_j\to\infty$ such that $e^{\eta\tau}\|v(\tau_j)\|_{H^{\ell}}\to 0$, then
		\begin{equation} \label{Equ_App_LowSph_L^2_tau est by Strehk}
			\int_0^\infty e^{2\eta\tau}\|v(\tau)\|_{H^{\ell+1}}^2\  d\tau
			\lesssim_{\ell, \eta} \|v(0)\|_{H^{\ell}}^2
			+ \int_0^\infty e^{2\eta\tau}\|(\pr_\tau-L)v(\tau)\|_{H^{\ell-1}}^2 \ d\tau.
		\end{equation}
	\end{Lem}
	The proof is essentially the same as Corollaries 3.3 and 3.4 in \cite{Strehlke20_Asym_Max}, based on Lemma 3.2 therein, so we skip the proof here.
	
	For any $(u,\phi)\in \mfk{Y}_{\ell,\sigma,\eta}\times \mfk{X}_m$ with $\|u\|_{\ell, \sigma, \eta}+\|\phi\|_{H^\ell}\leq 1$, we define $U(\cdot;u,\phi)\in \mfk{Y}_{\ell,\sigma,\eta}$ to be the unique solution to the equation
	\begin{equation} \label{Equ_App_LowSpher_Def U}
		\begin{cases}
			(\pr_\tau-L)U(\cdot; u,\phi)=\sN(u+e^{-\frac{\tau}{m}}\phi),
			\\
			\Pi\ U(\cdot, 0;u,\phi)=0.
		\end{cases}
	\end{equation}
	The existence of the solution in $\mfk{Y}_{\ell,\sigma,\eta}$ is a direct consequence of the semigroup argument together with Lemma \ref{Lem:AppC_Nonlinear}. In fact, we have an explicit expression of $U$:
	\begin{equation}\label{eq:AppCExpressionU}
		U(\cdot, \tau;u,\phi)
		= \int_0^\tau e^{(\tau-s)L}\Pi \sN(u+e^{-\frac{s}{m}}\phi)\ ds 
		- \int_\tau^\infty e^{(\tau-s)L}(1-\Pi )\sN(u+e^{-\frac{s}{m}}\phi)\ ds \,,
	\end{equation}
	where $e^{\tau L}$ denotes the heat semi-group for $L=L_{\SSp^m}$, and we can write down $U(0;u,\phi)$ explicitly:
	\[
	U(\cdot, 0;u,\phi)=-\int_0^\infty e^{-sL}(1-\Pi)\sN(u+e^{-\frac{s}{m}}\phi)\ ds.
	\]
	
	It follows directly from \eqref{eq:AppCExpressionU} that $U(\cdot; 0, 0)=0$. The following stronger estimate for $U$ allows us to conclude that $U(\cdot, u, \phi)\in \mfk{Y}_{\ell,\sigma, \eta}$ by taking $(v, \psi)=(0,0)$ therein.
	
	\begin{Lem}\label{Lem:AppCContraction}
		For $\ell>m/2+1$, there exist $r_{\ref{Lem:AppCContraction}}=r_{\ref{Lem:AppCContraction}}(\ell, \sigma, \eta)>0$ such that when $u, v\in \mfk{Y}_{\ell, \sigma, \eta}$ and $\phi, \psi\in \mfk{X}_m$ with $(\|u\|_{\ell,\sigma,\eta}+\|v\|_{\ell,\sigma,\eta}+\|\phi\|+\|\psi\|)\leq r_{\ref{Lem:AppCContraction}}$, we have
		\begin{equation}
			\| U(\cdot;u,\phi)-U(\cdot;v,\psi) \|_{\ell,\sigma,\eta}
			\lesssim_{\ell, \sigma, \eta} (\|u\|_{\ell,\sigma,\eta}+\|v\|_{\ell,\sigma,\eta}+\|\phi\|+\|\psi\|)(\|u-v\|_{\ell,\sigma,\eta}+\|\phi-\psi\|) \,.
		\end{equation}
	\end{Lem}
	\begin{proof}
		Let $D(\tau)=U(\cdot, \tau;u,\phi)-U(\cdot, \tau;v,\psi)$, then $D$ satisfies the equations
		\[
		\begin{split}
			(\pr_\tau-L)D & = \sN(u+e^{-\frac{\tau}{m}}\phi) - \sN(v+e^{-\frac{\tau}{m}}\psi) \,,  \\
			D(\cdot, 0) & = -\int_0^\infty e^{-Ls}(1-\Pi)\big(
			\sN(u+e^{-\frac{\tau}{m}}\phi) - \sN(v+e^{-\frac{\tau}{m}}\psi)
			\big)\ ds \,.
		\end{split}
		\]
		Then the proof is similar to the proof of Theorem 3.7 in \cite{Strehlke20_Asym_Max}\footnote{In \cite{Strehlke20_Asym_Max}, it was labeled as ``proof of Theorem 3.3'', and it seems to be a typo.}, but we need to single out the terms $e^{-\frac{\tau}{m}}\phi$ and $e^{-\frac{\tau}{m}}\psi$, as they are not in $\mfk{Y}_{\ell,\sigma,\eta}$. 
		For the readers' convenience, we sketch the proof as follows. From the expression of $D$, 
		\[
		\begin{split}
			D(\tau) = &\, \int_0^\tau e^{(\tau-s)L}\Pi \big(\sN(u+e^{-\frac{s}{m}}\phi)-\sN(v+e^{-\frac{s}{m}}\psi)\big) \ ds \\
			& - \int_\tau^\infty e^{(\tau-s)L}(1-\Pi )\big(\sN(u+e^{-\frac{s}{m}}\phi)-\sN(v+e^{-\frac{s}{m}}\psi)\big) \ ds \,.
		\end{split}
		\]
		When $\tau\geq s$, $\|e^{(\tau-s)L}\|\leq e^{-\lambda_4(\tau-s)}\leq e^{-\frac{2}{m}(\tau-s)}$ on $\Pi (H^{\ell})$, and when $\tau\leq s$, $\|e^{(\tau-s)L}\|\leq e^{-\frac{1}{m}(\tau-s)}=e^{\frac{1}{m}(s-\tau)}$ on $(1-\Pi)(H^{\ell})$.
		Now we estimate $\|(1-\Pi )D(\tau)\|_{\ell,\sigma,\eta}$ and $\|\Pi D(\tau)\|_{\ell,\sigma,\eta}$ separately. 
		
		\noindent {\bf $\bullet$ Estimate of $\|(1-\Pi) D(\tau)\|_{{\ell,\sigma,\eta}}$.} By the expression of the solution,
		\[
		(1-\Pi)D(\tau)
		= -\int_\tau^\infty e^{(\tau-s)L}(1-\Pi)\big(\sN(u+e^{-\frac{s}{m}}\phi)-\sN(v+e^{-\frac{s}{m}}\psi)\big)ds.
		\]
		Therefore, if we denote for simplicity $\phi_s:= e^{-\frac{s}{m}}\phi$ and $\psi_s:=e^{-\frac{s}{m}}\psi$, then
		\begin{equation}
			\begin{split}
				e^{\sigma\tau}\|(1-\Pi)D(\tau)\|_{H^{\ell}}
				& \lesssim_\ell \int_\tau^\infty e^{\frac{1}{m}(s-\tau)+\sigma\tau}\|(1-\Pi)\big(\sN(u+\phi_s)-\sN(v+\psi_s)\big)\|_{H^{\ell-1}}\ ds    \\ 
				& \lesssim_\ell \int_\tau^\infty e^{\frac{1}{m}(s-\tau)+\sigma\tau}\|\sN(u+\phi_s)-\sN(v+\psi_s)\|_{H^{\ell-1}}ds.
			\end{split}
		\end{equation}
		Here we use the fact that $(1-\Pi)$ is a finite rank operator and all the norms are equivalent. Then by Lemma \ref{Lem:AppC_Nonlinear} and triangle inquality,
		\[
		\begin{split}
			& \int_\tau^\infty e^{\frac{1}{m}(s-\tau)+\sigma\tau}\|\sN(u+\phi_s)-\sN(v+\psi_s)\|_{H^{\ell-1}}ds \\
			&\quad \lesssim_\ell \int_\tau^\infty e^{\frac{1}{m}(s-\tau)+\sigma\tau} \Big(
			(\|u(s)\|_{H^{\ell+1}} + e^{-\frac{s}{m}}\|\phi\|)\cdot
			(\|u(s)-v(s)\|_{H^{\ell}}+e^{-\frac{s}{m}}\|\phi-\psi\|) \\ 
			& \qquad\qquad\qquad\qquad\;\;\; + (\|v(s)\|_{H^{\ell}}+e^{-\frac{s}{m}}\|\psi\|)\cdot (\|u(s)-v(s)\|_{H^{\ell+1}}+e^{-\frac{s}{m}}\|\phi-\psi\|)
			\Big)\ ds.
		\end{split}
		\]
		We can bound it by $(\|u\|_{\ell,\sigma,\eta}+\|v\|_{\ell,\sigma,\eta}+\|\phi\|+\|\psi\|)(\|u-v\|_{\ell,\sigma,\eta}+\|\phi-\psi\|)$. For example:
		\[
		\begin{split}
			\int_\tau^\infty e^{\frac{1}{m}(s-\tau)+\sigma\tau} \|u(s)\|_{H^{\ell+1}}\|u(s)-v(s)\|_{H^{\ell}} \ ds
			\leq \|u(s)-v(s)\|_{\ell,\sigma,\eta} \int_\tau^\infty e^{(\frac{1}{m}-\sigma)(s-\tau)} \|u(s)\|_{H^{\ell+1}}\ ds\,,
		\end{split}
		\]
		and by the Cauchy-Schwarz inequality,
		\[
		\begin{split}
			\int_\tau^\infty e^{(\frac{1}{m}-\sigma)(s-\tau)} \|u(s)\|_{H^{\ell+1}} \ ds
			& \leq \left(
			\int_\tau^\infty e^{2(\frac{1}{m}-\sigma)(s-\tau)-2\eta s}\ ds 
			\right)^{1/2} \left(
			\int_\tau^\infty e^{2\eta s} \|u(s)\|^2_{H^{\ell+1}}\ ds
			\right)^{1/2}  \\
			& \lesssim_{\sigma, \eta}\|u(s)\|_{\ell,\sigma,\eta}.
		\end{split}
		\]
		Another term can be estimated as follows:
		\[
		\begin{split}
			\int_\tau^\infty e^{\frac{1}{m}(s-\tau)+\sigma\tau} \|u(s)\|_{H^{\ell+1}}\cdot e^{-\frac{s}{m}}\|\phi-\psi\|_{H^{\ell}}\ ds
			\lesssim_\ell \|\phi-\psi\|\int_\tau^\infty e^{-\frac{1}{m}\tau+\sigma\tau} \|u(s)\|_{H^{\ell+1}}\ ds \,,
		\end{split}
		\]
		and similarly, Cauchy-Schwarz inequality yields the desired estimate. All the rest terms can be estimated similarly.
		
		Once we bound $e^{\sigma\tau}\|(1-\Pi)D(\tau)\|_{H^{\ell}}$ as above, because $(1-\Pi)$ has finite rank, we also have \[
		e^{\eta \tau}\|(1-\Pi)D(\tau)\|_{H^{\ell+1}}\lesssim_{\ell, \sigma, \eta} e^{-(\sigma-\eta)\tau}(\|u\|_{\ell,\sigma,\eta}+\|v\|_{\ell,\sigma,\eta}+\|\phi\|+\|\psi\|)(\|u-v\|_{\ell,\sigma,\eta}+\|\phi-\psi\|) \,,
		\]
		for every $\tau\geq 0$. Integrating over $(0, +\infty)$ and using $\eta<\sigma$ show that
		\[
		\|(1-\Pi)D\|_{\ell,\sigma,\eta}\lesssim_{\ell, \sigma, \eta} (\|u\|_{\ell,\sigma,\eta}+\|v\|_{\ell,\sigma,\eta}+\|\phi\|+\|\psi\|)(\|u-v\|_{\ell,\sigma,\eta}+\|\phi-\psi\|).
		\]
		{\bf $\bullet$ Estimate of $\|\Pi D(\tau)\|_{\ell,\sigma,\eta}$.} We observe that $\Pi D(0)=0$. Then by \eqref{Equ_App_LowSph_C^0_tau est by Strehk} of Lemma \ref{Lem:AppCNormEst},
		\[
		\begin{split}
			e^{2\sigma \tau}\|\Pi D(\tau)\|_{H^{\ell}}^2
			&\leq  
			\frac{\lambda_4}{2(\lambda_4-\sigma)}\int_0^\tau e^{2\sigma s}
			\left\|
			\Pi \big(\sN(u+e^{-\frac{s}{m}}\phi)-\sN(v+e^{-\frac{s}{m}}\psi)\big)
			\right\|^2_{H^{\ell-1}}ds
			\\
			&\leq 
			\frac{\lambda_4}{2(\lambda_4-\sigma)}\int_0^\tau e^{2\sigma s}
			\left\| 
			\sN(u+e^{-\frac{s}{m}}\phi)-\sN(v+e^{-\frac{s}{m}}\psi) 
			\right\|^2_{H^{\ell-1}}ds.
		\end{split}
		\]
		Then Lemma \ref{Lem:AppC_Nonlinear} shows that
		\[
		\begin{split}
			e^{2\sigma \tau}\|\Pi D(\tau)\|_{H^{\ell}}^2 
			& \lesssim_{\ell, \sigma} \int_0^\tau e^{2\sigma s} \Big(
			\|u+e^{-\frac{s}{m}}\phi\|^2_{H^{\ell+1}}\cdot \|(u-v) + e^{-\frac{s}{m}}(\phi-\psi)\|^2_{H^{\ell}} \\
			& \qquad \qquad \quad + \|v+e^{-\frac{s}{m}}\psi\|^2_{H^{\ell}}\cdot \|(u-v)+e^{-\frac{s}{m}}(\phi-\psi)\|^2_{H^{\ell+1}}
			\Big)\ ds \,.
		\end{split}
		\]
		which is bounded by
		\[\begin{split}
			& \int_0^\infty e^{2\sigma s}
			\Big(
			(\|u\|^2_{H^{\ell+1}}+e^{-\frac{2s}{m}}\|\phi\|^2) \cdot (\|u-v\|^2_{H^{\ell}}+e^{-\frac{2s}{m}}\|\phi-\psi\|^2) \\\
			& \qquad\quad\; +
			(\|v\|^2_{H^{\ell}} + e^{-\frac{2s}{m}}\|\psi\|^2) \cdot (\|u-v\|^2_{H^{\ell+1}} + e^{-\frac{2s}{m}}\|\phi-\psi\|^2)
			\Big)\ ds \\
			& \lesssim_{\ell, \sigma, \eta} (\|u\|_{\ell,\sigma,\eta}+\|v\|_{\ell,\sigma,\eta}+\|\phi\|+\|\psi\|)^2(\|u-v\|_{\ell,\sigma,\eta}+\|\phi-\psi\|)^2 \,.
		\end{split}
		\]
		Here we use the fact that $e^{2\sigma s}=e^{2\eta s}e^{2(\sigma-\eta)s}$, and $\sigma-\eta-1/m<\sigma-2/m<0$.
		
		Similarly, by \eqref{Equ_App_LowSph_L^2_tau est by Strehk} of Lemma \ref{Lem:AppCNormEst}, \[
		\int_0^\infty e^{2\eta\tau}\|\Pi D(\tau)\|_{H^{\ell+1}}^2 \ d\tau
		\lesssim_{\ell, \eta} \int_0^\infty e^{2\eta\tau}\|\sN(u+e^{-\frac{s}{m}}\phi)-\sN(v+e^{-\frac{s}{m}}\psi)\|_{H^{\ell-1}}^2\ d\tau.
		\]
		Because $e^{2\eta\tau}\leq e^{2\sigma\tau}$ when $\tau\geq 0$, we obtain similar bound as above. Therefore, we have \[
		\|\Pi D\|_{\ell,\sigma,\eta} 
		\lesssim_{\ell, \sigma, \eta}(\|u\|_{\ell,\sigma,\eta} + \|v\|_{\ell,\sigma,\eta}+\|\phi\|+\|\psi\|)(\|u-v\|_{\ell,\sigma,\eta}+\|\phi-\psi\|) \,.
		\]
		Combining the estimates of $\|(1-\Pi )D\|_{\ell,\sigma,\eta}$ and $\|\Pi D\|_{\ell,\sigma,\eta}$ above gives the desired inequality.
	\end{proof}

	\subsection{Proof of Lemma \ref{Lem:AppCMainSpherical}}
	\begin{proof}
		We first construct the map $\Xi$. Fix a sufficiently large integer $\ell=\ell(m)$ so that $H^\ell(\SSp^m)$ is embedded in $C^4(\SSp^m)$; and fix $1/m<\eta<\sigma<2/m$. Let $\mfk{P}:=\mfk{Y}_{\ell,\sigma,\eta}\times \mfk{X}_n$, equipped with the norm $\|(u,\phi)\|_{\mfk{P}}=\|u\|_{\ell,\sigma,\eta}+\|\phi\|$. We define a map $\mfk T: \mfk{P}\to \mfk{P}$ by 
		\begin{align}
			\mfk T(u,\phi) := (U(\cdot; u,\phi),0) \,. \label{Equ_App_LowSpher_Define mfkT}
		\end{align}
		\textbf{Claim.} {\it
			There exist $0<\varkappa<\kappa<r_{\ref{Lem:AppCContraction}}(\ell, \sigma, \eta)$ such that $\id-\mfk T$ is a $C^1$ diffeomorphism from a subdomain of $B^{\mfk{P}}_{\kappa}$ onto $B_{\varkappa}^{\mfk{P}}$.
		}
		\begin{proof}[Proof of the Claim.]
			First note that $\id-\mfk T$ is $C^1$ from the explicit expression of $U$ in \eqref{eq:AppCExpressionU}. Also by Lemma \ref{Lem:AppCContraction}, for every $(u_i, \phi_i)\in \mfk{P}$ with $\|(u_i, \phi_i)\|_{\mfk{P}}\leq r_{\ref{Lem:AppCContraction}}/2$, we have 
			\begin{align}
				\|U(\tau;u_1,\phi_1)-U(\tau;u_2,\phi_2)\|_{\ell,\sigma,\eta} \leq \bar C(\|(u_1,\phi_1)\|_{\mfk{P}}+\|(u_2,\phi_2)\|_{\mfk{P}})\cdot \|(u_1,\phi_1)-(u_2,\phi_2)\|_{\mfk{P}} \,.  \label{Equ_App_LowSpher_|U_1-U_2| est}
			\end{align}
			where $\bar C=\bar C(\ell, \sigma, \eta)$. Therefore, by fixing the choice of $\varkappa\in (0, \min\{r_{\ref{Lem:AppCContraction}}/2,1/(8\bar C)\})$ and $\kappa=2\varkappa$, we have $D(\id-\mfk T)$ is invertible at every point in $B^{\mfk{P}}_\kappa$.
			
			We are left to show that for every $(v, \psi)\in B^{\mfk{P}}_\varkappa$, there's a unique element in $B^{\mfk{P}}_\kappa\cap (\id-\mfk T)^{-1}(v, \psi)$. To see this, we define a map \[
			\mfk F_{v, \psi}: \mfk{P} \to \mfk{P}, \quad \mfk F(u,\phi):=(v,\psi) + \mfk T(u,\phi) \,.
			\]
			By the choice of $\kappa, \varkappa$ above (and recall that $U(\cdot;0,0)=0$), one can easily check that $\mfk F_{v, \psi}$ is a contraction map from $B^{\mfk{P}}_\kappa$ to itself, and thus has a unique fixed point $(\bar v, \bar \psi)\in B^{\mfk{P}}_\kappa$. This shows that $B^{\mfk{P}}_\kappa\cap (\id-\mfk T)^{-1}(v, \psi) = \{(\bar v, \bar \psi)\}$ consists of a single element.
		\end{proof}
		Now for every $\psi\in \mfk{X}_m$ with $\|\psi\|\leq \varkappa$, we define \[
		\Xi(\psi)=\xi(\psi)+e^{-\frac{\tau}{m}}\psi\,,
		\]
		where $(\xi(\psi),\psi):=(\id-\mfk T)^{-1}(0,\psi)$. Also by the Claim, $\Xi$ is a $C^1$-map into $\mfk{Y}_{\ell, \sigma, \eta}$. Let us verify that $\Xi$ satisfies all the items in Lemma \ref{Lem:AppCMainSpherical}.
		
		Item \ref{Item_App_LowSph_Xi(0)=0}: Because $(\id-\mfk T)(0,0)=(0,0)$, we have $(\xi(0), 0)=(\id-\mfk T)^{-1}(0,0)=(0, 0)$. Thus $\Xi(0)=\xi(0)+0=0$.
		
		Item \ref{Item_App_LowSph_Xi RMCF}: Since by \eqref{Equ_App_LowSpher_Define mfkT}, \[
		(0, \psi) = (\id - \mfk T)(\xi(\psi), \psi) = (\xi(\psi)- U(\cdot; \xi(\psi), \psi), \psi)\,,
		\] 
		we see that $\Xi(\psi)=U(\cdot;\xi(\psi);\psi)+e^{-\frac{\tau}m}\psi$ satisfies \[
		(\partial_\tau-L)\Xi(\psi) = (\partial_\tau-L)U(\cdot; \xi(\psi);\psi) = \sN(\xi(\psi)+e^{-\frac{\tau}{m}}\psi) = \sN(\Xi(\psi))\,,
		\]
		where the second equality follows from \eqref{Equ_App_LowSpher_Def U}. Thus $\Xi(\psi)$ solves the rescaled mean curvature flow equation \eqref{Equ_App_LowSpher_RMCF equ}. By the choice of $\ell$ and the Sobolev embedding theorem, $\mfk{Y}_{\ell,\sigma,\eta}$ is embedded in $L^\infty((0,+\infty), C^4(\bS^m))$. Hence the fact that $\Xi(\psi)\in \mfk{Y}_{\ell, \sigma, \eta}$ solves the rescaled mean curvature flow equation implies that $\Xi$ is a $C^1$-map into $C^2(\bS^m\times\R_{\geq 0})$ as well; and when $\varkappa=\varkappa(m, \sigma)$ is chosen sufficiently small, by Sobolev embedding, $\|\Xi(\psi)\|_{C^2}\leq \|\xi(\psi)\|_{C^2}+\|e^{-\frac{\tau}{m}}\psi\|_{C^2}\leq 1$.
		
		Item \ref{Item_App_LowSph_Xi-psi decay}: Let $w_i=\Xi(\psi_i)$, then $u_i:=w_i-e^{-\frac{\tau}{m}}\psi_i=U(\cdot;u_i,\psi)$, and by Lemma \ref{Lem:AppCContraction},
		\[
		\begin{split}
			\|u_1 - u_2\|_{\ell,\sigma,\eta} 
			& \lesssim_{\ell, \sigma, \eta} (\|u_1\|_{\ell,\sigma,\eta}+\|u_2\|_{\ell,\sigma,\eta}+\|\psi_1\|+\|\psi_2\|)(\|u_1-u_2\|_{\ell,\sigma,\eta}+\|\psi_1-\psi_2\|) \\
			& \lesssim_{\ell, \sigma, \eta} \varkappa (\|u_1-u_2\|_{\ell,\sigma,\eta}+\|\psi_1-\psi_2\|) \,.
		\end{split}
		\]
		Since $\ell, \eta$ are fixed with respect to $m, \sigma$, by choosing $\varkappa(m, \sigma)>0$ sufficiently small, $\|u_1-u_2\|_{\ell,\sigma,\eta}$ on the right-hand side can be absorbed by the left-hand side, and hence \[
		\|u_1-u_2\|_{\ell,\sigma,\eta}\lesssim_{m, \sigma} \varkappa\|\psi_1-\psi_2\| \,.
		\]
		Recalling the definition of $\|\cdot\|_{\ell, \sigma, \eta}$ in \eqref{Equ_App_LowSpher_Def ||_l,sigma,eta}, combine this with Sobolev embedding and notice that $u_i=w_i-e^{-\frac{\tau}{m}}\psi_i$, we get
		\[
		\|(w_1 - w_2)(\cdot, \tau) - e^{-\frac{\tau}m}(\psi_1-\psi_2)\|_{C^4(\bS^m)} \lesssim_{m, \sigma} \varkappa e^{-\sigma\tau}\|\psi_1 - \psi_2\|_{L^2(\bS^m)} \,.
		\]
		The discussion above yields the same estimate if $\varkappa$ is replaced by any $\varkappa'\in (0, \varkappa]$, and $\|\psi_i\|\leq \varkappa'$.
	\end{proof}
	
	\begin{remark}\label{rmk_App_spherical_rot}
		Suppose $\mbfB\in \rmO(m+1)$ is an element in the rotation group of the sphere, and $\psi\in B_r^{\mfk{X}_m}$. Note that the expression of rescaled mean curvature flow and the linearized operator is invariant under rotation. Therefore, the expression of $\sN$ is also invariant under rotation, i.e., $\sN(f\circ \mbfB)=\sN(f)\circ \mbfB$. This implies that $U(\tau;u\circ \mbfB,\phi\circ \mbfB)=U(\tau;u,\phi)\circ \mbfB$. As a consequence, $\mfk T(u\circ \mbfB,\phi\circ \mbfB)=\mfk T(u,\phi)\circ\mbfB$, which implies that $\Xi(\psi\circ\mbfB)=\Xi(\psi)\circ\mbfB$. In particular, by the uniqueness in the fixed point theorem, this implies that if $\psi\circ\mbfB=\psi$, then $\Xi(\psi)\circ\mbfB=\Xi(\psi)$.
	\end{remark}
	
	\subsection{Proof of Lemma \ref{Lem_L^2Noncon_LowSphericalMode}.}
	We first construct a family of spherical arrival time functions $\scV$ in $\R^{n-k+1}$, and then define \[
	\scU:=\{\bu:\bu(x,y)=\bu_\circ(x),\, \bu_\circ\in\scV\}.
	\]
	As a family of rescaled mean curvature flows, $\scV$ is explicitly constructed in Appendix \ref{App:ExistenceSphericalArrival} as the image of the closed ball of radius $\kappa_{\ref{Lem_L^2Noncon_LowSphericalMode}}(n)\leq \varkappa$ in Lemma \ref{Lem:AppCMainSpherical} with $m=n-k$ and $\sigma=3/(2m)=\frac32 \gamma_{n,k}$. In fact, any graphical function constructed in Appendix \ref{App:ExistenceSphericalArrival}, denoted by $\varphi_{\bu_\circ}$, is the rescaled mean curvature flow for some spherical arrival time function $\bu_\circ$ in $\R^{n-k+1}$. We check that when $\kappa_{\ref{Lem_L^2Noncon_LowSphericalMode}}(n)$ is chosen sufficiently small, $\scU$ satisfies the properties of this lemma.
	
	To prove item \ref{item_0_Lem_L^2Noncon_LowSphericalMode}, because we extend $\bu$ based on the spherical part, we only need to consider those $\mbfB$ that rotate the spherical part of $\cC_{n,k}$. Then the proof of item \ref{item_0_Lem_L^2Noncon_LowSphericalMode} is a consequence of Remark \ref{rmk_App_spherical_rot}.
	
	To prove item \ref{item_1_Lem_L^2Noncon_LowSphericalMode}, we only need to check that there exists $\beta(n)>0$ so that \ref{Item_ArrivalT2} holds for all $\bu_\circ\in\scV$. Let us estimate $-\nabla^2\bu_\circ$. In \cite[Page 189]{Huisken93_ArrivalTimeC2}, Huisken showed that at the singular point, $D^2\bu_{\circ}=-I/(n-k)$; at any regular point, for any unit vector $\nu$, 
	\[
	D_\nu D_\nu \bu_{\circ}=-\frac{1}{(n-k)}-\left(\frac{\Delta H}{H^3}+\frac{|A|^2-H^2/(n-k)}{H^2}\right).
	\]
	Note that the above quantity is exactly $-1/(n-k)$ for the shrinking sphere, and in general, $D^2\bu_\circ+I/(n-k)$ is uniformly bounded by the $C^4$-norm of $\varphi_{\bu_\circ}$, which is small whenever $\kappa_{\ref{Lem_L^2Noncon_LowSphericalMode}}$ small, in terms of item \ref{item_2_Lem_L^2Noncon_LowSphericalMode}. This gives the uniform Hessian bound for $\bu_\circ$. 
	
	To prove item \ref{item_2_Lem_L^2Noncon_LowSphericalMode}, because $\varphi_\bu(\theta,y)=\varphi_{\bu_\circ}(\theta)$, we only need to prove the desired estimates for $\bu_\circ$. The estimates are direct results of items \ref{Item_App_LowSph_Xi(0)=0} and \ref{Item_App_LowSph_Xi-psi decay} of Lemma \ref{Lem:AppCMainSpherical} by the choice of $\sigma$ above, and possibly choosing $\kappa_{\ref{Lem_L^2Noncon_LowSphericalMode}}(n)$ smaller.

	\section{Analysis of parabolic Jacobi equation} \label{Append_Ana Parab Jac}
	
	\begin{Prop}\label{Prop_App_Solve Parab Jac equ w spacetime C^0 est}
		Suppose $\eps\in (0, 1)$, $\lambda\in \RR\setminus \BB_\eps(\sigma(\cC_{n,k}))$, $\lambda'\geq \lambda+1+\eps$, $|\lambda|, |\lambda'|\leq \eps^{-1}$. Suppose $h\in C^\al(\cC_{n,k}\times \RR_{\geq 0})$ satisfies the pointwise bound, \[
		|h(\theta, y, \tau)| \leq e^{-\lambda \tau}(1+|y|^2)^{\lambda'}\,.
		\]
		Then there exists a solution $w\in C^{2,\al}(\cC_{n,k}\times \RR_{\geq 0})$ to 
		\begin{align*}
			\begin{cases}
				(\partial_\tau - L_{n,k})w = h\,, & \text{ on }\cC_{n,k}\times \RR_{>0}\,, \\
				\Pi_{\geq \lambda}w(\cdot, 0) = 0\,, & \text{ on } \cC_{n,k} \,.
			\end{cases}
		\end{align*}
		with pointwise estimate, \[
		|w(\theta, y, \tau)| \lesssim_{n,\eps} e^{-\lambda \tau}(1+|y|^2)^{\lambda'} \,.
		\]
	\end{Prop}
	
	\begin{Prop} \label{Prop_App_Parab equ Spacetime C^0 bd from initial data}
		Suppose $\eps\in (0, 1)$, $\eps'\in (0, \eps]$; $\lambda, \lambda', \mu>0$ satisfy 
		\begin{align*}
			&  [\mu - \eps', \mu + \eps']\cap \sigma(\cC_{n,k}) \subset \{\mu\}\,, & &
			\mu+\eps\leq \lambda\leq \min(\sigma(\cC_{n,k})\cap \RR_{>\mu})-\eps\,, \\
			&
			|\lambda|, |\lambda'|\leq \eps^{-1}\,, & &
			\lambda'\geq \lambda+1+\eps\,.
		\end{align*}  
		Let $w\in C^2_{loc}\cap L^2(\cC_{n, k}\times \RR_{\geq 0})$ be a solution to 
		\begin{align*}
			(\partial_\tau - L_{n,k})w = 0\,,   
		\end{align*}
		on $\cC_{n,k}\times \RR_{>0}$, satisfying 
		\begin{align*}
			\limsup_{\tau\to +\infty} e^{(\mu-\eps')\tau}\|w(\cdot, \tau)\|_{L^2} < +\infty\,, & &
			\liminf_{\tau\to +\infty} e^{(\mu+\eps')\tau}\|w(\cdot, \tau)\|_{L^2} > 0\,,
		\end{align*}
		and pointwise bound on initial data, \[
		|w(\theta, y, 0)| \leq (1+|y|^2)^{\lambda'}\,.
		\]
		Then $\mu\in \sigma(\cC_{n,k})$, $\psi:= \Pi_{=\mu}(w(\cdot, 0))$ is non-zero, and the following pointwise estimate holds:
		\begin{align*}
			|w(\theta, y, \tau) - e^{-\mu\tau}\psi(\theta, y)| \lesssim_{n, \eps} e^{-\lambda\tau}(1+|y|^2)^{\lambda'}\,.
		\end{align*}
	\end{Prop}

	In the following, we fix $n,k$, and let us denote $e^{\tau L}$ as the heat semi-group for $L_{n,k}$. The proof relies on the following estimates of the parabolic Jacobi fields. First, we prove a crude bound on the parabolic Jacobi fields.

	\begin{lemma}\label{lem:growth_parabolic_Jacobi_proj_large}
		Suppose $\lambda'\geq \lambda+1+\eps$, $|\lambda|, |\lambda'|\leq \eps^{-1}$; $|h(\theta,y)|\leq (1+|y|^2)^{\lambda'}$ and $\Pi_{<\lambda}h=0$, then
		\begin{equation}
			|e^{\tau L}h(\theta,y)|\lesssim_{n, \eps} e^{-\lambda\tau}(1+|y|^2)^{\lambda'}.
		\end{equation}
	\end{lemma}
	
	\begin{proof}
		Recall that $L=L_{n,k}=\Delta_{\cC_{n,k}}-\frac{y}{2}\cdot \nabla_y+ 1$. We first observe that,  
		\begin{align*}
			& (\pr_\tau-L)\left(e^{-\lambda\tau}(1+|y|^2)^{\lambda'}\right) \\
			& = e^{-\lambda\tau}(1+|y|^2)^{\lambda'-1}
			\left(
			(-\lambda-1-2k\lambda')+(\lambda'-\lambda-1)|y|^2
			-\frac{2\lambda'(2\lambda'-2)|y|^2}{1+|y|^2}
			\right) \,.
		\end{align*}
		This implies that $\exists\, R=R(n, \eps)>0$ such that when $|y|\geq R$, $(\pr_\tau-L)(e^{-\lambda\tau}(1+|y|^2)^{\lambda'})>0$.
		
		Now we will estimate $|e^{\tau L}h(\theta,y)|$ for $|y|\leq R$ and $|y|>R$ separately. First, by the Fourier decomposition and $\Pi_{<\lambda}h=0$,  \[
		\|e^{\tau L}h\|_{L^2} \leq e^{-\lambda\tau}\|h\|_{L^2}\lesssim_{n, \eps} e^{-\lambda\tau} \,.
		\]
		Then by the local weak maximum principle of parabolic PDE (see e.g. \cite[Chapter VII Section 9]{Lieberman96_ParabolicPDE}), we have 
		\begin{equation}\label{eq:Parabolic_Jacobi_Less_R}
			\sup_{|y|\leq R}|e^{\tau L}h(\theta,y)|
			\lesssim_{n, \eps} \|e^{\tau L}h\|_{L^2_c(B_{2R})} 
			\lesssim_{n, \eps} \|e^{\tau L}h\|_{L^2}
			\lesssim_{n, \eps} e^{-\lambda\tau} \,.
		\end{equation}
		Here $\|\cdot\|_{L^2_c}$ is the classical un-weighted $L^2$ norm. We used the fact that $R=R(n, \eps)$.
		
		Next we estimate $|e^{\tau L}h(\theta,y)|$ for $|y|>R$. We need the following maximum principle on the unbounded region. Recall that $Q_R=B^{n-k+1}_R\times B_R^k$.
		
		\begin{lemma}
			Suppose $R>2n$, $w(\cdot,\tau)\in L^2(\cC_{n,k}\backslash Q_R)$ satisfies \[
			\begin{cases}
				(\pr_\tau-L)w(\theta,y,\tau)\leq 0,& \text{on $(\cC_{n,k}\backslash Q_R)\times\R_{>0}$}, \\
				w(\theta,y,\tau)\leq 0,& \text{on $\partial(\cC_{n,k}\backslash Q_R)\times\R_{>0}$}, \\
				w(\theta,y,0)\leq 0,&\text{on $\cC_{n,k}\backslash Q_R$}.
			\end{cases}  \]
			Then $w(\theta,y,\tau)\leq 0$ on $(\cC_{n,k}\backslash Q_R)\times\R_{>0}$.
		\end{lemma}
		\begin{proof}
			Let $w_+=\max\{0,w\}$ and $d\mu:= e^{-|y|^2/4}dyd\theta$. Then 
			\begin{align*}
				0 & \geq \int_{\cC_{n,k}\backslash Q_R}(\pr_\tau w-Lw)w_+d\mu \\
				& = \frac{1}{2}\pr_\tau\int_{\cC_{n,k}\backslash Q_R}w_+^2
				+ \int_{\cC_{n,k}\backslash Q_R}|\nabla w_+|^2 d\mu
				- \int_{\cC_{n,k}\backslash Q_R}ww_+d\mu
				+ \int_{\cC_{n, k}\cap\pr B_R}\pr_{\nu}w\, w_+d\mu \\
				& \geq \frac{1}{2}\pr_\tau\int_{\cC_{n,k}\backslash B_R}w_+^2
				- \int_{\cC_{n,k}\backslash B_R}w_+^2d\mu \,.
			\end{align*}
			where $\nu$ denotes the outward unit normal of $\cC_{n,k}\cap Q_R$ in $\cC_{n,k}$, and note that $\nabla w_+=0$ a.e. on $\{w_+=0\}$.
			Here in the last inequality, we use the fact that $w_+=0$ on $\pr Q_R$. Therefore we have $\pr_\tau [e^{-2\tau}\int_{\cC_{n,k}\backslash Q_R} w_+^2d\mu]\leq 0$, which implies that $w(\theta,y,\tau)\leq 0$ on $(\cC_{n,k}\backslash Q_R)\times\R_{>0}$.
		\end{proof}
		Now we apply this lemma to $\pm e^{\tau L}h(\theta,y) - \bar C e^{-\lambda\tau}(1+|y|^2)^{\lambda'}$, where here $\bar C=\bar C(n, \eps)>1$ is chosen large enough such that $|e^{\tau L}h(\theta,y)|\leq \bar C e^{-\lambda\tau}(1+|R|^2)^{\lambda'}$ for $|y|=R$, according to \eqref{eq:Parabolic_Jacobi_Less_R}. Then we obtain that $|e^{\tau L}h(\theta,y)|\leq \bar C(n,\eps) e^{-\lambda\tau}(1+|y|^2)^{\lambda'}$ when $|y|>R$.
	\end{proof}

	\begin{proof}[Proof of Proposition \ref{Prop_App_Solve Parab Jac equ w spacetime C^0 est}]
		The existence of the solution can be easily seen from the Fourier decomposition. In fact, suppose $h(\cdot,\tau)=\sum_{j\geq 0}a_j(\tau)\phi_j$ is the $L^2$-decomposition, then we can construct
		\begin{equation}
			w(\cdot,\tau)=\sum_{\lambda_j< \lambda}\phi_j e^{-\lambda_j \tau}
			\left(-\int_\tau^\infty a_j(s)e^{\lambda_j s}ds\right)
			+ \sum_{\lambda_j> \lambda}\phi_j e^{-\lambda_j \tau}
			\left(\int_0^\tau a_j(s)e^{\lambda_j s}ds\right) \,.
		\end{equation} 
		We first show that this solution is well-defined. The pointwise bound of $|h(\cdot,\tau)|$ implies that weighted-$L^2$ norm of $h(\cdot,\tau)$ is also bounded by $e^{-\lambda\tau}$, therefore $|a_j(\tau)|\lesssim_{n, \eps} e^{-\lambda\tau}$, and hence \[
		\left|\int_\tau^{\tau'} a_j(s)e^{\lambda_j s}\ ds\right|
		\lesssim_{n, \eps} \left| e^{(\lambda_j - \lambda)\tau}-e^{(\lambda_j - \lambda)\tau'} \right| \,,  \]
		So $w$ is well-defined. In particular,  \[
		\left|\Pi_{<\lambda}w(\cdot,\tau)\right| 
		= \left| \sum_{\lambda_j< \lambda}\phi_j e^{-\lambda_j \tau}
		\left( -\int_\tau^\infty a_j(s) e^{\lambda_j s}\ ds\right) \right|
		\lesssim_{n, \eps} e^{-\lambda\tau} (1+|y|^2)^{\lambda'-\eps}.
		\]
		Here we use the fact that $|\phi_j(\theta,y)|\lesssim_\eps(1+|y|^2)^{\lambda_j +1}\lesssim_\eps(1+|y|^2)^{\lambda+1}\lesssim_\eps (1+|y|^2)^{\lambda'-\eps}$ when $\lambda_j<\lambda$.
		
		Next we estimate $\left|\Pi_{>\lambda}w(\cdot,\tau)\right|$. By Duhamel's principle and the assumption of $\lambda$, \[
		\Pi_{>\lambda}w(\cdot,\tau) = \Pi_{>\lambda+\eps}w(\cdot,\tau)
		= \int_0^\tau e^{(\tau-s)L}\Pi_{>\lambda+\eps}h(\cdot,s)\ ds.
		\]
		Then we use Lemma \ref{lem:growth_parabolic_Jacobi_proj_large} and the fact that $|\Pi_{>\lambda+\eps}h(\cdot,s)|\lesssim_{n, \eps} e^{-\lambda s}(1+|y|^2)^{\lambda'}$ to obtain,
		\begin{align*}
			\left|\Pi_{>\lambda} w(\cdot,\tau)\right|
			\lesssim_{n,\eps} \int_0^{\tau} e^{-(\lambda+\eps)(\tau-s)}\cdot e^{-\lambda s}(1+|y|^2)^{\lambda'}\ ds
			\lesssim_{n, \eps} e^{-\lambda\tau}(1+|y|^2)^{\lambda'}.
		\end{align*}
		Combining the estimates of $\left|\Pi_{>\lambda}w(\cdot,\tau)\right|$ and $\left|\Pi_{<\lambda}w(\cdot,\tau)\right|$ gives the desired bound.
	\end{proof}
	
	\begin{proof}[Proof of Proposition \ref{Prop_App_Parab equ Spacetime C^0 bd from initial data}]
		First of all, because $w$ is an $L^2$-solution, we can write it down by the Fourier decomposition
		\begin{equation}
			w(\cdot,\tau)=\sum_{j\geq 0}a_j e^{-\lambda_j\tau}\phi_j.
		\end{equation}
		Suppose $K$ is the smallest integer such that $a_K\neq 0$. Then $\|w(\cdot,\tau)\|^2_{L^2}=\sum_{j\geq K}|a_j|^2 e^{-2\lambda_j\tau}$. In particular, we have \[
		e^{(\mu-\eps')\tau}\|w(\cdot, \tau)\|_{L^2}
		\geq e^{(\mu-\eps')\tau}|a_K|e^{-\lambda_K\tau}, \qquad 
		e^{(\mu+\eps')\tau}\|w(\cdot, \tau)\|_{L^2}
		\leq e^{(\mu+\eps')\tau}e^{-\lambda_K\tau} \|w(\cdot,0)\|_{L^2}.
		\]
		This implies that $\lambda_K\in[\mu-\eps',\mu+\eps']$, and hence, $\lambda_K=\mu$. This proves that $\mu\in\sigma(\cC_{n,k})$.
		
		Suppose $\psi:= \Pi_{=\mu}(w(\cdot, 0))$. Because $\mu+\eps\leq \lambda\leq \lambda'-1-\eps<\eps^{-1}$, we have $|\psi(\theta,y)|\lesssim_{n, \eps} (1+|y|^2)^{\lambda'}$. This implies that $|w(\theta, y, 0) - \psi(\theta, y)|\lesssim_{n,\eps} (1+|y|^2)^{\lambda'}$. So we can apply Lemma \ref{lem:growth_parabolic_Jacobi_proj_large} to get the desired bound.
	\end{proof}
	
	We also recall some basic properties of parabolic Jacobi fields over $\cC_{n,k}$. The proof can be found in \cite[Appendix A]{SunWangXue1_Passing}.
	
	\begin{Lem} \label{Lem_App_Analysis of Parab Jacob field}
		Let $a+1< b$, $v\in C^2_{loc}(\cC_{n,k}\times (a,b))$ be a non-zero function so that $\|v(\cdot, \tau)\|_{L^2}<+\infty$ for every $\tau\in (a, b)$ and it solves
		\begin{align*}
			(\partial_\tau - L_{n,k})v = 0
		\end{align*}
		on $\cC_{n,k}\times (a,b)$. Define the \textbf{linear decay order} of $v$ at time $\tau$ by \[
		N_{n,k}(\tau; v) := \log\left( \frac{\|v(\cdot, \tau)\|_{L^2}}{\|v(\cdot, \tau + 1)\|_{L^2}}\right).
		\]
		Then we have,
		\begin{enumerate}[label={\normalfont(\roman*)}] 
			\item $N_{n,k}(\tau; v)\geq -1$ and is monotone non-increasing in $\tau\in (a, b-1)$.
			\item If for some $\tau_0\in (a, b)$, $\gamma\in \RR$ and $\sim\in \{\geq, >, =, <, \leq\}$, we have $\|\Pi_{\sim \gamma} (v(\cdot, \tau_0))\|_{L^2} = \|v(\cdot, \tau_0)\|_{L^2}$, where $\Pi_{\sim \gamma}$ is defined in \eqref{Equ_Pre_Proj onto sum of eigenspace}. Then for every $\tau\in (a,b-1)$, \[
			N_{n,k}(\tau; v)\sim \gamma\,.
			\] 
			\item If for some $\tau_1<\tau_2\in (a, b-1)$, $N_{n,k}(\tau_1; v) = N_{n,k}(\tau_2; v) = \gamma$, then \[
			v(X, \tau) = e^{-\gamma\tau}\psi(X)
			\] 
			for some non-zero eigenfunction $\psi\in \rmW_\gamma(\cC_{n,k})$ of $-L_{n,k}$.
		\end{enumerate}
	\end{Lem}

	\section{Graphs over generalized cylinders} \label{Append_Graph over Cylinder}
	Throughout this section, we parametrize $\cC_{n,k}=\SSp^{n-k}\times \RR^k \subset \RR^{n+1}$ by $(\theta, y)$ as before. Note that for every $(\theta, y)\in \cC_{n,k}$, $|\theta|^2=2(n-k)=:\varrho^2$. Let $\hat\theta := \theta/|\theta|$.
	
	Let $\Omega = \SSp^{n-k}\times \Omega_\circ\subset \cC_{n,k}$ be a subdomain, $v\in C^1(\Omega)$ such that $\inf v>-\sqrt{2(n-k)}$.   We use $\nabla_\theta v$ and $\nabla_y v$ to denote the components of $\nabla v$ parallel to $\SSp^{n-k}$ factor and $\RR^k$ factor correspondingly. For later reference, we also denote 
	\begin{align*}
		\hat\nabla_\theta v := \left(1+\frac{v}{|\theta|}\right)^{-1}\cdot\nabla_\theta v\,, & &
		\hat\nabla v := \hat\nabla_\theta v + \nabla_y v \,.
	\end{align*}
	
	In the following, $\bu\in \scU$ is an arrival time function that we constructed in Lemma \ref{Lem_L^2Noncon_LowSphericalMode}.  
	\begin{Lem} \label{Lem_App_Graph over Cylinder}
		Let $\Omega, v$ be specified as above, let \[
		\Sigma = \graph_{\cC_{n,k}}(v) = \{ ( \theta + v(\theta, y)\hat\theta, y): (\theta, y)\in \Omega \} \,.
		\]
		And we parametrized $\Sigma$ by \[
		\Phi_v: \Omega \to \Sigma\,, \quad (\theta, y)\mapsto ( \theta + v(\theta, y)\hat\theta, y) \,.
		\]
		Then $\Sigma$ is a hypersurface in $\RR^{n+1}$, $\Phi_v$ is a diffeomorphism and we have the following for $\tau\geq 0$.
		\begin{enumerate} [label={\normalfont(\roman*)}]
			\item\label{Item_GraphCyliner_odist} For every $(\theta, y)\in \Omega$, \[
			\odist_{\bu}(\tau, \Phi_v(\theta, y)) = \chi(v(\theta, y)-\varphi_\bu(\theta, \tau)),   \]
			\item\label{Item_GraphCyliner_normal} The unit normal field of $\Sigma$ pointing away from $\{\orig\}\times \RR^k$ is \[
			\nu_\Sigma|_{\Phi_v(\theta, y)} = \left.(1+|\hat\nabla v|^2)^{-1/2}\cdot \left( \hat\theta - \hat\nabla v\right) \right|_{(\theta, y)} \,,
			\] 
			\item\label{Item_GraphCyliner_d_u vs |v-varphi_u|_L^2} For every function $f\in C^0(\Sigma)$, \[
			\int_\Sigma f(x)\ dx = \int_\Omega f\circ\Phi_v(X) \cA[v]\ dX\,,
			\]
			where \[
			\cA[v](\theta, y) := \left(1+\frac{v}{|\theta|}\right)^{n-k}\cdot\sqrt{1+|\hat\nabla v|^2} \,.
			\]
			In particular, there exists $\kappa_{\ref{Lem_App_Graph over Cylinder}}(n)\in (0, 1/2)$ such that if $\|\varphi_\bu(\cdot, \tau)\|_{C^1},\,\|v\|_{C^1}\leq \kappa_{\ref{Lem_App_Graph over Cylinder}}(n)$, then \[
			\left|\mbfd_{\bu}(\tau, \Sigma)^{-1}\cdot \|v-\varphi_\bu(\cdot, \tau)\|_{L^2} - 1\right| \lesssim_n \|v\|_{C^1} \,.
			\]
			\item\label{Item_GraphCylinder_Rot} 
			$\exists\,\kappa'_{\ref{Lem_App_Graph over Cylinder}}(n)\in (0, 1/2)$ such that if $(\bx, \by)\in \RR^{n-k+1}\times \RR^k$, $\mbfA\in \mfk g_{n,k}^\perp$, $\lambda>0$, $R>2n$ satisfy, \[
			\delta:=\|v\|_{C^2(Q_R(0, \lambda^{-1}\by))} + |\bx| + R|\lambda - 1| + R|\mbfA| \leq \kappa'_{\ref{Lem_App_Graph over Cylinder}} \,,
			\]
			then in $Q_{R - C_n\delta}$, $e^\mbfA(\lambda\Sigma - (\bx, \by))$ is also a graph over $\cC_{n,k}$, and the graphical function $\bar v$ satisfies, 
			\begin{equation}\label{eq:AppE(v)}
				\begin{split}
					\sup_{(\theta', y')\in \cC_{n,k}\cap Q_{R-C_n\delta}} &\, \left|\bar v(\theta', y') - v(\theta', \lambda^{-1}\by +y') - \varrho(\lambda-1) + \psi_\bx(\theta') - \psi_\mbfA(\theta', y')\right| \\
					& \lesssim_n \delta (|\bx| + R|\lambda-1| + R|\mbfA|)\,. 
				\end{split}
			\end{equation}
			where $\psi_\mbfx, \psi_\mbfA$ are eigenfunctions of $-L_{n,k}$ defined in \eqref{Equ_Pre_Transl-like eigen} and \eqref{Equ_Pre_Rot-like eigen} in the Preliminary.  
		\end{enumerate}
	\end{Lem}
	
	\begin{proof}
		The items \ref{Item_GraphCyliner_odist}, \ref{Item_GraphCyliner_normal} and \ref{Item_GraphCyliner_d_u vs |v-varphi_u|_L^2} are routine. In \ref{Item_GraphCylinder_Rot}, the fact the $e^\mbfA(\lambda\Sigma-(\bx, \by))$ is a graph over $\cC_{n,k}$ in $Q_{R-C_n\delta}$ follows from implicit function theorem. So we only focus on the proof of \eqref{eq:AppE(v)}. 
		
		Given a point $(\theta+v(\theta,y)\hat\theta,y)$ in the graph over $\cC_{n,k}\cap Q_R(0, \lambda^{-1}\by)$, we suppose \[
		X':= \left(\theta'+\bar v(\theta', y')\hat\theta',\, y'\right) = e^{\mbfA}\left(\lambda \left(\theta + v(\theta,y)\hat\theta,\, y \right)-(\bx,\by)\right)
		\] 
		be the point after the transformation. Note that $\mbfA\in \mfk g_{n,k}^\perp$, we may assume $\mbfA=\begin{bmatrix}
			0 & \bl \\ -\bl^\top & 0
		\end{bmatrix}$. Also note that we have an a priori bound $|X'|\lesssim_n R$. Then we can write
		\begin{align} 
			\label{eq:AppE_theta}
			\theta'+\bar v(\theta',y')\hat\theta'
			& = \lambda \left(\theta+v(\theta,y)\hat\theta \right) -\bx + \bl (\lambda y-\by)  +  O(|\mbfA|^2R) \,,  \\
			\label{eq:AppE_y}
			y' & = \lambda y - \by - \bl^\top\left(\lambda \left(\theta+v(\theta,y)\hat\theta\right)-\bx \right)
			+ O(|\mbfA|^2R) \,.
		\end{align}
		Also, by the expression $e^{-\mbfA}(\theta'+\bar v(\theta',y')\hat\theta',\, y')=\lambda(\theta+v(\theta,y)\hat\theta,\, y)-(\bx,\by)$, we can write 
		\begin{align}
			\lambda y-\by = y' + \bl^\top (\theta'+\bar v(\theta',y')\hat\theta') + O(|\mbfA|^2R) \,. \label{eq:AppE_lambda y-by}
		\end{align} 
		
		If we take the inner product of \eqref{eq:AppE_theta} with $\hat\theta'$, by \eqref{Equ_Pre_Transl-like eigen} and \eqref{Equ_Pre_Rot-like eigen} we get \[
		\varrho + \bar v(\theta',y') =
		\lambda \varrho +\lambda v(\theta,y)\hat\theta\cdot\hat\theta' - \psi_\bx(\theta') + \psi_\mbfA(\theta', \lambda y-\by) + O(|\mbfA|^2R) \,,
		\]
		which implies that $|\bar v(\theta', y')|\lesssim_n |\lambda-1|+|v(\theta, y)|+|\bx|+R|\mbfA| \lesssim_n \delta$ and,
		\begin{align}
			\begin{split}
				& \bar v(\theta', y') - v(\theta', \lambda^{-1}\by +y') - \varrho(\lambda-1) + \psi_\bx(\theta') - \psi_\mbfA(\theta', y') \\
				& \quad = \left(\lambda v(\theta,y)\hat\theta\cdot\hat\theta'  -  v(\theta', \lambda^{-1}\by +y') \right)  +  \psi_\mbfA(\theta', \lambda y-\by-y') + O(|\mbfA|^2R) \,. 
			\end{split} \label{Equ_App_GraphCylind_Pf_bar v-v-psi_x+psi_A}
		\end{align}
		
		Now it suffices to estimate the terms on the right-hand side. If we project \eqref{eq:AppE_theta} onto $\SSp^{n-k}$ and take $\kappa'_{\ref{Lem_App_Graph over Cylinder}}(n)\ll 1$, we get \[
		|\hat\theta'-\hat\theta| \lesssim_n \left|-\bx + \bl (\lambda y-\by)  +  O(|\mbfA|^2R) \right| \lesssim_n |\bx| + |\mbfA|R
		\]
		And hence \[
		|1-\hat\theta\cdot \hat\theta'| = \frac12|\hat\theta'-\hat\theta|^2 \lesssim_n |\bx| + |\mbfA|R \,.
		\]
		Combining these with \eqref{eq:AppE_lambda y-by} yields, 
		\begin{align*}
			\left|\lambda v(\theta,y)\hat\theta\cdot\hat\theta'  -  v(\theta', \lambda^{-1}\by +y') \right| 
			& \leq (|\lambda-1|+|\bx|+|\mbfA|R)|v(\theta, y)| + |v(\theta, y)-v(\theta', \lambda^{-1}\by +y')| \\
			& \lesssim_n (R|\lambda-1|+|\bx|+R|\mbfA|)\cdot \|v\|_{C^1(Q_R(0, \lambda^{-1}\by))} \,;
		\end{align*}
		and \[
		|\psi_\mbfA(\theta', \lambda y-\by-y')| \lesssim_n |\mbfA||\lambda y-\by-y'| \lesssim_n |\mbfA|^2 \,.
		\]
		Together with \eqref{Equ_App_GraphCylind_Pf_bar v-v-psi_x+psi_A}, these proves \eqref{eq:AppE(v)}.    
	\end{proof}
	
	The item \ref{Item_GraphCyliner_d_u vs |v-varphi_u|_L^2} of Lemma \ref{Lem_App_Graph over Cylinder} shows that $\mbfd_\bu(\tau, \Sigma)$ is comparable to $\|v-\varphi_\bu(\cdot, \tau)\|_{L^2}$, if $\Sigma$ is a graph of function $v$ over $\cC_{n,k}$. Therefore, by the classical Schauder theory, we can use $\mbfd_\bu(\Sigma)$ to bound the $C^2$-norm of $v-\varphi_\bu$. We first recall the equation of the graph function of a rescaled mean curvature flow over $\cC_{n,k}$. Suppose $v(\cdot,\tau)$ is the graph function of a rescaled mean curvature flow over $\cC_{n,k}$, then it satisfies the equation
	\begin{equation}\label{eq_RMCF_graph_over_cyl}
		\pr_\tau v = L_{n,k}v + \cQ(v,\nabla v,\nabla^2 v) =: L_{n,k}v + \cQ[v],
	\end{equation}
	where $\cQ$ is a smooth function in $(z, \xi, \eta)\in \RR\times T\cC_{n,k}\times T^{\otimes 2}\cC_{n,k}$ with \[
	\cQ(0, 0, 0) = \cQ_z(0, 0, 0) = \cQ_\xi(0,0,0) = \cQ_\eta(0,0,0) = 0\,.
	\]
	The precise calculations in various formulations can be found in Appendix A of \cite{CM15_Lojasiewicz}, Appendix A of \cite{SunXue2022_generic_cylindrical}, among others.

	\begin{Lem} \label{Lem_L^2 clos => C^2 close}
		For every $\eps\in (0, 1/2)$ and $R>2n$, there exists $\delta_{\ref{Lem_L^2 clos => C^2 close}}(\eps, n, R)\in (0, \eps)$ such that the following hold.  If $T\geq 0$, $\delta\in (0, \delta_{\ref{Lem_L^2 clos => C^2 close}}]$, $\bu\in \scU$ satisfies $\|\varphi_\bu(\cdot,\tau)\|\leq\delta$ for $\tau\in[T,T+1]$ (see the notations in Section \ref{Subsec_Mod Low SphericalMode}), and $\cM$ is a RMCF in $\RR^{n+1}$ over $[T, T+1]$ such that \[
		\lambda[\cM]\leq \eps^{-1}\,, \qquad
		\text{$\cM$ is $\delta$-$L^2$ close to $\cC_{n,k}$} \,.
		\]
		Then $\cM(T+1)$ is a $C^2$ graph over $\cC_{n,k}$ in $Q_{R}$ with graphical function $v(\cdot, T+1)$ satisfying \[
		\|(v-\varphi_\bu)(\cdot, T+1)\|_{C^2(Q_{R}\cap \cC_{n,k})} \lesssim_{n, \eps, R}\mbfd_\bu(T, \cM) \,.
		\]
	\end{Lem}
	
	\begin{proof}
		Let $\eps_1\in(0, \delta_0(n))$ to be determined. First by Brakke-White's $\eps$-regularity and a compactness argument, there exists $\delta'(\eps, n, R, \eps_1)\ll 1$ such that if $\cM, \bu$ are specified as in the Lemma with $\delta\leq \delta'(\eps, n, R, \eps_1)$, then for all $\tau\in [1/2,1]$, $\cM(T+\tau)$ is a $C^2$ graph over $\cC_{n,k}$ in $Q_{3R}$ with graphical function $v(\cdot, T+\tau)$ satisfies $\|v(\cdot, T+\tau)\|_{C^4(Q_{3R}\cap \cC_{n,k})}\leq \eps_1$, and $v$ solves the RMCF equation \eqref{eq_RMCF_graph_over_cyl}. 
		
		Now we prove the $C^2$ estimate of $v-\varphi_{\bu}$. Let $w(\cdot, \tau):=(v-\varphi_\bu)(\cdot, T+\tau)$, then inside $(Q_{3R}\cap \cC_{n,k})\times[1/2,1]$, $w$ satisfies the equation \[
		\pr_\tau w=Lw+\cQ[v]-\cQ[\varphi_\bu] 
		= \pr_\tau w=Lw+aw+b_j\pr_jw+c_{ij} \pr_{ij}w,
		\]
		where $a,b_j,c_{ij}$ are terms of the partial derivatives of $\cQ$ with variables involving $v$, $\nabla v$, $\nabla^2 v$, $\varphi_\bu$, $\nabla\varphi_\bu$, $\nabla^2 \varphi_\bu$. Therefore, if we choose $\eps_1=\eps_1(n, R)$ to be sufficiently small and $\delta_{\ref{Lem_L^2 clos => C^2 close}}=\delta'(\eps, n, R, \eps_1)$, we can make $|a|$, $|b_j|$, $|c_{ij}|$ sufficiently small, such that $\pr_\tau w=Lw+aw+b_j\pr_jw+c_{ij} \pr_{ij}w$ is a strictly parabolic equation with uniformly bounded coefficients in $(Q_{3R}\cap \cC_{n,k})\times[1/2,1]$. Then the standard interior estimate of parabolic equation, e.g. \cite[Lemma 7.38]{Lieberman96_ParabolicPDE}, shows that \[
		\|w\|_{C^0((Q_{2R}\cap \cC_{n,k})\times[3/4,1])}
		\lesssim_{n, \eps,R} \sup_{\tau \in[1/2,1]}\|w\|_{L^2(Q_{3R}\cap \cC_{n,k})}
		\lesssim_{n, \eps,R} \mbfd_\bu(T, \cM)\,,
		\]
		where the last inequality follows from Lemma \ref{Lem_App_Graph over Cylinder} \ref{Item_GraphCyliner_d_u vs |v-varphi_u|_L^2} and Corollary \ref{Lem_L^2 Noncon_for RMCF}.
		Finally, the interior Schauder estimate shows that \[
		\|w\|_{C^{2,\alpha}((Q_{R}\cap \cC_{n,k})\times[7/8,1])} \lesssim_{n, \eps,R} \|w\|_{C^0((Q_{2R}\cap \cC_{n,k})\times[3/4,1])}\lesssim_{n, \eps,R} \mbfd_\bu(T, \cM).
		\]
	\end{proof}

	\bibliographystyle{alpha}
	\bibliography{GMT}
\end{document}